\documentclass[10pt,reqno]{amsart}
\usepackage{amsmath,amssymb,latexsym,esint,cite,mathrsfs}
\usepackage{verbatim,cite,relsize,color,soul}
\usepackage[latin1]{inputenc}
\usepackage{microtype}
\usepackage{color,enumitem,graphicx}
\usepackage[colorlinks=true,urlcolor=blue, citecolor=red,linkcolor=blue,
linktocpage,pdfpagelabels, bookmarksnumbered,bookmarksopen]{hyperref}
\usepackage[hyperpageref]{backref}
\usepackage[english]{babel}
\usepackage{tikz}
\usepackage[font=small,skip=5pt]{caption}
\usepackage{amsrefs}

\usepackage{geometry}
\numberwithin{equation}{section}
\newtheorem{theorem}{Theorem}[section]
\newtheorem{lemma}[theorem]{Lemma}
\newtheorem{proposition}[theorem]{Proposition}

\newtheoremstyle{remarkstyle}
{}{}{}{}{\bfseries}{.}{ }{\thmname{#1}\thmnumber{ #2}\thmnote{ (#3)}}
\theoremstyle{remarkstyle}
\newtheorem{remark}{Remark}[section]

\newcommand{\R}{\mathbb R}
\newcommand{\C}{\mathbb C}
\newcommand{\Ac}{\mathcal A}

\newcommand{\vareps}{\varepsilon}

\DeclareMathOperator*{\gamc}{\gamma_c}
\DeclareMathOperator*{\sigc}{\sigma_c}
\DeclareMathOperator*{\deltc}{\delta_c}
\DeclareMathOperator*{\loc}{loc}
\DeclareMathOperator*{\ct}{c}

\DeclareMathOperator*{\opt}{opt}

\DeclareMathOperator*{\inls}{INLS}
\DeclareMathOperator*{\rea}{Re}
\DeclareMathOperator*{\ima}{Im}

\newcommand{\scal}[1]{\left\langle #1 \right\rangle}

\title[Dynamics for Inhomogeneous NLS]{Long time dynamics of non-radial solutions to inhomogeneous nonlinear Schr\"odinger equations} 
\author[V. D. Dinh]{Van Duong Dinh}
\address[V. D. Dinh]{Laboratoire Paul Painlev\'e UMR 8524, Universit\'e de Lille CNRS, 59655 Villeneuve d'Ascq Cedex, France
	and 
	Department of Mathematics, HCMC University of Education, 280 An Duong Vuong, Ho Chi Minh, Vietnam}
\email{contact@duongdinh.com}

\author[S. Keraani]{Sahbi Keraani}
\address[S. Keraani]{Laboratoire Paul Painlev\'e UMR 8524, Universit\'e de Lille CNRS, 59655 Villeneuve d'Ascq Cedex, France}
\email{sahbi.keraani@univ-lille.fr}


\begin{document}
	
	\begin{abstract}
		We study long time dynamics of non-radial solutions to the focusing inhomogeneous nonlinear Schr\"odinger equation. By using the concentration/compactness and rigidity method, we establish a scattering criterion for non-radial solutions to the equation. We also prove a non-radial blow-up criterion for the equation whose proof makes use of localized virial estimates. As a byproduct of these criteria, we study long time dynamics of non-radial solutions to the equation with data lying below, at, and above the ground state threshold. In addition, we provide a new argument showing the existence of finite time blow-up solution to the equation with cylindrically symmetric data. The ideas developed in this paper are robust and can be applicable to other types of nonlinear Schr\"odinger equations.
	\end{abstract}
	
	\maketitle
	
	\section{Introduction} 
	\label{sec:intro}
	\setcounter{equation}{0}
	The nonlinear Schr\"odinger equation (NLS) is one of the most important equations in nonlinear optics. It models the propagation of intense laser beams in a homogeneous bulk medium with a Kerr nonlinearity. It is well-known that NLS governed the beam propagation cannot support stable high-power propagation in a homogeneous bulk media. At the end of the last century, it was suggested that stable high-power propagation can be achieved in plasma by sending a preliminary laser beam that creates a channel with a reduced electron density, and thus reduces the nonlinear inside the channel (see e.g., \cite{Gill, LT}). Under these conditions, the beam propagation can be modeled by the inhomogeneous nonlinear Schr\"odinger equation of the form
	\begin{align} \label{eq:inls-K}
	i \partial_t  u + \Delta u + K(x) |u|^\alpha u =0, \quad (t,x) \in \R \times \R^N,
	\end{align}
	where $u$ is the electric field in laser and optics, $\alpha>0$ is the power of nonlinear interaction, and the potential $K(x)$ is proportional to the electron density. By means of variational approximation and direct simulations, Towers and Malomed \cite{TM} observed that for a certain type of nonlinear medium, \eqref{eq:inls-K} gives rise to completely stable beams. 
	
	The equation \eqref{eq:inls-K} has been attracted a lot of interest from the mathematical community. When the potential $K(x)$ is constant, \eqref{eq:inls-K} is the usual nonlinear Schr\"odinger equation which has been studied extensively in the past decades (see e.g., the monographs \cite{Cazenave, SS, Tao}). 
	
	In the case of non-constant bounded potential $K(x)$, Merle \cite{Merle} proved the existence and nonexistence of minimal blow-up solutions to \eqref{eq:inls-K} with $\alpha=\frac{4}{N}$ and $K_1\leq K(x) \leq K_2$, where $K_1$ and $K_2$ are positive constants. Based on the work of Merle, Rapha\"el and Szeftel \cite{RS} established sufficient conditions for the existence, uniqueness, and charaterization of minimial blow-up solutions to the equation. Fibich and Wang \cite{FW}, and Liu and Wang \cite{LWW} investigated the stability and instability of solitary waves for \eqref{eq:inls-K} with $\alpha \geq \frac{4}{N}$ and $K(x) = K(\epsilon x)$, where $\epsilon>0$ is a small parameter and $K \in C^4(\R^N) \cap L^\infty(\R^N)$. 
	
	When the potential $K(x)$ is unbounded, the problem becomes more subtle. The case $K(x)= |x|^b, b>0$ was studied in several works, for instance, Chen and Guo \cite{CG}, and Chen \cite{Chen} established sharp criteria for the global existence and blow-up, and Zhu \cite{Zhu} studied the existence and dynamical properties of blow-up solutions. When $K(x)$ behaves like $|x|^{-b}$ with $b>0$, De Bouard and Fukuizumi \cite{BF} studied the stability of standing waves for \eqref{eq:inls-K} with $\alpha<\frac{4-2b}{N}$. Fukuizumi and Ohta \cite{FO} established the instability of standing waves for \eqref{eq:inls-K} with $\alpha>\frac{4-2b}{N}$ (see also \cite{JC, GS} and references therein for other studies related to standing waves for this type of equation). 
	
	In this paper, we consider the Cauchy problem for a class of focusing inhomogeneous nonlinear Schr\"odinger equations (INLS)
	\begin{equation} \label{eq:inls}
	\left\{ 
	\begin{array}{rcl}
	i\partial_t u +\Delta u &=& - |x|^{-b}|u|^\alpha u, \quad (t,x) \in \R \times \R^N, \\
	\left.u\right|_{t=0}&=& u_0 \in H^1,
	\end{array}
	\right.
	\end{equation}
	where $u: \mathbb{R} \times \mathbb{R}^N \rightarrow \mathbb{C}$, $u_0: \mathbb{R}^N \rightarrow \mathbb{C}$, $N\geq 1$, $0<b<\min\{2,N\}$, and $\frac{4-2b}{N}<\alpha <\alpha(N)$ with
	\begin{align}\label{eq:alpha-N}
	\alpha(N) := \left\{
	\begin{array}{ccl}
	\frac{4-2b}{N-2} &\text{if}& N\geq 3, \\
	\infty &\text{if}& N=1,2.	
	\end{array}
	\right.
	\end{align} 
	This equation plays an important role as a limiting equation in the analysis of \eqref{eq:inls-K} with $K(x) \sim |x|^{-b}$ as $|x|\rightarrow \infty$ (see e.g., \cite{GS, Genoud-2d}).
	
	The local well-posedness for \eqref{eq:inls} was studied by Geneoud and Stuart \cite[Appendix]{GS}. More precisely, they proved that \eqref{eq:inls} is locally well-posed in $H^1$ for $N\geq 1$, $0<b<\min \{2, N\}$, and $0<\alpha<\alpha(N)$. The proof of this result is based on the energy method developed by Cazenave \cite{Cazenave}, which does not use Strichartz estimates. See also \cite{Guzman, Dinh-weight} for other proofs based on Strichartz estimates and the contraction mapping argument. Note that the local well-posedness in \cite{Guzman, Dinh-weight} is more restrictive than the one in \cite{GS}. However, it provides more information on the local solutions, for instance, local solutions belong to $L^q_{\loc}((-T_*,T^*),W^{1,r}(\R^N))$ for any Schr\"odinger admissible pair $(q,r)$ (see Section \ref{sec:scat} for the definition of $L^2$ admissibility), where $(-T_*,T^*)$ is the maximal time interval of existence. Note that the latter property plays an important role in the scattering theory.
	
	It is well-known that solutions to \eqref{eq:inls} satisfy the conservation laws of mass and energy
	\begin{align*}
	M(u(t)) &= \|u(t)\|^2_{L^2} = M(u_0), \tag{Mass} \\
	E(u(t)) &= \frac{1}{2} \|\nabla u(t)\|^2_{L^2} - \frac{1}{\alpha+2} \int |x|^{-b} |u(t,x)|^{\alpha+2} dx = E(u_0). \tag{Energy}
	\end{align*}  
	The equation \eqref{eq:inls} also has the following scaling invariance
	\begin{align} \label{eq:scaling}
	u_\lambda (t,x) := \lambda^{\frac{2-b}{\alpha}} u(\lambda^2 t, \lambda x), \quad \lambda>0.
	\end{align}
	A direct calculation gives
	\[
	\|u_\lambda(0)\|_{\dot{H}^\gamma} = \lambda^{\gamma+\frac{2-b}{\alpha} -\frac{N}{2}} \|u_0\|_{\dot{H}^\gamma}
	\]
	which shows that \eqref{eq:scaling} leaves the $\dot{H}^{\gamc}$-norm of initial data invariant, where
	\begin{align}  \label{eq:gamc}
	\gamc:= \frac{N}{2}-\frac{2-b}{\alpha}.
	\end{align}
	The condition $\frac{4-2b}{N}<\alpha<\alpha(N)$ is equivalent to $0<\gamc<1$ which corresponds to the mass-supercritical and energy-subcritical range (intercritical range, for short). For later uses, it is convenient to introduce the following exponent
	\begin{align} \label{eq:sigc}
	\sigc:= \frac{1-\gamc}{\gamc} = \frac{4-2b-(N-2)\alpha}{N\alpha-4+2b}.
	\end{align}
	
	The main purpose of the present paper is to study long time dynamics (global existence, energy scattering, and finite time blow-up) of non-radial solutions to \eqref{eq:inls}. Before stating our contributions, let us recall known results related to dynamics of \eqref{eq:inls} in the intercritical range. 
	
	In \cite{Farah}, Farah showed the global existence for \eqref{eq:inls} with $N\geq 1$ and $0<b<\min \{2,N\}$ by assuming $u_0 \in H^1$ and
	\begin{align} 
	E(u_0) [M(u_0)]^{\sigc} &< E(Q) [M(Q)]^{\sigc}, \label{eq:ener-below} \\
	\|\nabla u_0\|_{L^2} \|u_0\|^{\sigc}_{L^2} &< \|\nabla Q\|_{L^2} \|Q\|^{\sigc}_{L^2}, \label{eq:grad-glob-below}
	\end{align}
	where $Q$ is the unique postive radial solution to the elliptic equation
	\begin{align} \label{eq:ground}
	-\Delta Q+ Q - |x|^{-b} |Q|^\alpha Q =0.
	\end{align}
	He also proved the finite time blow-up for \eqref{eq:inls} with $u_0 \in \Sigma := H^1 \cap L^2(|x|^2 dx)$ satisfying \eqref{eq:ener-below} and 
	\begin{align} \label{eq:grad-blow-below}
	\|\nabla u_0\|_{L^2} \|u_0\|^{\sigc}_{L^2}>\|\nabla Q\|_{L^2} \|Q\|^{\sigc}_{L^2}.
	\end{align}
	The latter result was extended to radial data by the first author in \cite{Dinh-NA}. Note that the uniqueness of positive radial solution to \eqref{eq:ground} was established by Yanagida \cite{Yanagida} for $N\geq 3$, Genoud \cite{Genoud-2d} for $N=2$, and Toland \cite{Toland} for $N=1$. 
	
	The energy scattering (or asymptotic behavior) for \eqref{eq:inls} was first established by Farah and Guzm\'an \cite{FG-JDE} with $0<b<\frac{1}{2}, \alpha=2, N=3$, and radial data. The proof of this result is based on the concentration/compactness and rigidity argument introduced by Kenig and Merle \cite{KM}. This scattering result was later extended to dimensions $N\geq 2$ in \cite{FG-BBMS} by using the same concentration/compactness and rigidity method. 
	
	Later, Campos \cite{Campos} made use of a new idea of Dodson and Murphy \cite{DM-rad} to give an alternative simple proof for the radial scattering results of Farah and Guzm\'an. He also extends the validity of $b$ in dimensions $N\geq 3$. Note that the idea of Dodson and Murphy is a combination of a scattering criterion of Tao \cite{Tao-DPDE}, localized virial estimates, and radial Sobolev embedding.
	
	Afterwards, Xu and Zhao \cite{XZ}, and the first author \cite{Dinh-2D} have simultaneously showed the energy scattering for \eqref{eq:inls} with $0<b<1, N=2$, and radial data. The proof relies on a new approach of Arora, Dodson, and Murhpy \cite{ADM}, which is a refined version of the one in \cite{DM-rad}. 
	
	In \cite{CC}, Campos and Cardoso studied long time dynamics such as global existence, energy scattering, and finite time blow-up of $H^1$-solutions to \eqref{eq:inls} with data in $\Sigma$ lying above the ground state threshold. 
	
	Recently, Miao, Murphy, and Zheng \cite{MMZ} showed a new nonlinear profile for non-radial solutions related to \eqref{eq:inls}. In particular, they constructed nonlinear profiles with data living far away from the origin. This allows them to show the energy scattering of non-radial solution to \eqref{eq:inls} with $0<b<\frac{1}{2}$, $\alpha=2$, and $N=3$. This result was extended to any dimensions $N\geq 2$ and $0<b<\min \left\{2, \frac{N}{2}\right\}$ by Cardoso, Farah, Guzm\'an, and Murphy \cite{CFGM}. 
	
	We also mention the works \cite{Dinh-JEE, Dinh-2D} for the energy scattering for the defocusing problem INLS and \cite{CHL} for the energy scattering for the focusing energy-critical INLS.
	
	Motivated by the aforementioned works, we study the global existence, energy scattering, and finite time blow-up of non-radial solutions to \eqref{eq:inls}. To this end, let us start with the following scattering criterion for \eqref{eq:inls}. 
	
	\begin{theorem} [Scattering criterion] \label{theo-scat-crite}
		Let $N\geq 1$, $0<b<\min \{2,N\}$, and $\frac{4-2b}{N} < \alpha < \alpha(N)$. Let $u$ be a solution to \eqref{eq:inls} defined on the maximal forward time interval of existence $[0,T^*)$. Assume that
		\begin{align} \label{scat-crite}
		\sup_{t\in [0,T^*)} P(u(t))[M(u(t))]^{\sigc} < P(Q) [M(Q)]^{\sigc},
		\end{align}
		where
		\begin{align} \label{eq:P}
		P(f):= \int |x|^{-b}|f(x)|^{\alpha+2} dx.
		\end{align}
		Then $T^*=\infty$. Moreover, if we assume in addition that $N\geq 2$ and $0<b<\min \left\{2,\frac{N}{2}\right\}$, then the solution scatters in $H^1$ forward in time, i.e., there exists $u_+\in H^1$ such that
		\begin{align} \label{defi-scat}
		\lim_{t\rightarrow \infty} \|u(t) - e^{it\Delta} u_+\|_{H^1} =0.
		\end{align}
		A similar statement holds for negative times.
	\end{theorem}
	
	We note that a scattering condition similar to \eqref{scat-crite} was first introduced by Duyckaerts and Roudenko in \cite[Theorem 3.7]{DR-CMP}, where it was used to show the scattering beyond the ground state threshold for the focusing Schr\"odinger equation. The condition \eqref{scat-crite} was inspired by a recent work of Gao and Wang \cite{GW} (see also \cite{Dinh-DCDS}). 
	
	The proof of Theorem \ref{theo-scat-crite} is based on the concentration/compactness and rigidity method. The main difficulty comes from the fact that the potential energy $P(u(t))$ is not conserved along the time evolution of \eqref{eq:inls}. To overcome the difficulty, we establish a Pythagorean expansion along bounded nonlinear flows. Since we are interested in non-radial solutions, we need to construct nonlinear profiles associated with the linear ones living far away from the origin. The latter was recently showed by Miao, Murphy, and Zheng \cite{MMZ} in three dimensions (see also \cite{CFGM} for dimensions $N\geq 2$). This type of nonlinear profiles is constructed by observing that in the regime $|x| \rightarrow \infty$, the nonlinearity becomes weak, and solutions to \eqref{eq:inls} can be approximated by solutions to the underlying linear Schr\"odinger equation. Thanks to an improved nonlinear estimate (see Lemma \ref{lem-non-est}), we give a refined result with a simple proof of these results (see Lemma \ref{lem-non-prof-dive-spac}). For more details, we refer to Section \ref{sec:scat}.
	
	Our next result is the following blow-up criterion for \eqref{eq:inls}. 
	
	\begin{theorem} [Blow-up criterion] \label{theo-blow-crite}
		Let $N\geq 1$, $0<b<\min \left\{2, N\right\}$, and $\frac{4-2b}{N} < \alpha < \alpha(N)$. Let $u$ be a solution to \eqref{eq:inls} defined on the maximal forward time interval of existence $[0,T^*)$. Assume that
		\begin{align} \label{blow-crite}
		\sup_{t\in [0,T^*)} G(u(t)) \leq -\delta
		\end{align}
		for some $\delta>0$, where
		\begin{align} \label{eq:G}
		G(f):= \|\nabla f\|^2_{L^2} - \frac{N\alpha+2b}{2(\alpha+2)} P(f).
		\end{align}
		Then either $T^*<\infty$, or $T^*=\infty$ and there exists a time sequence $t_n \rightarrow \infty$ such that $\|\nabla u(t_n)\|_{L^2} \rightarrow \infty$ as $n \rightarrow \infty$. Moreover, if we assume in addition that $u$ has finite variance, i.e., $|x|u(t) \in L^2(|x|^2 dx)$ for all $t\in [0,T^*)$, then $T^*<\infty$. A similar statement holds for negative times.
	\end{theorem}
	The proof of this blow-up result is based on a contradiction argument using localized virial estimates for general (non-radial and infinite variance) solutions to \eqref{eq:inls} (see Lemma \ref{prop-viri-est}). We also take the advantage of the decay of the nonlinear term outside a large ball. It is conjectured that if a general (not finite variance or radially symmetric) solution to \eqref{eq:inls} satisfy \eqref{blow-crite}, then it blows up in finite time. However, there is no affirmative answer for this conjecture up to date even for the classical nonlinear Schr\"odinger equation.
	
	A first application of Theorems \ref{theo-scat-crite} and \ref{theo-blow-crite} is the following long time dynamics below the ground state threshold.
	\begin{theorem} [Dynamics below the ground state threshold] \label{theo-dyna-below}
		Let $N\geq 1$, $0<b<\min \left\{2, N\right\}$, and $\frac{4-2b}{N} < \alpha < \alpha(N)$. Let $u_0\in H^1$ satisfy \eqref{eq:ener-below}.
		
		\noindent (1) If $u_0$ satisfies \eqref{eq:grad-glob-below}, then the corresponding solution to \eqref{eq:inls} satisfies
		\begin{align} \label{eq:est-solu-glob-below}
		\sup_{t\in (-T_*,T^*)} P(u(t)) [M(u(t))]^{\sigc} < P(Q) [M(Q)]^{\sigc}.
		\end{align}
		In particular, the solution exists globally in time. Moreover, if we assume in addition that $N\geq 2$ and $0<b<\min \left\{2,\frac{N}{2}\right\}$, then the corresponding solution scatters in $H^1$ in both directions.
		
		\noindent (2) If $u_0$ satisfies \eqref{eq:grad-blow-below}, then the corresponding solution to \eqref{eq:inls} satisfies 
		\begin{align}\label{eq:est-solu-blow-below}
		\sup_{t\in (-T_*,T^*)} G(u(t)) \leq -\delta
		\end{align}
		for some $\delta>0$. In particular, the solution either blows up in finite time, or there exists a time sequence $(t_n)_{n\geq 1}$ satisfying $|t_n| \rightarrow \infty$ such that $\|\nabla u(t_n)\|_{L^2} \rightarrow \infty$ as $n\rightarrow \infty$. Moreover, if we assume in addition that 
		\begin{itemize}[leftmargin=5mm]
			\item $u_0$ has finite variance, 
			\item or $N\geq 2$, $\alpha \leq 4$, and $u_0$ is radially symmetric,
			\item or $N\geq 3$, $\alpha \leq 2$, and $u_0 \in \Sigma_{N}$, where
			\begin{align} \label{eq:Sigma-N}
			\Sigma_N:= \left\{f \in H^1 \ : \ f(y,x_N) = f(|y|,x_N), ~ x_N f \in L^2\right\}
			\end{align}
			with $x=(y,x_N)$, $y=(x_1, \cdots, x_{N-1})\in \R^{N-1}$, and $x_N \in \R$,
		\end{itemize}
		then the corresponding solution blows up in finite time, i.e., $T_*,T^*<\infty$.
	\end{theorem}
	
	For the scattering part, Theorem \ref{theo-dyna-below} provides an alternative proof of a recent result of Cardoso, Farah, Guzm\'an, and Murphy \cite{CFGM}. For the blow-up part, Theorem \ref{theo-dyna-below} extends earlier results of \cite{Farah} (for finite variance data) and the first author \cite{Dinh-NA} (for radial data) to the case of cylindrically symmetric data. Note that the first work addressed the finite time blow-up for NLS with cylindrically symmetric data is due to Martel \cite{Martel}, where the blow-up was shown for data with negative energy. Recently, Bellazzini and Forcella \cite{BF-arXiv} extended Martel's result to the case of focusing cubic NLS for data with non-negative energy data lying below the ground state threshold. Our result not only extends the ones of \cite{Martel, BF-arXiv} to the focusing inhomogeneous NLS but also provides an alternative simple proof for these results. In particular, our choice of cutoff function is simpler than that in \cite{Martel, BF-arXiv}. Our argument is robust and can be applied to show the existence of finite time blow-up solutions with cylindrically symmetric data for other Schr\"odinger-type equations.  
	
	Another application of Thereorems \eqref{theo-scat-crite} and \eqref{theo-blow-crite} is the following long time dyanmics at the ground state threshold.
	
	\begin{theorem} [Dynamics at the ground state] \label{theo-dyna-at}
		Let $N\geq 1$, $0<b<\min \left\{2, N\right\}$, and $\frac{4-2b}{N} < \alpha < \alpha(N)$. Let $u_0\in H^1$ be such that
		\begin{align} \label{eq:ener-at}
		E(u_0) [M(u_0)]^{\sigc} = E(Q)[M(Q)]^{\sigc}.
		\end{align}
		
		\noindent (1) If 
		\begin{align} \label{eq:grad-at-1}
		\|\nabla u_0\|_{L^2}\|u_0\|_{L^2}^{\sigc} < \|\nabla Q\|_{L^2}\|Q\|_{L^2}^{\sigc},
		\end{align}
		then the corresponding solution to \eqref{eq:inls} exists globally in time. Moreover, the solution either satisfies
		\begin{align} \label{eq:est-solu-at}
		\sup_{t\in \R} P(u(t)) [M(u(t))]^{\sigc} < P(Q)[M(Q)]^{\sigc}
		\end{align}
		or there exists a time sequence $(t_n)_{n\geq 1}$ satisfying $|t_n| \rightarrow \infty$ such that 
		\begin{align} \label{eq:conver-u-tn}
		u(t_n) \rightarrow e^{i\theta} Q \quad \text{ strongly in } H^1
		\end{align}
		for some $\theta \in \R$ as $n\rightarrow \infty$. In particular, if we we assume in addition that $N\geq 2$ and $0<b<\min \left\{2, \frac{N}{2}\right\}$, then the solution either scatters in $H^1$ forward in time, or there exist a time sequence $t_n \rightarrow \infty$ and a sequence $(x_n)_{n\geq 1} \subset \R^N$ such that \eqref{eq:conver-u-tn} holds.
		
		\noindent (2) If 
		\begin{align} \label{eq:grad-at-2}
		\|\nabla u_0\|_{L^2}\|u_0\|_{L^2}^{\sigc} = \|\nabla Q\|_{L^2}\|Q\|_{L^2}^{\sigc},
		\end{align}
		then $u(t,x)=e^{it} e^{i\theta} Q(x)$ for some $\theta \in \R$.
		
		\noindent (3) If
		\begin{align} \label{eq:grad-at-3}
		\|\nabla u_0\|_{L^2}\|u_0\|_{L^2}^{\sigc} > \|\nabla Q\|_{L^2}\|Q\|_{L^2}^{\sigc},
		\end{align}
		then the corresponding solution to \eqref{eq:inls} 
		\begin{itemize} [leftmargin=5mm]
			\item[i.] either blows up forward in time, i.e., $T^*<\infty$, 
			\item[ii.] or there exists a time sequence $t_n \rightarrow \infty$ such that $\|\nabla u(t_n)\|_{L^2} \rightarrow \infty$ as $n\rightarrow \infty$,
			\item[iii.] or there exists a time sequence $t_n\rightarrow \infty$ such that \eqref{eq:conver-u-tn} holds.
		\end{itemize}
		Moreover, if we assume in addition that 
		\begin{itemize} [leftmargin=5mm]
			\item $u_0$ has finite variance, 
			\item or $N\geq 2$, $\alpha\leq 4$, and $u_0$ is radially symmetric,
			\item or $N\geq 3$, $\alpha \leq 2$, and $u_0 \in \Sigma_N$,
		\end{itemize}
		then the possibility in Item ii. can be excluded.
	\end{theorem}
	
	To our knowledge, Theorem \ref{theo-dyna-at} is the first result addressing long time dynamics of solutions to \eqref{eq:inls} with data lying at the ground state threshold. For the classical NLS, dynamics at the ground state threshold was first studied by Duyckaerts and Roudenko \cite{DR} for the 3D focusing cubic NLS. The proof in \cite{DR} relies on delicate spectral estimates which make it difficult to extend to higher dimensions. Recently, the first author in \cite{Dinh-DCDS} gave a simple approach to study the dynamics at the threshold for the focusing NLS in any dimensions. Our result is an extension of the one in \cite{Dinh-DCDS} to the focusing inhomogeneous NLS. The proof of Theorem \ref{theo-dyna-at} relies on the scattering and blow-up criteria given in Theorems \ref{theo-scat-crite}, \ref{theo-blow-crite}, and the compactness property of optimizing sequence for the Gagliardo-Nirenberg inequality \eqref{eq:GN-ineq} (see Lemma \ref{lem-comp-mini-GN}). We refer the reader to Section \ref{sec:dynamic} for more details.
	
	Finally, we study long time dynamics above the ground state threshold. Before stating our result, we introduce the virial quantity
	\begin{align} \label{eq:V}
	V(t):= \int |x|^2 |u(t,x)|^2 dx.
	\end{align}
	If $V(0)<\infty$, then $V(t)<\infty$ for all $t$ in the existence time. Moreover, the following identities hold
	\begin{align} \label{eq:prop-V}
	\begin{aligned}
	V'(t) &= 4 \ima \int \overline{u}(t,x) x \cdot \nabla u(t,x) dx, \\
	V''(t) &= 8 \|\nabla u(t)\|^2_{L^2} - \frac{4(N\alpha+2b)}{\alpha+2} P(u(t)).
	\end{aligned}
	\end{align}
	
	\begin{theorem}[Dynamics above the ground state] \label{theo-dyna-above}
		Let $N\geq 1$, $0<b<\min \left\{2, N\right\}$, and $\frac{4-2b}{N} < \alpha < \alpha(N)$. Let $u_0\in \Sigma$ satisfy 
		\begin{align} 
		E(u_0) [M(u_0)]^{\sigc} &\geq E(Q) [M(Q)]^{\sigc}, \label{eq:ener-above-1} \\
		\frac{E(u_0)[M(u_0)]^{\sigc}}{E(Q)[M(Q)]^{\sigc}} &\left(1-\frac{(V'(0))^2}{32 E(u_0) V(0)}\right) \leq 1. \label{eq:ener-above-2}
		\end{align} 
		
		\noindent (1) If 
		\begin{align} 
		P(u_0) [M(u_0)]^{\sigc} &< P(Q)[M(Q)]^{\sigc}, \label{eq:gwp-above-1} \\
		V'(0) &\geq 0, \label{eq:gwp-above-2}
		\end{align}
		then the corresponding solution to \eqref{eq:inls} satisfies \eqref{scat-crite}. In particular, if $N\geq 2$ and $0<b<\min \left\{2,\frac{N}{2}\right\}$, then the solution exists globally in time and scatters in $H^1$ in the sense of \eqref{defi-scat}.
		
		\noindent (2) If 
		\begin{align} 
		P(u_0) [M(u_0)]^{\sigc} &> P(Q) [M(Q)]^{\sigc}, \label{eq:blow-above-1} \\
		V'(0)& \leq 0, \label{eq:blow-above-2}
		\end{align}
		then the corresponding solution to \eqref{eq:inls} blows up forward in time, i.e., $T^*<\infty$.
	\end{theorem}
	
	For the scattering part, Theorem \ref{theo-dyna-above} improves a recent result of Campos and Cardoso \cite{CC} at two points: (1) removing the radial assumption and (2) extending the validity of $b$. For the blow-up part, we extend the one in \cite{CC} to any dimensions $N\geq 1$. The proof of Theorem \ref{theo-dyna-above} is based on virial identities and a continuity argument in the same spirit of Duyckaerts and Roudenko \cite{DR-CMP}. 
	
	We finish the introduction by outlining the structure of the paper. In Section \ref{sec:scat}, we give the proof of the scattering criterion given in Theorem \ref{theo-scat-crite}. In Section \ref{sec:blow}, we prove the blow-up criterion given in Theorem \ref{theo-blow-crite}. Finally, we study long time dynamics of $H^1$-solutions lying below, at, and above the ground state threshold in Section \ref{sec:dynamic}.
	
	\section{Scattering criterion}
	\label{sec:scat}
	\setcounter{equation}{0}
	
	\subsection{Local theory}
	In this subsection, we recall the well-posedness theory for \eqref{eq:inls} due to \cite{Guzman, FG-JDE, FG-BBMS}. To this end, we introduce some notations. Let $\gamma \geq 0$. A pair $(q,r)$ is called $\dot{H}^\gamma$-admissible if 
	\[
	\frac{2}{q} +\frac{N}{r} =\frac{N}{2}-\gamma
	\]
	and
	\begin{align} \label{eq:cond-r}
	\left\{
	\renewcommand*{\arraystretch}{1.3}
	\begin{array}{ccl}
	\frac{2N}{N-2\gamma} <r<\frac{2N}{N-2} &\text{if}& N\geq 3, \\
	\frac{2}{1-\gamma} <r<\infty &\text{if} & N=2, \\
	\frac{2}{1-2\gamma}<r<\infty &\text{if} &N=1.
	\end{array}
	\right.
	\end{align}
	The set of all $\dot{H}^\gamma$-admissible pairs is denoted by $\Ac_\gamma$. Similarly, a pair $(q,r)$ is called $\dot{H}^{-\gamma}$-admissible if 
	\[
	\frac{2}{q}+\frac{N}{r} =\frac{N}{2}+\gamma
	\]
	and $r$ satisfies \eqref{eq:cond-r}. The set of all $\dot{H}^{-\gamma}$-admissible pairs is denoted by $\Ac_{-\gamma}$. Note that we do not consider the pair $\left(\infty, \frac{2N}{N-2\gamma}\right)$ as a $\dot{H}^\gamma$-admissible pair. The reason for doing so will be clear in Subsection \ref{subsec:profile}. When $\gamma=0$, 
	we denote $L^2$ instead of $\dot{H}^0$. In this case, the $L^2$-admissible pair is also called Schr\"odinger admissible.
	
	Let $I\subset \R$ be an interval and $\gamma \geq 0$. We define the Strichartz norm
	\[
	\|u\|_{S(I,\dot{H}^\gamma)} := \sup_{(q,r)\in \Ac_\gamma} \|u\|_{L^q_t(I,L^r_x)}.
	\]
	For a set $A \subset \R^N$, we denote 
	\[
	\|u\|_{S(I,\dot{H}^\gamma(A))} := \sup_{(q,r)\in \Ac_\gamma} \|u\|_{L^q_t(I,L^r_x(A))}.
	\]
	When $I =\R$, we omit the dependence on $\R$ and simply denote $\|u\|_{S(\dot{H}^\gamma)}$ and $\|u\|_{S(\dot{H}^\gamma(A))}$. Similarly, we define
	\[
	\|u\|_{S'(I,\dot{H}^{-\gamma})} := \inf_{(q,r)\in \Ac_{-\gamma}} \|u\|_{L^{q'}_t(I,L^{r'}_x)}
	\]
	and for $A \subset \R^N$,
	\[
	\|u\|_{S'(I,\dot{H}^{-\gamma}(A))} := \inf_{(q,r)\in \Ac_{-\gamma}} \|u\|_{L^{q'}_t(I,L^{r'}_x(A))}.
	\]
	As before, when $I=\R$, we simply use $\|u\|_{S'(\dot{H}^{-\gamma})}$ and $\|u\|_{S'(\dot{H}^{-\gamma}(A))}$. 
	
	We have the following Strichartz estimates (see e.g., \cite{Cazenave, KT, Foschi}).
	
	\begin{proposition}[Strichartz estimates \cite{Cazenave, KT, Foschi}] \label{prop-stri-est}
		Let $\gamma\geq 0$ and $I \subset \R$ be an interval. Then there exists a constant $C>0$ independent of $I$ such that
		\[
		\|e^{it\Delta} f\|_{S(I, \dot{H}^\gamma)} \leq C \|f\|_{\dot{H}^\gamma}
		\]
		and
		\[
		\left\| \int_0^t e^{i(t-s)\Delta} F(s) ds \right\|_{S(I,\dot{H}^\gamma)} \leq C \|F\|_{S'(I,\dot{H}^{-\gamma})}.
		\]
		Moreover, the above estimates still hold with $L^\infty_t(I,L^{\frac{2N}{N-2\gamma}}_x)$-norm in place of $S(I,\dot{H}^\gamma)$-norm.
	\end{proposition}
	
	We also need the following nonlinear estimates due to \cite[Lemma 2.5]{Campos} and \cite[Lemma 2.1]{CFGM}.
	
	\begin{lemma}[Nonlinear estimates \cite{Campos,CFGM}] \label{lem-non-est}
		Let $N\geq 2$, $0<b<\min\left\{2,\frac{N}{2}\right\}$, and $\frac{4-2b}{N}<\alpha<\alpha(N)$. Then there exists $\theta \in (0,\alpha)$ sufficiently small so that
		\begin{align*}
		\||x|^{-b} |u|^\alpha v\|_{S'(\dot{H}^{-\gamc})} &\lesssim \|u\|^\theta_{L^\infty_t H^1_x} \|u\|^{\alpha-\theta}_{S(\dot{H}^{\gamc})} \|v\|_{S(\dot{H}^{\gamc})}, \\
		\||x|^{-b} |u|^\alpha v\|_{S'(L^2)} &\lesssim \|u\|^\theta_{L^\infty_t H^1_x} \|u\|^{\alpha-\theta}_{S(\dot{H}^{\gamc})} \|v\|_{S(L^2)}, \\
		\|\nabla (|x|^{-b} |u|^\alpha u)\|_{S'(L^2)} &\lesssim \|u\|^\theta_{L^\infty_t H^1_x} \|u\|^{\alpha-\theta}_{S(\dot{H}^{\gamc})} \|\nabla u\|_{S(L^2)}.
		\end{align*}
		Note that if $b=0$, we can take $\theta =0$ in the above estimates.
	\end{lemma}
	
	\begin{proof}
		The first two estimates were proved in \cite[Lemma 2.5]{Campos} (for $N\geq 3$) and \cite[Lemma 2.1]{CFGM} (for $N\geq 2$). An estimate similar to the last one was proved in \cite[Lemma 2.5]{Campos} for $N\geq 3$. However, the proof in \cite{Campos} used the dual pair of the end-point $\left(2,\frac{2N}{N-2}\right)$ which, however, is excluded in our definition of $L^2$-admissible pair (see \eqref{eq:cond-r}). Thus we need a different argument. Let $\theta>0$ be a small parameter to be chosen later. We denote
		\begin{align*}
		q' &= \frac{4}{2+\theta}, & r' &= \frac{2N}{N+2-\theta}, \\
		\overline{a} &= \frac{4\alpha(\alpha+1-\theta)}{4-2b-(N-2)\alpha +\theta \alpha}, & \overline{r} &= \frac{2N\alpha(\alpha+1-\theta)}{(N+2-2b)\alpha -\theta(4-2b+\alpha)}, \\
		\overline{q} &= \frac{4\alpha(\alpha+1-\theta)}{\alpha(N\alpha-2+2b)-\theta(N\alpha-4+2b-\alpha)}, & \overline{m}_{\pm} &= \frac{N\alpha}{2-b\mp N\alpha \theta}.
		\end{align*}
		Here $(q',r')$ is the dual pair of $\left(\frac{4}{2-\theta}, \frac{2N}{N-2+\theta}\right) \in \Ac_0$. We can readily check that $(\overline{q}, \overline{r}) \in \Ac_0$ and $(\overline{a}, \overline{r}) \in \Ac_{\gamc}$ provided that $\theta>0$ is taken sufficiently small. Moreover, as $\frac{4-2b}{N}<\alpha<\frac{4-2b}{N-2}$, we have $2<\overline{m}_\pm <\frac{2N}{N-2}$ for $\theta>0$ sufficiently small. 
		
		We observe that 
		\begin{align} \label{eq:obser}
		\nabla(|x|^{-b} |u|^\alpha u) = |x|^{-b} \nabla(|u|^\alpha u) - b \frac{x}{|x|} |x|^{-b} \left( |x|^{-1} |u|^{\alpha} u\right) 
		\end{align}
		and
		\[
		\||x|^{-b} f\|_{L^{r'}_x(A)} \leq \||x|^{-b} \|_{L^{r_1}_x(A)} \|f\|_{L^{r_2}_x},
		\]
		where $A$ stands for either $B=B(0,1)$ or $B^c=\R^N \backslash B(0,1)$. To ensure $\||x|^{-b}\|_{L^{r_1}_x(A)}<\infty$, we take
		\[
		\frac{1}{r_1} = \frac{b}{N} \pm \theta^2,
		\]
		where the plus sign is for $A=B$ and the minus one is for $A=B^c$. It follows that 
		\[
		\frac{1}{r_2}=\frac{1}{r'} -\frac{1}{r_1} = \frac{N+2-2b-\theta}{2N} \mp \theta^2.
		\]
		As $\frac{1}{N} <\frac{N+2-2b}{2N}<1$ for $N\geq 2$ and $0<b<\frac{N}{2}$, we choose $\theta>0$ sufficiently small so that $1<r_2<N$ which allows us to use the Hardy's inequality (see e.g., \cite{OK})
		\[
		\||x|^{-1} f\|_{L^{r_2}_x} \leq \frac{r_2}{N-r_2} \|\nabla f\|_{L^{r_2}_x}.
		\]
		Applying the above inequality to $f=|u|^\alpha u$ and using \eqref{eq:obser}, we see that
		\[
		\|\nabla(|x|^{-b} |u|^\alpha u)\|_{L^{r'}_x} \lesssim \|\nabla(|u|^\alpha u)\|_{L^{r_2}_x}.
		\]
		By H\"older's inequality and the fact that
		\[
		\frac{1}{r_2} = \frac{\theta}{\overline{m}_{\pm}} + \frac{\alpha+1-\theta}{\overline{r}},
		\] 
		we have
		\[
		\|\nabla(|x|^{-b} |u|^\alpha u)\|_{L^{r'}_x} \lesssim \|u\|^\theta_{L^{\overline{m}_\pm}_x} \|u\|^{\alpha-\theta}_{L^{\overline{r}}_x} \|\nabla u\|_{L^{\overline{r}}_x}.
		\]
		By H\"older's inequality in time with 
		\[
		\frac{1}{q'} = \frac{\alpha-\theta}{\overline{a}} + \frac{1}{\overline{q}},
		\] 
		we get
		\begin{align*}
		\|\nabla(|x|^{-b} |u|^\alpha u)\|_{L^{q'}_t L^{r'}_x} &\lesssim \|u\|^\theta_{L^\infty_t L^{\overline{m}_\pm}_x} \|u\|^{\alpha-\theta}_{L^{\overline{a}}_t L^{\overline{r}}_x} \|\nabla u\|_{L^{\overline{q}}_t L^{\overline{r}}_x} \\
		&\lesssim \|u\|^\theta_{L^\infty_t H^1_x} \|u\|^{\alpha-\theta}_{L^{\overline{a}}_t L^{\overline{r}}_x} \|\nabla u\|_{L^{\overline{q}}_t L^{\overline{r}}_x},
		\end{align*}
		where the last inequality follows from the Sobolev embedding as $2<\overline{m}_\pm <\frac{2N}{N-2}$. The proof is complete.
	\end{proof}

	Using Proposition \ref{prop-stri-est} and Lemma \ref{lem-non-est}, we have the following result.
	
	\begin{proposition}[Local theory \cite{Guzman,FG-JDE, FG-BBMS}]  \label{lem-local-theory}
		Let $N\geq 2$, $0<b<\min \left\{2,\frac{N}{2}\right\}$, and $\frac{4-2b}{N}<\alpha<\alpha(N)$.
		\begin{itemize} [leftmargin=5mm]
			\item[(1)] (Local well-posedness) Let $u_0 \in H^1$. Then there exist $T_*,T^* \in (0,\infty]$, and a unique local solution to \eqref{eq:inls} satisfying
			\[
			u \in C((-T_*, T^*), H^1) \cap  L^q_{\loc}(-T_*,T^*), W^{1,r})
			\] 
			for any $(q,r) \in \Ac_0$. If $T^*<\infty$ (resp. $T_*<\infty$), then $\lim_{t\nearrow T^*} \|\nabla u(t)\|_{L^2} =\infty$ (resp. $\lim_{t\searrow -T_*} \|\nabla u(t)\|_{L^2} =\infty$).
			\item[(2)] (Small data scattering) Let $T>0$ be such that $\|u(T)\|_{H^1} \leq A$ for some constant $A>0$. Then there exists $\delta=\delta(A)>0$ such that if 
			\[
			\|e^{i(t-T)\Delta} u(T)\|_{S([T,\infty),\dot{H}^{\gamc})} <\delta,
			\]
			then the corresponding solution to \eqref{eq:inls} with initial data $\left.u\right|_{t=T} = u(T)$ exists globally in time and satisfies
			\begin{align*}
			\|u\|_{S([T,\infty), \dot{H}^{\gamc})} &\leq 2 \|e^{i(t-T)\Delta} u(T)\|_{S([T,\infty), \dot{H}^{\gamc})}, \\
			\|\scal{\nabla} u\|_{S([T,\infty),L^2)} &\leq C \|u(T)\|_{H^1}.
			\end{align*}
			\item[(3)] (Scattering condition) Let $u$ be a global solution to \eqref{eq:inls}. Assume that
			\[
			\|u\|_{L^\infty_t(\R, H^1_x)} \leq A, \quad \|u\|_{S(\dot{H}^{\gamc})} <\infty.
			\]
			Then $u$ scatters in $H^1$ in both directions. 
		\end{itemize}
	\end{proposition}
	Here we have used the following convention
	\[
	\|\scal{\nabla} f\|_X := \|f\|_X + \|\nabla f\|_X, \quad f \in X.
	\]
	
	We also recall the following stability result due to \cite{FG-JDE, FG-BBMS}.
	
	\begin{lemma}[Stability] \label{lem-sta}
		Let $N\geq 2$, $0<b<\min \left\{2,\frac{N}{2}\right\}$, and $\frac{4-2b}{N}<\alpha <\alpha(N)$. Let $0 \in I \subseteq \R$ and $\tilde{u}: I \times \R^N \rightarrow \C$ be a solution to 
		\[
		i\partial_t \tilde{u} + \Delta \tilde{u} + |x|^{-b}|\tilde{u}|^\alpha \tilde{u} = e
		\]
		with $\left.\tilde{u}\right|_{t=0} = \tilde{u}_0$ satisfying
		\[
		\|\tilde{u}\|_{L^\infty_t(I,H^1_x)} \leq M, \quad \|\tilde{u}\|_{S(I,\dot{H}^{\gamc})} \leq L
		\]
		for some constants $M, L>0$. Let $u_0 \in H^1$ be such that
		\[
		\|u_0- \tilde{u}_0\|_{H^1} \leq M', \quad \|e^{it\Delta} (u_0 - \tilde{u}_0)\|_{S(I,\dot{H}^{\gamc})} \leq \vareps
		\]
		for some $M'>0$ and some $0<\vareps <\vareps_1=\vareps_1(M,M',L)$. Suppose that
		\[
		\|\scal{\nabla} e\|_{S'(I,L^2)} + \|e\|_{S'(I,\dot{H}^{-\gamc})} \leq \vareps.
		\]
		Then there exists a unique solution $u: I \times \R^N \rightarrow \C$ to \eqref{eq:inls} with $\left. u\right|_{t=0} = u_0$ satisfying
		\begin{align*}
		\|u-\tilde{u}\|_{S(I,\dot{H}^{\gamc})} &\leq C(M,M', L)\vareps, \\
		\|u\|_{L^\infty_t(I,H^1_x)}+\|\scal{\nabla} u\|_{S(I,L^2)} + \|u\|_{S(I,\dot{H}^{\gamc})} &\leq C(M,M',L).
		\end{align*}
	\end{lemma}
	
	\begin{remark} \label{rem-sta}
		If we assume in addition that
		\[
		\|e^{it\Delta} (u_0-\tilde{u}_0)\|_{L^\infty_t(I, L^{\frac{2N}{N-2\gamc}}_x)} \leq \vareps, 
		\]
		then
		\[
		\|u-\tilde{u}\|_{L^\infty_t(I, L^{\frac{2N}{N-2\gamc}}_x)} \leq C(M, M',L)\vareps.
		\]
		In fact, by Duhamel's formula, we have
		\begin{align*}
		u(t) - \tilde{u}(t) = e^{it\Delta}(u_0 - \tilde{u}_0) &+ i \int_0^t e^{i(t-s)\Delta} (|x|^{-b} |u(s)|^\alpha u(s) - |x|^{-b} |\tilde{u}(s)|^\alpha \tilde{u}(s)) ds \\
		&+ i \int_0^t e^{i(t-s)\Delta} e(s) ds.
		\end{align*}
		By Strichartz estimates and Lemma \ref{lem-non-est}, we have
		\begin{align*}
		\|u-\tilde{u}\|_{L^\infty_t(I, L^{\frac{2N}{N-2\gamc}}_x)} &\leq \|e^{it\Delta}(u_0-\tilde{u}_0)\|_{L^\infty_t (I, L^{\frac{2N}{N-2\gamc}}_x)}  + \|e\|_{S'(I,\dot{H}^{-\gamc})} \\
		&\mathrel{\phantom{\leq \|e^{it\Delta}(u_0-\tilde{u}_0)\|_{L^\infty_t (I, L^{\frac{2N}{N-2\gamc}}_x)}}}+ C\||x|^{-b} |u|^\alpha u - |x|^{-b} |\tilde{u}|^\alpha \tilde{u}\|_{S'(I,\dot{H}^{\gamc})} \\
		&\leq \|e^{it\Delta}(u_0-\tilde{u}_0)\|_{L^\infty_t (I, L^{\frac{2N}{N-2\gamc}}_x)}  + \|e\|_{S'(I,\dot{H}^{-\gamc})}  \\
		&\mathrel{\phantom{\leq}} + C \left(\|u\|^\theta_{L^\infty_t(I,H^1_x)} \|u\|^{\alpha-\theta}_{S(I,\dot{H}^{\gamc})}+ \|\tilde{u}\|^\theta_{L^\infty_t(I,H^1_x)} \|\tilde{u}\|^{\alpha-\theta}_{S(I,\dot{H}^{\gamc})} \right) \|u-\tilde{u}\|_{S(I,\dot{H}^{\gamc})} \\
		&\leq C(M,M', L) \vareps.
		\end{align*}
	\end{remark}

	\subsection{Variational analysis}
	We recall some properties of the ground state $Q$ which is the unique positive radial solution to \eqref{eq:ground}. The ground state $Q$ optimizes the weighted Gagliardo-Nirenberg inequality: for $N\geq 1$ and $0<b<\min \{2, N\}$,
	\begin{align} \label{eq:GN-ineq}
	P(f) \leq C_{\opt} \|\nabla f\|^{\frac{N\alpha+2b}{2}}_{L^2} \|f\|^{\frac{4-2b-(N-2)\alpha}{2}}_{L^2}, \quad f \in H^1(\R^N),
	\end{align}
	that is
	\[
	C_{\opt} = P(Q) \div \left[\|\nabla Q\|^{\frac{N\alpha+2b}{2}}_{L^2} \|Q\|^{\frac{4-2b-(N-2)\alpha}{2}}_{L^2} \right],
	\]
	where $P(f)$ is as in \eqref{eq:P}. We have the following Pohozaev's identities (see e.g., \cite{Farah})
	\begin{align} \label{eq:poho-iden}
	\|Q\|^2_{L^2} = \frac{4-2b-(N-2)\alpha}{N\alpha+2b} \|\nabla Q\|^2_{L^2} = \frac{4-2b-(N-2)\alpha}{2(\alpha+2)} P(Q).
	\end{align}
	In particular, we have
	\begin{align} \label{eq:opt-cons}
	C_{\opt} = \frac{2(\alpha+2)}{N\alpha+2b} \left(\|\nabla Q\|_{L^2} \|Q\|^{\sigc}_{L^2} \right)^{-\frac{N\alpha-4+2b}{2}}.
	\end{align}
	We also have
	\begin{align} \label{eq:ener-Q}
	E(Q) = \frac{N\alpha-4+2b}{2(N\alpha+2b)} \|\nabla Q\|^2_{L^2} = \frac{N\alpha-4+2b}{4(\alpha+2)} P(Q)
	\end{align}
	hence
	\begin{align} \label{eq:prop-Q}
	E(Q) [M(Q)]^{\sigc} = \frac{N\alpha-4+2b}{2(N\alpha+2b)} \left(\|\nabla Q\|_{L^2} \|Q\|^{\sigc}_{L^2} \right)^2.
	\end{align}
	
	\subsection{Profile decompositions} \label{subsec:profile}
	In this subsection, we recall the linear profile decomposition and construct some nonlinear profiles associated to \eqref{eq:inls}. Let us start with the following result due to \cite{FXC, Guevara} (see also \cite{FG-JDE,FG-BBMS}).
	
	\begin{lemma}[Linear profile decomposition \cite{FXC, Guevara, FG-JDE, FG-BBMS}] \label{lem-line-prof} Let $N\geq 1$, $0<b<\min\{2,N\}$, and $\frac{4-2b}{N} <\alpha <\alpha(N)$. Let $(\phi_n)_{n\geq 1}$ be a uniformly bounded sequence in $H^1$. Then for each integer $J\geq 1$, there exists a subsequence, still denoted by $\phi_n$, and
		\begin{itemize}[leftmargin=5mm]
			\item for each $1\leq j \leq J$, there exists a fixed profile $\psi^j \in H^1$;
			\item for each $1\leq j \leq J$, there exists a sequence of time shifts $(t^j_n)_{n\geq 1} \subset \R$;
			\item for each $1\leq j \leq J$, there exists a sequence of space shifts $(x^j_n)_{n\geq 1} \subset \R^N$;
			\item there exists a sequence of remainders $(W^J_n)_{n\geq 1} \subset H^1$;
		\end{itemize}
		such that
		\begin{align} \label{line-prof}
		\phi_n(x) = \sum_{j=1}^J e^{-it^j_n \Delta} \psi^j(x-x^j_n) + W^J_n(x).
		\end{align}
		The time and space shifts have a pairwise divergence property, i.e., for $1\leq j \ne k \leq J$, we have
		\begin{align} \label{diver-proper}
		\lim_{n\rightarrow \infty} |t^j_n-t^k_n| + |x^j_n - x^k_n| = \infty.
		\end{align}
		The remainder has the following asymptotic smallness property
		\begin{align*}
		\lim_{J \rightarrow \infty} \left[ \lim_{n\rightarrow \infty} \|e^{it\Delta} W^J_n\|_{S(\dot{H}^{\gamc}) \cap L^\infty_t(\R, L^{\frac{2N}{N-2\gamc}}_x)} \right] =0,
		\end{align*}
		where $\gamc$ is as in \eqref{eq:gamc}. Moreover, for fixed $J$ and $\gamma \in [0,1]$, we have the asymptotic Pythagorean expansions
		\begin{align*} 
		\|\phi_n\|^2_{\dot{H}^\gamma} = \sum_{j=1}^J \|\psi^j\|^2_{\dot{H}^\gamma} + \|W^J_n\|^2_{\dot{H}^\gamma} + o_n(1).
		\end{align*}
		Finally, we may assume either $t^j_n \equiv 0$ or $t^j_n \rightarrow \pm \infty$, and either $x^j_n \equiv 0$ or $|x^j_n| \rightarrow \infty$.
	\end{lemma}
	
	
	In the next lemmas, we will construct nonlinear profiles associated to the linear ones with either divergent time or divergent space shifts.
	
	\begin{lemma}[Nonlinear profile with divergent time shift and no space translation] \label{lem-non-prof-dive-time}
		Let $N\geq 2$, $0<b<\min\left\{2,\frac{N}{2}\right\}$, and $\frac{4-2b}{N}<\alpha<\alpha(N)$. Let $\psi \in H^1$ and $t_n\rightarrow \infty$. Let $v_n:C((-T_*,T^*),H^1)$ denote the maximal solution to \eqref{eq:inls} with initial data
		\begin{align} \label{defi:vn-0}
		v_n(0,x) = e^{-it_n \Delta} \psi(x).
		\end{align}
		Then for $n$ sufficiently large, $v_n$ exists globally backward in time, i.e., $T_*=\infty$. Moreover, we have for any $0\leq T<T^*$, 
		\begin{align} \label{small-diff-vn}
		\lim_{n\rightarrow \infty} \|\scal{\nabla}(v_n - \psi_n)\|_{S((-\infty,T), L^2)} + \|v_n-\psi_n\|_{S((-\infty,T),\dot{H}^{\gamc})} =0,
		\end{align}
		where
		\begin{align} \label{defi:psi-n}
		\psi_n(t,x):= e^{i(t-t_n)\Delta} \psi(x).
		\end{align}
		In addition, we have
		\begin{align} \label{H1-norm-prof}
		\lim_{n\rightarrow \infty} \|v_n- \psi_n\|_{L^\infty_t((-\infty,T), H^1_x)} =0. 
		\end{align}
		Similarly, if $t_n \rightarrow -\infty$ and $v_n:C((-T_*,T^*), H^1)$ is the maximal solution to \eqref{eq:inls} with initial data \eqref{defi:vn-0}, then for $n$ sufficiently large, $v_n$ exists globally forward in time, i.e., $T^*=\infty$. Moreover, we have for any $0 \leq T < T_*$, 
		\[
		\lim_{n\rightarrow \infty} \|\scal{\nabla}(v_n-\psi_n)\|_{S((-T,\infty),L^2)} + \|v_n-\psi_n\|_{S((-T,\infty),\dot{H}^{\gamc})} =0,
		\]
		where $\psi_n$ is as in \eqref{defi:psi-n}. Moreover,
		\[
		\lim_{n\rightarrow \infty} \|v_n- \psi_n\|_{L^\infty_t((-T, \infty), H^1_x)} =0.
		\]
	\end{lemma}
	
	\begin{proof}
		We only treat the first point, the second point is similar. We see that $\psi_n$ satisfies
		\[
		i\partial_t \psi_n + \Delta \psi_n + |x|^{-b} |\psi_n|^\alpha \psi_n = e_n
		\]
		with $e_n:= |x|^{-b} |\psi_n|^\alpha \psi_n$. Since $v_n(0) = \psi_n(0)$, the result follows from the stability given in Lemma \ref{lem-sta} provided that
		\begin{align} \label{est-error}
		\lim_{n\rightarrow \infty} \|\scal{\nabla} e_n\|_{S'((-\infty,T),L^2)} + \|e_n\|_{S'((-\infty,T),\dot{H}^{-\gamc})} =0.
		\end{align}
		By Lemma \ref{lem-non-est}, we have
		\begin{align*}
		\|\scal{\nabla} e_n\|_{S'((-\infty,T),L^2)} &= \|\scal{\nabla} (|x|^{-b} |\psi_n|^\alpha \psi_n)\|_{S'((-\infty,T),L^2)} \\
		&= \|\scal{\nabla} (|x|^{-b} |e^{it\Delta} \psi|^\alpha e^{it\Delta} \psi) \|_{S'((-\infty, T-t_n), L^2)} \\
		&\lesssim \|e^{it\Delta} \psi\|^\theta_{L^\infty_t((-\infty,T-t_n),H^1_x)} \|e^{it\Delta} \psi\|^{\alpha-\theta}_{S((-\infty,T-t_n),\dot{H}^{\gamc})} \\
		&\mathrel{\phantom{\lesssim \|e^{it\Delta} \psi\|^\theta_{L^\infty_t((-\infty,T-t_n),H^1_x)} }} \times \|\scal{\nabla} e^{it\Delta} \psi\|_{S((-\infty,T-t_n),L^2)} \rightarrow 0
		\end{align*}
		as $n\rightarrow\infty$ as $\scal{\nabla} e^{it\Delta} \psi \in S(L^2)$ and $e^{it\Delta} \psi \in S(\dot{H}^{\gamc})$. Here we do not include the pairs $(\infty, 2)$ and $\left(\infty, \frac{2N}{N-2\gamc}\right)$ into the definitions of $L^2$ and $\dot{H}^{\gamc}$ admissibility, respectively. Similarly, we have
		\begin{align*}
		\|e_n\|_{S'((-\infty,0),\dot{H}^{-\gamc})} &= \||x|^{-b} |e^{it\Delta} \psi|^\alpha e^{it\Delta} \psi\|_{S'((-\infty,T-t_n),\dot{H}^{-\gamc})} \\
		&\lesssim \|e^{it\Delta}\psi\|^\theta_{L^\infty_t((-\infty,T-t_n),H^1_x)} \|e^{it\Delta} \psi\|^{\alpha+1-\theta}_{S((-\infty,T-t_n),\dot{H}^{\gamc})} \rightarrow 0
		\end{align*}
		as $n\rightarrow \infty$. This shows \eqref{est-error}.
		
		We next show \eqref{H1-norm-prof}. To see this, we have from \eqref{small-diff-vn},
		\[
		\|\scal{\nabla} \psi_n\|_{S((-\infty,T), L^2)} = \|\scal{\nabla} e^{it\Delta} \psi\|_{S((-\infty, T-t_n), L^2)} \rightarrow 0 \text{ as } n \rightarrow \infty,
		\]
		and similarly for $\|\psi_n\|_{S((-\infty, T), \dot{H}^{\gamc})}$ that 
		\[
		\lim_{n\rightarrow \infty} \|\scal{\nabla} v_n\|_{S((-\infty,T), L^2)} + \|v_n\|_{S((-\infty, T), \dot{H}^{\gamc})} =0.
		\]
		This together with Strichartz estimates, Lemma \ref{lem-non-est}, and the fact that $\psi_n(t,x) = e^{it\Delta} v_n(0,x)$ imply $\|v_n\|_{L^\infty_t ((-\infty,T), H^1_x)} \lesssim 1$. By Lemma \ref{lem-non-est}, we have
		\[
		\|v_n- \psi_n\|_{L^\infty_t((-\infty, T), H^1_x)} \lesssim  \|v_n\|^\theta_{L^\infty_t((-\infty, T),H^1_x)} \|v_n\|^{\alpha-\theta}_{S((-\infty,T), \dot{H}^{\gamc})} \|\scal{\nabla} v_n\|_{S((-\infty, T), L^2)} \rightarrow 0
		\]
		as $n\rightarrow \infty$. The proof is complete.
	\end{proof}
	
	\begin{lemma}[Nonlinear profile with divergent space shift] \label{lem-non-prof-dive-spac}
		Let $N\geq 2$, $0<b<\min\left\{2,\frac{N}{2}\right\}$, and $\frac{4-2b}{N}<\alpha<\alpha(N)$. Let $\psi \in H^1$ and $(t_n, x_n) \in \R \times \R^N$ satisfying $|x_n|\rightarrow \infty$ as $n\rightarrow \infty$. Let $v_n:C((-T_*,T^*),H^1)$ denote the maximal solution to \eqref{eq:inls} with initial data
		\begin{align} \label{defi:vn-0-xn}
		v_n(0,x) = e^{-it_n \Delta} \psi(x-x_n).
		\end{align}
		Then for $n$ sufficiently large, $v_n$ exists globally in time, i.e., $T_*=T^*=\infty$. Moreover, we have 
		\[
		\lim_{n\rightarrow \infty} \|\scal{\nabla}(v_n - \psi_n)\|_{S(L^2)} + \|v_n-\psi_n\|_{S(\dot{H}^{\gamc})} =0,
		\]
		where
		\begin{align} \label{defi:psi-n-xn}
		\psi_n(t,x):= e^{i(t-t_n)\Delta} \psi(x-x_n).
		\end{align}
	\end{lemma}
	
	\begin{remark}
		The construction of nonlinear profiles with divergent space translations was first established by Miao, Murphy, and Zheng \cite{MMZ} for \eqref{eq:inls} with $\alpha =2$ and $N=3$. This result was recently extended to \eqref{eq:inls} with $N\geq 2$ by Cardoso, Farah, Guzm\'an, and Murphy \cite{CFGM}. Here we give a refine result with a simple proof compared to the ones in \cite{MMZ, CFGM}. More precisely, for a linear profile with a divergent space shift, the associated nonlinear profile is close to the solution of the underlying linear Schr\"odinger equation.
	\end{remark}

	\noindent {\it Proof of Lemma \ref{lem-non-prof-dive-spac}.}
	As in the proof of Lemma \ref{lem-non-prof-dive-time}, it suffices to show
	\begin{align} \label{est-error-xn}
	\lim_{n\rightarrow \infty} \|\scal{\nabla} e_n\|_{S'(L^2)} + \|e_n\|_{S'(\dot{H}^{-\gamc})} =0.
	\end{align}
	To see this, we take $\vareps>0$. We have
	\begin{align*}
	\|\scal{\nabla} e_n\|_{S'(L^2)} &= \|\scal{\nabla}(|x|^{-b} |\psi_n|^\alpha \psi_n)\|_{S'(L^2)} \\
	&=\|\scal{\nabla} (|x+x_n|^{-b} |e^{it\Delta} \psi|^\alpha e^{it\Delta} \psi)\|_{S'(L^2)} \\
	&\leq \|\scal{\nabla} (|x+x_n|^{-b} |e^{it\Delta} \psi|^\alpha e^{it\Delta} \psi)\|_{S'(L^2(B_R))} \\
	&\mathrel{\phantom{\leq}} + \|\scal{\nabla} (|x+x_n|^{-b} |e^{it\Delta} \psi|^\alpha e^{it\Delta} \psi)\|_{S'(L^2(B_R^c))},  
	\end{align*}
	where $B_R:= \left\{x \in \R^N \ : \ |x| \leq R\right\}$ and $B_R^c= \R^N \backslash B_R$ with $R>0$ to be chosen later. 
	
	On $B_R^c$, by splitting $B_R^c = \Omega_1 \cup \Omega_2$ with
	\[
	\Omega_1 = \left\{x \in \R^N \ :\ |x| \geq R, |x+x_n| \leq 1\right\}, \quad \Omega_2 = \left\{x \in \R^N \ :\ |x| \geq R, |x+x_n| \geq 1\right\},
	\]
	the same argument as in the proof of Lemma \ref{lem-non-est} implies that
	\begin{align*}
	\|\scal{\nabla}(|x+x_n|^{-b} |\varphi|^\alpha \varphi)\|_{S'(L^2(B^c_R))} \lesssim \|\varphi\|^\theta_{L^\infty_t H^1_x(B^c_R)} \|\varphi\|^{\alpha-\theta}_{S(\dot{H}^{\gamc}(B_R^c))} \|\scal{\nabla} \varphi\|_{S(L^2(B^c_R))},
	\end{align*}
	where 
	\[
	\varphi(t,x):= e^{it\Delta} \psi(x).
	\]
	As $\varphi \in S(\dot{H}^{\gamc})$ and $\scal{\nabla} \varphi \in S(L^2)$, we see that 
	\[
	\|\varphi\|_{S(\dot{H}^{\gamc}(B_R^c))}, \|\scal{\nabla} \varphi\|_{S(L^2(B^c_R))} \rightarrow 0 \text{ as } R\rightarrow \infty.
	\]
	Note that it is crucial to exclude the pairs $(\infty,2)$ and $\left(\infty, \frac{2N}{N-2\gamc}\right)$ from the definitions of $L^2$ and $\dot{H}^{\gamc}$ admissible conditions, respectively. This shows that for $R_0>0$ sufficiently large,
	\[
	\|\scal{\nabla}(|x+x_n|^{-b} |\varphi|^\alpha \varphi)\|_{S'(L^2(B_{R_0}^c))} <\frac{\vareps}{4}
	\]
	for all $n\geq 1$.
	
	Next, for $x \in B_{R_0}$, as $|x_n| \rightarrow \infty$, we have $|x+x_n| \geq |x_n|-|x| \geq \frac{|x_n|}{2}$ for $n$ sufficiently large. It follows from Lemma \ref{lem-non-est} that
	\begin{align*}
	\||x+x_n|^{-b} |\varphi|^\alpha \varphi \|_{S'(L^2(B_{R_0}))} \lesssim |x_n|^{-b} \||\varphi|^\alpha \varphi\|_{S'(L^2)} \lesssim |x_n|^{-b} \|\varphi\|^\alpha_{S(\dot{H}^{\gamc})} \|\varphi\|_{S(L^2)} \rightarrow 0
	\end{align*}
	as $n\rightarrow \infty$. Similarly, we have
	\begin{align*}
	\|\nabla(|x+x_n|^{-b} |\varphi|^\alpha \varphi)&\|_{S'(L^2(B_{R_0}))} \\
	&\lesssim \||x+x_n|^{-b} \nabla(|\varphi|^\alpha \varphi)\|_{S'(L^2(B_{R_0}))} + \||x+x_n|^{-b-1} |\varphi|^\alpha \varphi\|_{S'(L^2(B_{R_0}))} \\
	&\lesssim |x_n|^{-b} \|\nabla(|\varphi|^\alpha \varphi)\|_{S'(L^2)} + |x_n|^{-b-1} \||\varphi|^\alpha \varphi\|_{S'(L^2)} \\
	&\lesssim |x_n|^{-b} \|\varphi\|^\alpha_{S(\dot{H}^{\gamc})} \|\nabla \varphi\|_{S(L^2)}  +|x_n|^{-b-1} \|\varphi\|^\alpha_{S(\dot{H}^{\gamc})} \|\varphi\|_{S(L^2)} \rightarrow 0
	\end{align*}
	as $n\rightarrow \infty$. Thus there exists $n_1>0$ sufficiently large such that for all $n\geq n_1$,
	\[
	\|\scal{\nabla}(|x+x_n|^{-b} |\varphi|^\alpha \varphi)\|_{S'(L^2(B_{R_0}))} <\frac{\vareps}{4},
	\]
	hence
	\[
	\|\scal{\nabla}(|x+x_n|^{-b} |\varphi|^\alpha \varphi)\|_{S'(L^2)} <\frac{\vareps}{2}.
	\]
	A similar argument show that for all $n\geq n_2$ with $n_2>0$ sufficiently large,
	\[
	\||x+x_n|^{-b} |\varphi|^\alpha \varphi\|_{S'(\dot{H}^{-\gamc})} <\frac{\vareps}{2}.
	\]
	Therefore, we have for all $n\geq \max \{n_1, n_2\}$,
	\[
	\|\scal{\nabla}(|x+x_n|^{-b} |\varphi|^\alpha \varphi)\|_{S'(L^2)} + \||x+x_n|^{-b} |\varphi|^\alpha \varphi\|_{S'(\dot{H}^{-\gamc})} <\vareps
	\]
	which proves \eqref{est-error-xn}. The proof is complete.
	\hfill $\Box$
	
	
	
	\subsection{Energy scattering}
	In this section, we give the proof of the scattering criterion given in Theorem \ref{theo-scat-crite}. To this end, we need the following coercivity lemma.
	
	\begin{lemma} \label{lem-coer}
		Let $N\geq 1$, $0<b<\min \left\{2, N\right\}$, and $\frac{4-2b}{N}<\alpha<\alpha(N)$. Let $f \in H^1$ satisfy 
		\begin{align} \label{eq:A}
		P(f) [M(f)]^{\sigc} \leq A < P(Q) [M(Q)]^{\sigc}
		\end{align}
		for some constant $A>0$. Then there exists $\nu = \nu (A, Q)>0$ such that
		\begin{align} 
		G(f) &\geq \nu \|\nabla f\|^2_{L^2}, \label{eq:coer-G} \\
		E(f) &\geq \frac{\nu}{2} \|\nabla f\|^2_{L^2}. \label{eq:coer-E}
		\end{align}
	\end{lemma}
	
	\begin{proof}
		We write
		\[
		A = (1-\rho) P(Q) [M(Q)]^{\sigc}
		\]
		for some $\rho=\rho(A, Q)\in (0,1)$. It follows from \eqref{eq:GN-ineq}, \eqref{eq:poho-iden}, \eqref{eq:opt-cons}, and \eqref{eq:A} that
		\begin{align*}
		[P(f)]^{\frac{N\alpha+2b}{4}} &\leq C_{\opt} \left( P(f) [M(f)]^{\sigc}\right)^{\frac{N\alpha-4+2b}{4}} \|\nabla f\|^{\frac{N\alpha+2b}{2}}_{L^2} \\
		&=\frac{2(\alpha+2)}{N\alpha+2b} \left(\frac{P(f)[M(f)]^{\sigc}}{\|\nabla Q\|^2_{L^2} \|Q\|^{2\sigc}_{L^2}}\right)^{\frac{N\alpha-4+2b}{4}} \|\nabla f\|^{\frac{N\alpha+2b}{2}}_{L^2} \\
		&= \left(\frac{P(f)[M(f)]^{\sigc}}{P(Q) [M(Q)]^{\sigc}_{L^2}}\right)^{\frac{N\alpha-4+2b}{4}} \left(\frac{2(\alpha+2)}{N\alpha+2b} \|\nabla f\|^2_{L^2}\right)^{\frac{N\alpha+2b}{4}} \\
		& \leq (1-\rho)^{\frac{N\alpha-4+2b}{4}} \left(\frac{2(\alpha+2)}{N\alpha+2b} \|\nabla f\|^2_{L^2}\right)^{\frac{N\alpha+2b}{4}}
		\end{align*}
		which implies
		\[
		P(f) \leq \frac{2(\alpha+2)}{N\alpha+2b} (1-\rho)^{\frac{N\alpha-4+2b}{N\alpha+2b}} \|\nabla f\|^2_{L^2}.
		\]
		Thus we get
		\[
		G(f) = \|\nabla f\|^2_{L^2} -\frac{N\alpha+2b}{2(\alpha+2)} P(f) \geq \left(1-(1-\rho)^{\frac{N\alpha-4+2b}{N\alpha+2b}}\right) \|\nabla f\|^2_{L^2}
		\]
		which proves \eqref{eq:coer-G}. As $N\alpha-4+2b >0$, we have
		\[
		E(f) =\frac{1}{2} G(f) +\frac{N\alpha-4+2b}{2(\alpha+2)} P(f) \geq \frac{1}{2} G(f) 
		\]
		which shows \eqref{eq:coer-E}. The proof is complete.
	\end{proof}

	We are now able to give the proof of Theorem \ref{theo-scat-crite}.
	
	\noindent {\it Proof of Theorem \ref{theo-scat-crite}.}
	Let $u: [0,T^*) \times \R^N \rightarrow \C$ be a $H^1$-solution to \eqref{eq:inls} satisfying \eqref{scat-crite}. By the conservation of mass and energy, we infer from \eqref{scat-crite} that 
	\[
	\sup_{t\in [0,T^*)} \|\nabla u(t)\|_{L^2} \leq C(E, Q) <\infty.
	\]
	By the local well-posedness given in Lemma \ref{lem-local-theory}, we have $T^*=\infty$. 
	
	Let $A>0$ and $\delta>0$. We define
	\begin{equation*}
	S(A,\delta):= \sup \left\{ \|u\|_{S([0,\infty), \dot{H}^{\gamc})} \ : \ u \text{ is a solution to } \eqref{eq:inls} \text{ satisfying } \eqref{eq:scat-cond} \right\},
	\end{equation*}
	where
	\begin{align} \label{eq:scat-cond}
	\sup_{t\in [0,\infty)} P(u(t))[M(u(t))]^{\sigc} \leq A, \quad E(u) [M(u)]^{\sigc} \leq \delta.
	\end{align}
	Thanks to the scattering condition (see again Lemma \ref{lem-local-theory}) and the definition of $S(A,\delta)$, Theorem \ref{theo-scat-crite} is reduced to show the following proposition.
	
	\begin{proposition} \label{prop-scat}
		Let $N\geq 2$, $0<b<\min \left\{2,\frac{N}{2}\right\}$, and $\frac{4-2b}{N}<\alpha<\alpha(N)$. If $A< P(Q) [M(Q)]^{\sigc}$, then for all \footnote{Note the energy is positive due to Lemma \ref{lem-coer}.} $\delta>0$, $S(A,\delta)<\infty$.
	\end{proposition}
	The proof of Proposition \ref{prop-scat} is based on the concentration/compactness and rigidity argument introduced by Kenig and Merle \cite{KM} (see also \cite{DHR}). The main difficulty comes from the fact that the potential energy $P(u(t))$ is not conserved along the time evolution of \eqref{eq:inls}. To overcome the difficulty, we establish a Pythagorean decomposition along the bounded INLS flow (see Lemma \ref{lem-pytha-expan}). In the context of the standard NLS, a similar result was shown by Guevara in \cite[Lemma 3.9]{Guevara}. 
	
	The proof of Proposition \ref{prop-scat} is done by several steps.
	
	\vspace{3mm}
	
	\noindent {\bf Step 1. Small data scattering.}  By \eqref{eq:coer-E}, we have
	\[
	\|u_0\|^{\frac{2}{\gamc}}_{\dot{H}^{\gamc}} \leq \|\nabla u_0\|^2_{L^2} \|u_0\|^{2\sigc}_{L^2} \leq \frac{2}{\nu} E(u_0) [M(u_0)]^{\sigc} \leq \frac{2\delta}{\nu}.
	\]
	By taking $\delta>0$ sufficiently small, we see that $\|u_0\|_{\dot{H}^{\gamc}}$ is small which, by the small data scattering given in Lemma \ref{lem-local-theory}, implies $S(A, \delta)<\infty$.
	
	\vspace{3mm}
	
	\noindent {\bf Step 2. Existence of a critical solution.} Assume by contradiction that $S(A,\delta)=\infty$ for some $\delta>0$. By Step 1, 
	\begin{align} \label{eq:deltc}
	\deltc:= \deltc(A) := \inf \left\{\delta>0 \ : \ S(A,\delta) =\infty\right\}
	\end{align}
	is well-defined and positive. From the definition of $\deltc$, we have the following observations:
	\begin{itemize}[leftmargin=5mm]
		\item[(1)] If $u$ is a solution to \eqref{eq:inls} satisfying
		\[
		\sup_{t\in [0,\infty)} P(u(t)) [M(u(t))]^{\sigc} \leq A, \quad E(u) [M(u)]^{\sigc} <\deltc,
		\]
		then $\|u\|_{S([0,\infty), \dot{H}^{\sigc})} <\infty$ and the solution scatters in $H^1$ forward in time.
		\item[(2)] There exists a sequence of solution $u_n$ to \eqref{eq:inls} with initial data $u_{n,0}$ such that
		\begin{align} \label{eq:cond-un}
		\begin{aligned}
		\sup_{t\in [0,\infty)} P(u_n(t))[M(u_n(t))]^{\sigc} &\leq A \text{ for all } n, \\
		E(u_n) [M(u_n)]^{\sigc} &\searrow \deltc \text{ as } n\rightarrow \infty, \\
		\|u_n\|_{S([0,\infty), \dot{H}^{\gamc})} &= \infty \text{ for all } n.
		\end{aligned}
		\end{align}
	\end{itemize}
	We will prove that there exists a critical solution $u_{\ct}$ to \eqref{eq:inls} with initial data $u_{\ct,0}$ satisfying
	\begin{align} \label{eq:uc}
	\begin{aligned}
	M(u_{\ct}) &=1, \\
	\sup_{t\in [0,\infty)} P(u_{\ct}(t)) &\leq A, \\
	E(u_{\ct}) &= \deltc, \\
	\|u_{\ct}\|_{S([0,\infty), \dot{H}^{\gamc})} &=\infty.
	\end{aligned}
	\end{align}
	To see this, we consider the sequence $(u_{n,0})_{n\geq 1}$. Thanks to the scaling \eqref{eq:scaling}, we can assume that $M(u_{n,0})=1$ for all $n$. By the conservation of mass and energy, \eqref{eq:cond-un} becomes
	\begin{align} \label{eq:un-0}
	\begin{aligned}
	M(u_{n,0}) &=1 \text{ for all } n, \\
	\sup_{t\in [0,\infty)} P(u_n(t)) &\leq A \text{ for all } n, \\
	E(u_{n,0}) &\searrow \deltc \text{ as } n\rightarrow \infty, \\
	\|u_n\|_{S([0,\infty), \dot{H}^{\gamc})} &=\infty \text{ for all } n.
	\end{aligned}
	\end{align}
	Since $(u_{n,0})_{n\geq 1}$ is bounded in $H^1$, we apply the linear profile decomposition to $u_{n,0}$ and get
	\begin{align} \label{eq:un-0-profile}
	u_{n,0} (x) = \sum_{j=1}^J e^{-it^j_n\Delta} \psi^j(x-x^j_n) + W^J_n(x)
	\end{align}
	with the following properties: 
	\begin{align} 
	1\leq j \ne k \leq J, \quad \lim_{n\rightarrow \infty} |t^j_n-t^k_n| + |x^j_n - x^k_n| &= \infty, \label{eq:orthogo} \\
	\lim_{J \rightarrow \infty} \left[ \lim_{n\rightarrow \infty} \|e^{it\Delta} W^J_n\|_{S(\dot{H}^{\gamc}) \cap L^\infty_t(\R, L^{\frac{2N}{N-2\gamc}}_x)} \right] &=0, \label{eq:remainder}
	\end{align}
	and for fixed $J$ and $\gamma \in [0,1]$,
	\begin{align} 
	\|u_{n,0}\|^2_{\dot{H}^\gamma} = \sum_{j=1}^J \|\psi^j\|^2_{\dot{H}^\gamma} + \|W^J_n\|^2_{\dot{H}^\gamma} + o_n(1). \label{eq:H-gam-expan}
	\end{align}
	Moreover, we also have the following Pythagorean expansions of the potential and total energies:
	\begin{align} 
	P(u_{n,0}) &= \sum_{j=1}^J P(e^{-it^j_n\Delta} \psi^j(\cdot-x^j_n)) + P(W^J_n) + o_n(1), \label{expan-P} \\
	E(u_{n,0}) &= \sum_{j=1}^J E (e^{-it^j_n \Delta} \psi^j(\cdot-x^j_n)) + E(W^J_n) + o_n(1). \label{expan-E}
	\end{align}
	
	For the proof of the above expansions, we refer to \cite{FG-JDE} (see also \cite{FG-BBMS}). We now define the nonlinear profiles $v^j:I^j \times \R^N \rightarrow \C$ associated to $\psi^j, t^j_n$, and $x^j_n$ as follows:
	\begin{itemize} [leftmargin=5mm]
		\item If $x^j_n \equiv 0$ and $t^j_n \equiv 0$, then $v^j$ is the maximal lifespan solution to \eqref{eq:inls} with initial data $\left. v^j\right|_{t=0} = \psi^j$. 
		\item If $x^j_n \equiv 0$ and $t^j_n \rightarrow -\infty$, then $v^j$ is the maximal lifespan solution to \eqref{eq:inls} that scatters to $e^{it\Delta} \psi^j$ as $t\rightarrow \infty$ (Such a solution exists due to Lemma \ref{lem-non-prof-dive-time}). In particular, $\|v^j\|_{S((0,\infty), \dot{H}^{\gamc})} <\infty$ and $\|v^j(-t^j_n) -e^{-it^j_n\Delta} \psi^j\|_{H^1} \rightarrow 0$ as $n \rightarrow \infty$.
		\item If $x^j_n \equiv 0$ and $t^j_n \rightarrow \infty$, then $v^j$ is the maximal lifespan solution to \eqref{eq:inls} that scatters to $e^{it\Delta} \psi^j$ as $t\rightarrow -\infty$. In particular, $\|v^j\|_{S((-\infty, 0),\dot{H}^{\gamc})} <\infty$ and $\|v^j(-t^j_n) -e^{-it^j_n\Delta} \psi^j\|_{H^1} \rightarrow 0$ as $n\rightarrow \infty$.
		\item If $|x^j_n| \rightarrow \infty$, then we simply take $v^j(t) = e^{it\Delta}\psi^j$. 
	\end{itemize}
	
	For each $j, n\geq 1$, we introduce $v^j_n: I^j_n \times \R^N \rightarrow \C$ defined by 
	\begin{itemize} [leftmargin=5mm]
		\item if $x^j_n\equiv 0$, then $v^j_n(t):= v^j(t-t^j_n)$, where $I^j_n:= \left\{t \in \R \ : \ t -t^j_n \in I^j\right\}$.
		\item if $|x^j_n|\rightarrow \infty$, we define $v^j_n$ a solution to \eqref{eq:inls} with initial data $v^j_n(0,x) = v^j(-t^j_n, x-x^j_n) = e^{-it^j_n\Delta} \psi^j(x-x^j_n)$. It follows from Lemma \ref{lem-non-prof-dive-spac} that for $n$ sufficiently large, $v^j_n$ exists globally in time and scatters in $H^1$ in both directions.   
	\end{itemize}
	
	We have from the definition of $v^j_n$ and the continuity of the linear flow that
	\begin{align} \label{eq:v-jn}
	\|v^j_n(0) - e^{-it^j_n\Delta} \psi^j(\cdot-x^j_n)\|_{H^1} \rightarrow 0 \quad \text{as } n\rightarrow \infty.
	\end{align}
	Thus we rewrite \eqref{eq:un-0-profile} as
	\begin{align} \label{eq:un-0-vj}
	u_{n,0}(x)= \sum_{j=1}^J v^j_n(0, x) + \tilde{W}^J_n(x),
	\end{align}
	where
	\[
	\tilde{W}^J_n(x)= \sum_{j=1}^J e^{-it^j_n\Delta} \psi^j(x-x^j_n) - v^j_n(0, x) + W^J_n(x).
	\]
	By Strichartz estimates, we have
	\[
	\|e^{it\Delta} \tilde{W}^J_n\|_{S(\dot{H}^{\gamc})  \cap L^\infty_t(\R, L^{\frac{2N}{N-2\gamc}}_x)} \lesssim \sum_{j=1}^J \|e^{-it^j_n\Delta} \psi^j(\cdot-x^j_n) - v^j_n(0)\|_{H^1} + \|e^{it\Delta} W^J_n\|_{S(\dot{H}^{\gamc})  \cap L^\infty_t(\R, L^{\frac{2N}{N-2\gamc}}_x)}
	\]
	which, by \eqref{eq:remainder} and \eqref{eq:v-jn}, implies that
	\begin{align} \label{eq:tilde-W-Jn}
	\lim_{J\rightarrow \infty} \left[\lim_{n\rightarrow \infty} \|e^{it\Delta} \tilde{W}^J_n\|_{S(\dot{H}^{\gamc}) \cap L^\infty_t(\R, L^{\frac{2N}{N-2\gamc}}_x)} \right] =0.
	\end{align}
	Using the fact that
	\[
	| \|\nabla f\|^2_{L^2} - \|\nabla g\|^2_{L^2} | \lesssim \|\nabla f - \nabla g\|_{L^2} (\|\nabla f\|_{L^2} + \|\nabla g\|_{L^2})
	\]
	and (see \cite[Lemma 4.3]{FG-BBMS})
	\begin{align}\label{est-P}
	|P(f) - P(g)| \lesssim \|f-g\|_{L^{\alpha+2}} \left(\|f\|^{\alpha+1}_{L^{\alpha+2}} + \|g\|^{\alpha+1}_{L^{\alpha+2}} \right)  + \|f-g\|_{L^r} \left(\|f\|^{\alpha+1}_{L^r} + \|g\|^{\alpha+1}_{L^r} \right)
	\end{align}
	for some $\frac{2N\alpha}{N-b}<r<2^*$, where
	\begin{align} \label{eq:2-star}
	2^*:= \left\{
	\begin{array}{ccl}
	\frac{2N}{N-2} &\text{if}& N\geq 3, \\
	\infty &\text{if} & N=1,2,
	\end{array}
	\right. 
	\end{align}
	we infer from \eqref{expan-E}, Sobolev embedding, and \eqref{eq:v-jn} that
	\begin{align} \label{expan-E-tilde}
	E(u_{n,0}) = \sum_{j=1}^J E(v^j_n(0)) + E(\tilde{W}^J_n) + o_n(1).
	\end{align}
	Next, we show the following Pythagorean expansion along the bounded INLS flow (see \cite[Lemma 3.9]{Guevara} for a similar result in the context of NLS).
	
	\begin{lemma} [Pythagorean expansion along the bounded INLS flow] \label{lem-pytha-expan}
		Let $T \in (0,\infty)$ be a fixed time. Assume that for all $n \geq 1$, $u_n(t):= \inls(t) u_{n,0}$ exists up to time $T$ and satisfies
		\begin{align} \label{eq:boun-un-t}
		\lim_{n\rightarrow \infty} \sup_{t \in [0,T]} \|\nabla u_n(t)\|_{L^2} <\infty,
		\end{align}
		where $\inls(t) f$ denotes the solution to \eqref{eq:inls} with initial data $f$ at time $t=0$. We consider the nonlinear profile \eqref{eq:un-0-vj}. Denote $\tilde{W}^J_n(t):= \inls(t) \tilde{W}^J_n$. Then for all $t\in [0,T]$,
		\begin{align} \label{eq:pytha-expan-grad}
		\|\nabla u_n(t)\|^2_{L^2} = \sum_{j=1}^J \|\nabla v^j_n(t)\|^2_{L^2} + \|\nabla \tilde{W}^J_n(t)\|^2_{L^2} + o_{J,n}(1),
		\end{align} 
		where $o_{J,n}(1) \rightarrow 0$ as $J,n \rightarrow \infty$ uniformly on $0\leq t\leq T$. In particular, we have for all $t\in [0,T]$,
		\begin{align} \label{eq:pytha-expan-P}
		P(u_n(t)) = \sum_{j=1}^J P(v^j_n(t)) + P(\tilde{W}^J_n(t)) + o_{J,n}(1).
		\end{align}
	\end{lemma}
	
	\begin{proof} By \eqref{eq:H-gam-expan}, there exists $J_0$ large enough such that $\|\psi^j\|_{H^1}$ sufficiently small for all $j \geq J_0 +1$. By the triangle inequality using \eqref{eq:v-jn}, we see that for $n$ large, $\|v^j_n(0)\|_{H^1}$ is small which, by the small data theory, implies that $v^j_n$ exists globally in time and scatters in $H^1$ in both directions. Moreover, we can assume that for all $1\leq j \leq J_0$, $x^j_n \equiv 0$ since otherwise, if $|x^j_n|\rightarrow \infty$, then by Lemma \ref{lem-non-prof-dive-spac}, we have for $n$ large, $v^j_n$ exists globally in time and scatters in $H^1$ in both directions. In particular, we have for all $j\geq J_0+1$,
		\begin{align} \label{eq:str-nor-v-jn-large}
		\|v^j_n\|_{S(\dot{H}^{\gamc})} <\infty
		\end{align}
		for $n$ large. We reorder the first $J_0$ profiles and let $0\leq J_2 \leq J_0$ such that
		\begin{itemize}
			\item for any $1\leq j \leq J_2$, the time shifts $t^j_n \equiv 0$ for all $n$. Here $J_2 \equiv 0$ means that there is no $j$ in this case. Note that by the pairwise divergence property \eqref{diver-proper}, we have $J_2 \leq 1$.
			\item for any $J_2 +1 \leq j \leq J_0$, the time shifts $|t^j_n| \rightarrow \infty$ as $n\rightarrow \infty$. Here $J_2 = J_0$ means that  there is no $j$ in this case.
		\end{itemize}
		In the following, we only consider the case $J_2=1$. The one for $J_2=0$ is treated similarly (even simpler). Fix $T\in (0,\infty)$ and assume that $u_n(t) = \inls(t) u_{n,0}$ exists up to time $T$ and satisfies \eqref{eq:boun-un-t}. We observe that for $2 \leq j \leq J_0$,
		\begin{align} \label{eq:obser-v-jn}
		\|v^j_n\|_{S([0,T],\dot{H}^{\gamc})} \rightarrow 0 \text{ as } n\rightarrow \infty.
		\end{align}
		Indeed, if $t^j_n \rightarrow \infty$, then as $\|v^j\|_{S((-\infty,0), \dot{H}^{\gamc})} <\infty$, we have
		\[
		\|v^j_n\|_{S([0,T],\dot{H}^{\gamc})} = \|v^j\|_{S([-t^j_n, T-t^j_n], \dot{H}^{\gamc})} \rightarrow 0
		\]
		as $n\rightarrow \infty$. Note that we do not consider $\left(\infty,\frac{2N}{N-2\gamc}\right)$ as a $\dot{H}^{\gamc}$-admissible pair. A similar argument goes for $t^j_n \rightarrow -\infty$. 
		
		Moreover, for $2 \leq j \leq J_0$, we have for all $2<r \leq 2^*$, 
		\begin{align} \label{eq:Lr-v-jn}
		\|v^j_n\|_{L^\infty_t([0,T], L^r_x)} \rightarrow 0 \text{ as } n \rightarrow \infty.
		\end{align}
		In fact, we have
		\begin{align*}
		\|v^j_n\|_{L^\infty_t([0,T],L^r_x)} &\leq \|e^{i(t-t^j_n)\Delta} \psi^j(\cdot-x^j_n)\|_{L^\infty_t([0,T],L^r_x)} + \|v^j_n - e^{i(t-t^j_n)\Delta} \psi^j(\cdot-x^j_n)\|_{L^\infty_t([0,T],L^r_x)} \\
		&\leq \|e^{i(t-t^j_n)\Delta} \psi^j\|_{L^\infty_t([0,T],L^r_x)} + C \|v^j_n - e^{i(t-t^j_n)\Delta} \psi^j(\cdot-x^j_n)\|_{L^\infty_t([0,T],H^1_x)}.
		\end{align*}
		By the decay of the linear flow, the first term tends to zero as $n$ tends to infinity due to $|t^j_n| \rightarrow \infty$. For the second term, we use the Duhamel formula
		\[
		v^j_n(t) = e^{it\Delta} v^j_n(0) +i \int_0^t e^{i(t-s)\Delta} |x|^{-b} |v^j_n(s)|^\alpha v^j_n(s) ds,
		\]
		Strichartz estimates, and Lemma \ref{lem-non-est} to have
		\begin{align*}
		\|v^j_n\|_{L^\infty_t([0,T],H^1_x)} &+ \|\scal{\nabla} v^j_n\|_{S([0,T],L^2)} \\
		&\lesssim \|v^j_n(0)\|_{H^1} + \|v^j_n\|^{\theta}_{L^\infty_t([0,T],H^1_x)} \|v^j_n\|^{\alpha-\theta}_{S([0,T],\dot{H}^{\gamc})} \|\scal{\nabla} v^j_n\|_{S([0,T],L^2)} \\
		&\lesssim \|e^{-it^j_n\Delta} \psi^j\|_{H^1} +1 + \left(\|v^j_n\|_{L^\infty_t([0,T],H^1_x)} +\|\scal{\nabla} v^j_n\|_{S([0,T],L^2)} \right)^{1+\theta}  \|v^j_n\|^{\alpha-\theta}_{S([0,T], \dot{H}^{\gamc})}.
		\end{align*}
		It follows from \eqref{eq:obser-v-jn} that
		\begin{align} \label{eq:bound-v-jn}
		\|v^j_n\|_{L^\infty_t([0,T],H^1_x)}+\|\scal{\nabla} v^j_n\|_{S([0,T],L^2)} \lesssim 1.
		\end{align}
		Similarly, we have
		\begin{align*}
		\|v^j_n &- e^{i(t-t^j_n)\Delta} \psi^j(\cdot-x^j_n)\|_{L^\infty_t([0,T], H^1_x)} \\
		&\lesssim \|e^{it\Delta} v^j_n(0) - e^{i(t-t^j_n)\Delta} \psi^j(\cdot-x^j_n)\|_{L^\infty_t([0,T],H^1_x)}  \\
		&\mathrel{\phantom{\lesssim}} + \|v^j_n\|^{\theta}_{L^\infty_t([0,T],H^1_x)} \|v^j_n\|^{\alpha-\theta}_{S([0,T],\dot{H}^{\gamc})} \|\scal{\nabla} v^j_n\|_{S([0,T],L^2)} \\
		&\lesssim \|v^j_n(0)-e^{-it^j_n\Delta} \psi^j(\cdot-x^j_n)\|_{H^1} + \|v^j_n\|^{\theta}_{L^\infty_t([0,T],H^1_x)} \|v^j_n\|^{\alpha-\theta}_{S([0,T],\dot{H}^{\gamc})} \|\scal{\nabla} v^j_n\|_{S([0,T],L^2)}
		\end{align*}
		which, by \eqref{eq:v-jn}, \eqref{eq:obser-v-jn}, and \eqref{eq:bound-v-jn}, implies
		\[
		\|v^j_n-e^{i(t-t^j_n)\Delta} \psi^j(\cdot-x^j_n)\|_{L^\infty_t([0,T],H^1_x)}  \rightarrow 0 \text{ as } n\rightarrow \infty.
		\]
		We thus prove \eqref{eq:Lr-v-jn}. 
		
		Denote
		\[
		B:= \max \left\{1, \lim_{n\rightarrow \infty} \sup_{t\in [0,T]}\|\nabla u_n(t)\|_{L^2} \right\} <\infty.
		\]
		and let $T^1$ the maximal forward time such that 
		\[
		\sup_{t\in [0,T^1]} \|\nabla v^1(t)\|_{L^2} \leq 2B.
		\]
		In what follows, we will show that for all $t\in [0, T^1]$, 
		\begin{align} \label{eq:pytha-expan-proof}
		\|\nabla u_n(t)\|^2_{L^2} = \sum_{j=1}^J \|\nabla v^j_n(t)\|^2_{L^2} + \|\nabla \tilde{W}^J_n(t)\|^2_{L^2} + o_{J,n}(1),
		\end{align}
		where $o_{J,n}(1) \rightarrow 0$ as $J,n \rightarrow \infty$ uniformly on $0\leq t\leq T^1$. We see that \eqref{eq:pytha-expan-proof} implies \eqref{eq:pytha-expan-grad} as $T^1 \geq T$. In fact, if $T^1 < T$, then by \eqref{eq:pytha-expan-proof},
		\[
		\sup_{t\in [0,T^1]} \|\nabla v^1(t)\|_{L^2} = \sup_{t\in [0,T^1]} \|\nabla v^1_n(t)\|_{L^2} \leq \sup_{t\in [0,T^1]} \|\nabla u_n(t)\|_{L^2} \leq \sup_{t\in [0,T]} \|\nabla u_n(t)\|_{L^2} \leq B.
		\]
		Note that $t^1_n \equiv 0$. By the continuity, it contradicts the maximality of $T^1$.
		
		We estimate $\|v^1_n\|_{S([0,T^1],\dot{H}^{\gamc})}$ as follows. For $N\geq 3$, by interpolation between endpoints and Sobolev embedding, we have
		\begin{align*}
		\|v^1_n\|_{S([0,T^1], \dot{H}^{\gamc})} &= \|v^1\|_{S([0,T^1],\dot{H}^{\gamc})} \\
		&\lesssim \|v^1\|_{L^{\frac{2}{1-\gamc}}_t([0,T^1], L^{\frac{2N}{N-2}}_x)} + \|v^1\|_{L^\infty_t([0,T^1], L^{\frac{2N}{N-2\gamc}}_x)} \\
		&\lesssim \|v^1\|_{L^{\frac{2}{1-\gamc}}_t([0,T^1], L^{\frac{2N}{N-2}}_x)} +  \|v^1\|^{1-\gamc}_{L^\infty_t([0,T^1],L^2_x)} \|v^1\|^{\gamc}_{L^\infty_t([0,T^1], L^{\frac{2N}{N-2}}_x)}\\
		&\lesssim  (T^1)^{\frac{1-\gamc}{2}}  \|\nabla v^1\|_{L^\infty_t([0,T^1], L^2_x)} + C \|\nabla v^1\|^{\gamc}_{L^\infty_t([0,T^1], L^{\frac{2N}{N-2}}_x)} \\
		&\lesssim  (T^1)^{\frac{1-\gamc}{2}} B + C B^{\gamc}.
		\end{align*}
		Here we have use the conservation of mass and the choice of $v^1$ to have that for all $t\in [0,T^1]$,
		\[
		\|v^1(t)\|_{L^2} = \lim_{n\rightarrow \infty} \|v^1(-t^1_n)\|_{L^2} = \lim_{n\rightarrow \infty} \|e^{-it^1_n\Delta} \psi^1\|_{L^2}=\|\psi^1\|_{L^2} \leq \|u_{n,0}\|_{L^2} \leq 1.
		\]
		When $N=2$, a similar estimate holds by interpolating between $\left(\infty, \frac{2}{1-\gamc}\right)$ and $\left(\frac{2}{1-\gamc}, r\right)$ with $r$ sufficiently large and using Sobolev embedding. This shows that 
		\begin{align} \label{eq:v-jn-small}
		\|v^1_n\|_{S([0,T^1], \dot{H}^{\gamc})} \leq C(T^1,B).	
		\end{align}
		
		Now we define the approximation
		\[
		\tilde{u}^J_n(t,x):= \sum_{j=1}^J v^j_n(t, x).
		\]
		We have 
		\[
		u_{n,0}(x) - \tilde{u}^J_n(0,x) = \tilde{W}^J_n(x).
		\]
		By \eqref{eq:tilde-W-Jn}, we have
		\begin{align} \label{ini-cond-sta}
		\lim_{J \rightarrow \infty} \left[\lim_{n\rightarrow \infty} \|e^{it\Delta} (u_{n,0} - \tilde{u}^J_n(0))\|_{S(\dot{H}^{\gamc})\cap L^\infty_t(\R, L^{\frac{2N}{N-2\gamc}}_x) }\right] =0.
		\end{align}
		We also have
		\[
		i\partial_t \tilde{u}^J_n + \Delta \tilde{u}^J_n + |x|^{-b} \left|\tilde{u}^J_n\right|^\alpha \tilde{u}^J_n = \tilde{e}^J_n,
		\]
		where
		\[
		\tilde{e}^J_n = \sum_{j=1}^J F(v^j_n) - F \left( \sum_{j=1}^J v^j_n\right)
		\]
		with $F(u):= |x|^{-b}|u|^\alpha u$. We also have the following properties of the approximate solutions.
		\begin{lemma} \label{lem-appro-solu}
			The functions $\tilde{u}^J_n$ and $\tilde{e}^J_n$ satisfy
			\begin{align} \label{prop-tilde-u-Jn-1}
			\limsup_{n\rightarrow \infty} \left( \|\tilde{u}^J_n\|_{L^\infty_t([0,T^1], H^1_x)} + \|\tilde{u}^J_n\|_{S([0,T^1],\dot{H}^{\gamc})} \right) \lesssim 1
			\end{align}
			uniformly in $J$ and
			\begin{align} \label{prop-tilde-u-Jn-2}
			\lim_{J\rightarrow \infty} \lim_{n\rightarrow \infty} \|\scal{\nabla}\tilde{e}^J_n\|_{S'([0,T^1], L^2)} + \|\tilde{e}^J_n\|_{S'([0,T^1],\dot{H}^{-\gamc})} =0.
			\end{align}
		\end{lemma} 
		\begin{proof}
			The boundedness of $\|\tilde{u}^J_n\|_{S([0,T^1],\dot{H}^{\gamc})}$ follows from \eqref{eq:str-nor-v-jn-large}, \eqref{eq:obser-v-jn}, and \eqref{eq:v-jn-small}. The boundedness of $\|\tilde{u}^J_n\|_{L^\infty_t([0,T^1],L^2_x)}$ follows from \eqref{eq:H-gam-expan} and the fact that
			\[
			\|v^j_n(t)\|_{L^2} = \|v^j(t-t^j_n)\|_{L^2} = \lim_{n\rightarrow \infty} \|v(-t^j_n)\|_{L^2} = \lim_{n\rightarrow \infty} \|e^{-it^j_n\Delta} \psi^j\|_{L^2} = \|\psi^j\|_{L^2}. 
			\]
			To see the boundedness of $\|\nabla \tilde{u}^J_n\|_{L^\infty_t([0,T^1],L^2_x)}$, we proceed as follows. For $j\geq J_0$, by \eqref{eq:str-nor-v-jn-large}, we split $[0,T^1]$ into finite subintervals $I_k, k=1, \cdots, M$ such that $\|v^j_n\|_{S(I_k, \dot{H}^{\gamc})}$ is small. By Duhamel's formula, Strichartz estimates, and Lemma \ref{lem-non-est}, we have
			\[
			\|\nabla v^j_n\|_{L^\infty_t (I_k, L^2_x)} \lesssim \|\nabla v^j_n(t_k)\|_{L^2}, \quad I_k = [t_k, t_{k+1}], \quad k=1, \cdots, M.
			\]
			Summing over these finite intervals, we get
			\[
			\|\nabla v^j_n\|_{L^\infty_t([0,T^1], L^2_x)} \lesssim \|\nabla v^j_n(0)\|_{L^2}.
			\]
			For $2 \leq j \leq J_0$, we have from the Duhamel formula, Strichartz estimates, Lemma \ref{lem-non-est}, and \eqref{eq:obser-v-jn}, we have
			\[
			\|\nabla v^j_n\|_{L^\infty_t([0,T^1], L^2_x)} \lesssim \|\nabla v^j_n(0)\|_{L^2}
			\]
			for $n$ sufficiently large. Thus we have
			\begin{align*}
			\|\nabla \tilde{u}^J_n\|^2_{L^\infty_t([0,T^1], L^2_x)} &\leq \|\nabla v^1\|^2_{L^\infty_t([0,T^1],L^2_x)} + \sum_{j=2}^{J} \|\nabla v^j_n\|^2_{L^\infty_t([0,T^1],L^2_x)}  \\
			&\lesssim B^2 + \sum_{j=2}^J \|\nabla v^j_n(0)\|^2_{L^2} \\
			&\lesssim B^2 + \sum_{j=2}^J  \|\nabla \psi^j\|^2_{L^2}  + o_{n}(1) \\
			&\lesssim B^2 + \|\nabla u_{n,0}\|^2_{L^2} + o_{n}(1) \\
			&\lesssim B^2 + o_{n}(1).
			\end{align*}
			This shows the boundedness of $\|\nabla \tilde{u}^J_n\|_{L^\infty_t([0,T^1], L^2_x)}$ and we prove \eqref{prop-tilde-u-Jn-1}. To see \eqref{prop-tilde-u-Jn-2}, we follow from the same argument as in \cite[Claim 1 (6.23)]{FG-BBMS}. We thus omit the details.
		\end{proof}
		
		Thanks to \eqref{ini-cond-sta} and Lemma \ref{lem-appro-solu}, the stability given in Lemma \ref{lem-sta} (see also Remark \ref{rem-sta}) implies 
		\[
		\lim_{J\rightarrow \infty} \left[\lim_{n\rightarrow \infty} \|u_n-\tilde{u}^J_n\|_{S([0,T^1], \dot{H}^{\gamc})\cap L^\infty_t([0,T^1], L^{\frac{2N}{N-2\gamc}}_x)}\right]=0.
		\]
		By interpolating between endpoints and using Sobolev embedding, we infer that
		\[
		\|u_n-\tilde{u}^J_n\|_{L^\infty_t([0,T^1],L^{\alpha+2}_x) \cap L^\infty_t([0,T^1], L^r_x)} \lesssim \|u_n- \tilde{u}^J_n\|_{L^\infty_t([0,T^1], L^{\frac{2N}{N-2\gamc}}_x)} \|\scal{\nabla}(u_n- \tilde{u}^J_n)\|_{L^\infty_t([0,T^1], L^2_x)} \rightarrow 0 
		\]
		as $J, n\rightarrow \infty$, where $r$ is an exponent satisfying $\frac{2N}{N-2\gamc}<\frac{N(\alpha+2)}{N-b}<r<2^*$. This estimate together with \eqref{est-P} yield
		\begin{align} \label{proof-1}
		|P(u_n(t)) - P(\tilde{u}^J_n(t))| \rightarrow 0
		\end{align}
		as $J, n\rightarrow \infty$ uniformly on $0 \leq t\leq T^1$. On the other hand, we have from the same argument as in \cite[Proposition 5.3]{FG-BBMS} using \eqref{eq:Lr-v-jn} that for all $t\in [0,T^1]$,
		\begin{align} \label{proof-2}
		P(\tilde{u}^J_n(t)) = \sum_{j=1}^J P(v^j_n(t)) +o_{J,n}(1) = \sum_{j=1}^J P(v^j_n(t)) + P(\tilde{W}^J_n(t))+o_{J,n}(1).
		\end{align}
		Here we have used the fact that $P(\tilde{W}^J_n(t)) = o_{J,n}(1)$ uniformly on $0\leq t \leq T^1$. In fact, by the Duhamel formula and Lemma \ref{lem-non-est}, we have
		\[
		\|\tilde{W}^J_n(t)\|_{S(\dot{H}^{\gamc})} \leq \|e^{it\Delta} \tilde{W}^J_n\|_{S(\dot{H}^{\gamc})} + C \|\tilde{W}^J_n(t)\|^\theta_{L^\infty_t(\R, H^1_x)} \|\tilde{W}^J_n(t)\|^{\alpha+1-\theta}_{S(\dot{H}^{\gamc})}
		\]
		for some $\theta>0$ sufficiently small. Since $\|\tilde{W}^J_n(t)\|_{L^\infty_t(\R,H^1_x)} \lesssim 1$ (by the small data theory), the continuity argument together with \eqref{eq:remainder} imply
		\begin{align} \label{eq:tilde-W-Jn-t}
		\lim_{J\rightarrow \infty} \left[ \lim_{n\rightarrow \infty} \|\tilde{W}^J_n(t)\|_{S(\dot{H}^{\gamc})}\right] =0.
		\end{align}
		Thanks to \eqref{eq:tilde-W-Jn-t}, Strichartz estimates, and \eqref{eq:tilde-W-Jn}, we have
		\[
		\lim_{J \rightarrow \infty}\left[\lim_{n\rightarrow \infty} \|\tilde{W}^J_n(t)\|_{L^\infty_t(\R, L^{\frac{2N}{N-2\gamc}}_x)}\right] =0
		\] 
		which together with \eqref{est-P} yield
		\[
		\lim_{J\rightarrow \infty} \left[\lim_{n\rightarrow \infty} \sup_{t\in \R} P(\tilde{W}^J_n(t))\right] =0.
		\]
		Moreover, by the conservation of energy, we have
		\begin{align}
		E(u_n(t)) = E(u_{n,0}) &= \sum_{j=1}^J E(v^j_n(0)) + E(\tilde{W}^J_n) + o_n(1) \nonumber \\
		&= \sum_{j=1}^J E(v^j_n(t)) + E(\tilde{W}^J_n(t)) + o_{J,n}(1). \label{proof-3}
		\end{align}
		Collecting \eqref{proof-1}, \eqref{proof-2}, and \eqref{proof-3}, we prove \eqref{eq:pytha-expan-proof}. The proof is complete. 
	\end{proof}
	We come back to the proof of Proposition \ref{prop-scat}. We will consider two cases.
	
	\vspace{3mm}
	
	\noindent {\bf Case 1. More than one non-zero profiles.} We have 
	\[
	M(v^j_n(t)) = M(v^j_n(0)) = M(e^{-it^j_n\Delta} \psi^j) = M(\psi^j) <1, \quad \forall j \geq 1.
	\]
	By \eqref{eq:cond-un} and \eqref{eq:pytha-expan-P}, we have
	\[
	\sup_{t\in [0,\infty)} P(v^j_n(t)) [M(v^j_n(t))]^{\sigc} <A, \quad \forall j \geq 1.
	\]
	Here we note that by \eqref{eq:pytha-expan-grad}, $\|\nabla v^j_n(t)\|_{L^2}$ is bounded uniformly which implies $v^j_n$ exists globally in time. By Lemma \ref{lem-coer}, we have $E(v^j_n(t))\geq 0$, hence
	\[
	E(v^j_n(t)) [M(v^j_n(t))]^{\sigc} <\deltc, \quad \forall j \geq 1.
	\]
	By Item (1) (see after \eqref{eq:deltc}), we have
	\[
	\|v^j_n\|_{S([0,\infty), \dot{H}^{\gamc})} <\infty, \quad \forall j \geq 1.
	\]
	We can approximate $u_n$ by
	\[
	u^J_n(t,x) := \sum_{j=1}^J v^j_n(t)
	\]
	and get for $J$ sufficiently large that
	\[
	\|u_n\|_{S([0,\infty),\dot{H}^{\gamc})} <\infty
	\]
	which is a contradiction.
	
	\vspace{3mm}
	
	\noindent {\bf Case 2. Only one non-zero profile.} We must have only one non-zero profile, i.e.,
	\[
	u_{n,0}(x) = e^{-it^1_n\Delta} \psi^1(x-x^1_n) + W_n(x), \quad \lim_{n\rightarrow \infty} \|e^{it\Delta} W_n\|_{S([0,\infty),\dot{H}^{\gamc})} =0.
	\]
	We note that $t^1_n$ cannot tend to $-\infty$. Indeed, if $t^1_n\rightarrow -\infty$, then we have
	\[
	\|e^{it\Delta} u_{n,0}\|_{S([0,\infty),\dot{H}^{\gamc})} \leq \|e^{it\Delta} \psi^1\|_{S([-t^1_n, \infty), \dot{H}^{\gamc})} + \|e^{it\Delta} W_n\|_{S([0,\infty), \dot{H}^{\gamc})} \rightarrow 0
	\]
	as $n\rightarrow \infty$. By the Duhamel formula, Lemma \ref{lem-non-est}, and the continuity argument, $\|u_n\|_{S([0,\infty),\dot{H}^{\gamc})} <\infty$ for $n$ sufficiently large which is a contradiction.
	
	We claim that $x^1_n \equiv 0$. Otherwise, if $|x^1_n| \rightarrow \infty$, then, by Lemma \ref{lem-non-prof-dive-spac}, for $n$ large, there exist global solutions $v_n$ to \eqref{eq:inls} satisfying $v_n(0,x) = e^{-it^1_n\Delta} \psi^1(x-x^1_n)$. Moreover, $v_n$ scatters in $H^1$ in both directions. In particular, $\|v_n\|_{S(\dot{H}^{\gamc})} <\infty$. Again, by the long time perturbation, we show that $\|u_n\|_{S([0,\infty),\dot{H}^{\gamc})} <\infty$ for $n$ sufficiently large which is a contradiction. 
	
	Let $v^1$ be the nonlinear profile associated to $\psi^1$ and $t^1_n$, we have
	\[
	u_{n,0}(x) = v^1(-t^1_n,x) + \tilde{W}_n(x).
	\]
	Set $v^1_n(t) = v^1(t-t^1_n)$. Arguing as above, we have
	\[
	M(v^1_n(t)) \leq 1, \quad \sup_{t\in [0,\infty)} P(v^1_n(t)) \leq A, \quad E(v^1_n(t))\leq \deltc, \quad \lim_{n\rightarrow \infty} \|\tilde{W}_n(t)\|_{S(\dot{H}^{\gamc})}=0.
	\]
	We infer that $M(v^1_n(t))=1$ and $E(v^1_n(t)) =\deltc$. Otherwise, if $M(v^1_n(t))<1$, then
	\[
	\sup_{t\in [0,\infty)} P(v^1_n(t)) [M(v^1_n(t))]^{\sigc}<A, \quad E(v^1_n) [M(v^1_n)]^{\sigc} <\deltc.
	\]
	By Item (1) (see again after \eqref{eq:deltc}), we have $\|v^1_n\|_{S([0,\infty),\dot{H}^{\gamc})}<\infty$. Thus we get a contradiction by the long time perturbation argument.
	
	Now we define $u_{\ct}$ the solution to \eqref{eq:inls} with initial data $\left. u_{\ct}\right|_{t=0} = v^1(0)$. We have
	\begin{align*}
	M(u_{\ct}) &= M(v^1(0)) = M(v^1(t-t^1_n)) = M(v^1_n(t)) =1, \\
	E(u_{\ct}) &= E(v^1(0)) = E(v^1(t-t^1_n)) = E(v^1_n(t)) = \deltc.
	\end{align*}
	Moreover, 
	\[
	\sup_{t\in [0,\infty)} P(u_{\ct}(t)) = \sup_{t \in [0,\infty)} P(v^1(t)) = \sup_{t\in [t^1_n,\infty)} P(v^1(t-t^1_n)) = \sup_{t\in [t^1_n,\infty)} P(v^1_n(t)) \leq A.
	\]
	By the definition of $\deltc$, we must have $\|u_{\ct}\|_{S([0,\infty), \dot{H}^{\gamc})} =\infty$. This shows \eqref{eq:uc}. 
	
	By the same argument as in the proof of \cite[Proposition 6.3]{FG-BBMS}, we show that the set
	\[
	\mathcal{K}:= \left\{u_{\ct}(t) \ : \ t \in [0,\infty)\right\}
	\]
	is precompact in $H^1$. 
	
	\vspace{3mm}
	
	\noindent {\bf Step 3. Exclusion of the critical solution.} Thanks to the above compactness result, the standard rigidity argument using localized virial estimates and Lemma \ref{lem-coer} shows that $u_{\ct} \equiv 0$ which contradicts \eqref{eq:uc}. We refer the reader to \cite[Section 7]{FG-BBMS} for more details. The proof of Proposition \ref{prop-scat} is now complete. This also ends the proof of Theorem \ref{theo-scat-crite}.
	\hfill $\Box$

	\section{Blow-up criterion}
	\label{sec:blow}
	\setcounter{equation}{0}
	
	In this section, we give the proof of the blow-up criterion given in Theorem \ref{theo-blow-crite}. Let us recall the following virial identity (see e.g., \cite{Dinh-NA}).
	
	\begin{lemma} \label{lem-viri-iden}
		Let $\varphi: \R^N\rightarrow \R$ be a sufficiently smooth and decaying function. Let $u$ be a solution to \eqref{eq:inls} defined on the maximal forward time interval of existence $[0,T^*)$. Define
		\begin{align} \label{defi-V-varphi}
		V_{\varphi}(t):= \int \varphi(x) |u(t,x)|^2 dx.
		\end{align}
		Then we have for all $t\in [0,T^*)$,
		\[
		V'_{\varphi}(t) = 2 \ima \int \nabla \varphi(x) \cdot \nabla u(t,x)  \overline{u}(t,x) dx 
		\]
		and
		\begin{align*}
		V''_{\varphi}(t) &= -\int \Delta^2 \varphi(x) |u(t,x)|^2 dx + 4 \sum_{j,k=1}^N \rea \int \partial^2_{jk} \varphi (x) \partial_j \overline{u}(t,x) \partial_k u(t,x) dx \\
		&\mathrel{\phantom{=}} - \frac{2\alpha}{\alpha+2} \int |x|^{-b} \Delta \varphi(x) |u(t,x)|^{\alpha+2} dx +\frac{4}{\alpha+2} \int \nabla \varphi(x) \cdot \nabla(|x|^{-b}) |u(t,x)|^{\alpha+2} dx.
		\end{align*}
	\end{lemma}
	
	\begin{remark} \label{rem-viri}
		(1) In the case $\varphi(x) = |x|^2$, we have
		\[
		\frac{d^2}{dt^2} \|xu(t)\|^2_{L^2} = 8 G(u(t)), 
		\]
		where $G(f)$ is as in \eqref{eq:G}.
		
		(2) In the case $\varphi$ is radially symmetric, it follows from 
		\[
		\partial_j = \frac{x_j}{r} \partial_r, \quad \partial^2_{jk} = \left( \frac{\delta_{jk}}{r} - \frac{x_j x_k}{r^3} \right) \partial_r + \frac{x_j x_k}{r^2} \partial^2_r
		\]
		that
		\begin{align*}
		\sum_{j,k=1}^N \rea &\int \partial^2_{jk} \varphi(x) \partial_j \overline{u}(t,x) \partial_k u(t,x) dx \\
		&= \int \frac{\varphi'(r)}{r} |\nabla u(t,x)|^2 dx + \int \left(\frac{\varphi''(r)}{r^2} - \frac{\varphi'(r)}{r^3} \right) |x \cdot \nabla u(t,x)|^2 dx.
		\end{align*}
		In particular, we have
		\begin{align} \label{eq:viri-rad}
		\begin{aligned}
		&V''_{\varphi}(t) \\
		&= - \int \Delta^2 \varphi(x) |u(t,x)|^2 dx + 4 \int \frac{\varphi'(r)}{r} |\nabla u(t,x)|^2 dx + 4 \int \left(\frac{\varphi''(r)}{r^2}-\frac{\varphi'(r)}{r^3}\right) |x \cdot \nabla u(t,x)|^2 dx\\
		&\mathrel{\phantom{=}}  - \frac{2\alpha}{\alpha+2} \int |x|^{-b} \Delta \varphi(x) |u(t,x)|^{\alpha+2} - \frac{4b}{\alpha+2} \int |x|^{-b} \frac{\varphi'(r)}{r} |u(t,x)|^{\alpha+2} dx.
		\end{aligned}
		\end{align}
		(3) Denote $x=(y,x_N)$ with $y=(x_1, \cdots, x_{N-1}) \in \R^{N-1}$ and $x_N \in \R$. Let $\psi:\R^{N-1} \rightarrow \R$ be a sufficiently smooth decaying function. Set $\varphi(x) = \varphi(y,x_N) = \psi(y)+x_N^2$. We have
		\begin{align*}
		V'_\varphi(t) = 2 \ima \int \left(\nabla_y \psi(y) \cdot \nabla_y u(t,x) + 2x_N \partial_N u(t,x)\right) \overline{u}(t,x)  dx
		\end{align*}
		and
		\begin{align*}
		V''_\varphi(t) &= -\int \Delta^2_y \psi(y) |u(t,x)|^2 dx + 4 \sum_{j,k=1}^{N-1} \rea \int \partial^2_{jk}\psi(y) \partial_j \overline{u}(t,x) \partial_k u(t,x) dx \\
		&\mathrel{\phantom{=}} - \frac{2\alpha}{\alpha+2} \int |x|^{-b} \Delta_y \psi(y) |u(t,x)|^{\alpha+2} dx -\frac{4b}{\alpha+2} \int \nabla_y \psi(y) \cdot y |x|^{-b-2} |u(t,x)|^{\alpha+2} dx \\
		&\mathrel{\phantom{=}} + 8 \|\partial_N u(t)\|^2_{L^2} -\frac{4\alpha}{\alpha+2} \int |x|^{-b} |u(t,x)|^{\alpha+2} dx - \frac{8b}{\alpha+2} \int x_N^2|x|^{-b-2} |u(t,x)|^{\alpha+2}dx.
		\end{align*}
	\end{remark}
	
	Let $\chi$ be a smooth radial function satisfying
	\begin{align*} 
	\chi(x) = \chi(r) = \left\{
	\begin{array}{ccc}
	r^2 &\text{if}& r\leq 1, \\
	0 &\text{if} & r\geq 2, 
	\end{array}
	\right.
	\quad 
	\chi''(r) \leq 2 \quad \forall r = |x| \geq 0.
	\end{align*}
	Given $R>1$, we define the radial function
	\begin{align} \label{eq:varphi-R}
	\varphi_R(x):= R^2 \chi(x/R).
	\end{align}
	We have the following localized virial estimate.
	\begin{proposition} \label{prop-viri-est}
		Let $N\geq 1$, $0<b<\min \{2,N\}$, and $\frac{4-2b}{N} <\alpha <\alpha(N)$.  Let $u$ be a solution to \eqref{eq:inls} defined on the maximal forward time interval of existence $[0,T^*)$. Let $\varphi_R$ be as in \eqref{eq:varphi-R} and define $V_{\varphi_R}(t)$ as in \eqref{defi-V-varphi}. Then we have for all $t\in [0,T^*)$,
		\[
		V'_{\varphi_R} (t) = 2\ima \int  \nabla \varphi_R(x) \cdot \nabla u(t,x) \overline{u}(t,x) dx
		\]
		and 
		\begin{align*}
		V''_{\varphi_R}(t) \leq 8 G(u(t)) + C R^{-2}+ C R^{-b} \|u(t)\|^{\alpha+2}_{H^1},
		\end{align*}
		where $G$ is as in \eqref{eq:G} and some constant $C>0$ independent of $R$. 
	\end{proposition}
	
	\begin{proof}
		It follows from \eqref{eq:viri-rad} that
		\begin{align*}
		V''_{\varphi_R}(t) &= 8 G(u(t)) - 8\|\nabla u(t)\|^2_{L^2} + \frac{4(N\alpha+2b)}{\alpha+2} \int |x|^{-b} |u(t,x)|^{\alpha+2} dx \\
		&\mathrel{\phantom{=}} - \int \Delta^2 \varphi_R(x) |u(t,x)|^2 dx + 4 \int \frac{\varphi'_R(r)}{r} |\nabla u(t,x)|^2 dx \\
		&\mathrel{\phantom{=}} + 4 \int \left( \frac{\varphi''_R(r)}{r^2} - \frac{\varphi'_R(r)}{r^3}\right) |x\cdot \nabla u(t,x)|^2 dx \\ 
		&\mathrel{\phantom{=}} -\frac{2\alpha}{\alpha+2} \int |x|^{-b} \Delta \varphi_R(x) |u(t,x)|^{\alpha+2} dx - \frac{4b}{\alpha+2} \int |x|^{-b} \frac{\varphi'_R(r)}{r} |u(t,x)|^{\alpha+2} dx.
		\end{align*}
		As $\|\Delta^2\varphi_R\|_{L^\infty} \lesssim R^{-2}$, the conservation of mass implies that
		\[
		\left|\int \Delta^2 \varphi_R(x) |u(t,x)|^2 dx \right| \lesssim R^{-2} \|u(t)\|^2_{L^2} \lesssim R^{-2}.
		\]
		By the Cauchy-Schwarz inequality $|x\cdot \nabla u| \leq |x| |\nabla u| = r |\nabla u|$ and the fact $\varphi''_R(r) \leq 2$, we see that
		\begin{align*}
		4 \int &\frac{\varphi'_R(r)}{r} |\nabla u(t,x)|^2 dx + 4 \int \left( \frac{\varphi''_R(r)}{r^2} - \frac{\varphi'_R(r)}{r^3}\right) |x\cdot \nabla u(t,x)|^2 dx - 8 \|\nabla u(t)\|^2_{L^2} \\
		&\leq  4\int \left(\frac{\varphi'_R(r)}{r} -2 \right) |\nabla u(t,x)|^2 dx + 4 \int \frac{1}{r^2}\left( 2 - \frac{\varphi'_R(r)}{r}\right) |x\cdot \nabla u(t,x)|^2 dx \leq 0.
		\end{align*}
		Moreover,
		\begin{align*}
		&\frac{4(N\alpha+2b)}{\alpha+2} \int |x|^{-b} |u(t,x)|^{\alpha+2} dx ~- \frac{2\alpha}{\alpha+2} \int |x|^{-b} \Delta \varphi_R(x) |u(t,x)|^{\alpha+2} dx \\
		&\mathrel{\phantom{\frac{4(N\alpha+2b)}{\alpha+2} \int |x|^{-b} |u(t,x)|^{\alpha+2} dx}}- \frac{4b}{\alpha+2} \int |x|^{-b} \frac{\varphi'_R(r)}{r} |u(t,x)|^{\alpha+2} dx  \\
		&= \frac{2\alpha}{\alpha+2} \int |x|^{-b} (2N-\Delta\varphi_R(x)) |u(t,x)|^{\alpha+2} dx + \frac{4b}{\alpha+2} \int |x|^{-b} \left(2 -\frac{\varphi'_R(r)}{r}\right) |u(t,x)|^{\alpha+2} dx.
		\end{align*}
		Since $\Delta \varphi_R \leq 2N$, $\frac{\varphi'_R(r)}{r} \leq 2$, $\Delta \varphi_R(x) =2N$, and $\frac{\varphi'_R(r)}{r} =2$ for $r=|x| \leq R$, the above quantity is bounded by 
		\[
		C\int_{|x|\geq R} |x|^{-b} |u(t)|^{\alpha+2} dx \leq C R^{-b} \|u(t)\|^{\alpha+2}_{L^{\alpha+2}} \leq C R^{-b} \|u(t)\|^{\alpha+2}_{H^1},
		\]
		where the last inequality follows from the Sobolev embedding as $\alpha <\alpha(N)$. Collecting the above estimates, we end the proof.
	\end{proof}
	
	\noindent {\it Proof of Theorem \ref{theo-blow-crite}.}
	Let $u:[0,T^*) \times \R^N \rightarrow \C$ be a solution to \eqref{eq:inls} satisfying \eqref{blow-crite}. If $T^*<\infty$, then we are done. If $T^*=\infty$, then we show that there exists $t_n \rightarrow \infty$ such that $\|\nabla u(t_n)\|_{L^2} \rightarrow \infty$ as $n\rightarrow \infty$. Assume by contradiction that  it does not hold, i.e., $\sup_{t\in [0,\infty)} \|\nabla u(t)\|_{L^2} \leq C_0$	for some $C_0>0$. By the conservation of mass, we have
	\begin{align}  \label{eq:bound-solu}
	\sup_{t\in [0,\infty)} \|u(t)\|_{H^1} \leq C_1
	\end{align}
	for some $C_1>0$.
	
	By Proposition \ref{prop-viri-est}, \eqref{blow-crite}, and \eqref{eq:bound-solu}, we have for all $t \in [0,\infty)$,
	\begin{align*}
	V''_{\varphi_R}(t) \leq 8 G(u(t)) + CR^{-2} + C R^{-b} \|u(t)\|^{\alpha+2}_{L^2} \leq -8\delta + CR^{-2} + C R^{-b} C_1^{\alpha+2}.
	\end{align*}
	By taking $R>1$ sufficiently large, we have for all $t\in [0,\infty)$,
	\[
	V''_{\varphi_R}(t) \leq -4\delta. 
	\]
	Integrating this estimate, there exists $t_0>0$ sufficiently large such that $V_{\varphi_R}(t_0)<0$ which is impossible. This finishes the first part of Theorem \ref{theo-blow-crite}.
	
	If we assume in addition that $u$ has finite variance, i.e., $u(t) \in L^2(|x|^2 dx)$ for all $t\in [0,T^*)$, then we have $T^*<\infty$. In fact, it follows from Remark \ref{rem-viri} and \eqref{blow-crite} that
	\[
	\frac{d^2}{dt^2} \|xu(t)\|^2_{L^2} =8 G(u(t)) \leq -8\delta
	\]
	for all $t\in [0,T^*)$. The convexity argument of Glassey \cite{Glassey} implies $T^*<\infty$. 
	\hfill $\Box$

	\section{Long time dynamics}
	\label{sec:dynamic}
	\setcounter{equation}{0}
	
	In this section, we give the proofs of long time dynamics of $H^1$-solutions given in Theorems \ref{theo-dyna-below}, \ref{theo-dyna-at} and \ref{theo-dyna-above}. 
	
	\noindent {\it Proof of Theorem \ref{theo-dyna-below}.} We will consider separately two cases. 
	
	\vspace{3mm}
	
	\noindent {\bf Case 1. Global existence and energy scattering.} Let $u_0 \in H^1$ satisfy \eqref{eq:ener-below} and \eqref{eq:grad-glob-below}. Let us prove \eqref{eq:est-solu-glob-below}. To see this, we first claim that there exists $\rho=\rho(u_0,Q)>0$ such that
	\begin{align} \label{eq:claim-below-1}
	\|\nabla u(t)\|_{L^2} \|u(t)\|^{\sigc}_{L^2} \leq (1-\rho) \|\nabla Q\|_{L^2} \|Q\|^{\sigc}_{L^2}
	\end{align}
	for all $t\in (-T_*,T^*)$. We assume \eqref{eq:claim-below-1} for the moment and prove \eqref{eq:est-solu-glob-below}. By \eqref{eq:GN-ineq} and \eqref{eq:claim-below-1}, we have
	\begin{align*}
	P(u(t)) [M(u(t))]^{\sigc} &\leq C_{\opt} \|\nabla u(t)\|^{\frac{N\alpha+2b}{2}}_{L^2} \|u(t)\|^{\frac{4-2b-(N-2)\alpha}{2} + 2\sigc}_{L^2} \\
	&= C_{\opt} \left(\|\nabla u(t)\|_{L^2}\|u(t)\|^{\sigc}_{L^2} \right)^{\frac{N\alpha+2b}{2}} \\
	&\leq C_{\opt} (1-\rho)^{\frac{N\alpha+2b}{2}} \left(\|\nabla Q\|_{L^2} \|Q\|^{\sigc}_{L^2} \right)^{\frac{N\alpha+2b}{2}}
	\end{align*}
	for all $t\in (-T_*,T^*)$. By \eqref{eq:opt-cons} and \eqref{eq:poho-iden}, we get
	\[
	P(u(t)) [M(u(t))]^{\sigc} \leq \frac{2(\alpha+2)}{N\alpha+2b} (1-\rho)^{\frac{N\alpha+2b}{2}} \left(\|\nabla Q\|_{L^2} \|Q\|^{\sigc}_{L^2}\right)^2 = (1-\rho)^{\frac{N\alpha+2b}{2}} P(Q) [M(Q)]^{\sigc}
	\]
	for all $t\in (-T_*,T^*)$ which shows \eqref{eq:est-solu-glob-below}. By Theorem \ref{theo-scat-crite}, the solution exists globally in time. Moreover, if $N\geq 2$ and $0<b<\min \left\{2,\frac{N}{2}\right\}$, then the solution scatters in $H^1$ in both directions. 
	
	Let us now prove the claim \eqref{eq:claim-below-1}. By the definition of energy and \eqref{eq:GN-ineq}, we have
	\begin{align}
	E(u(t)) [M(u(t))]^{\sigc} &\geq \frac{1}{2} \left( \|\nabla u(t)\|_{L^2} \|u(t)\|^{\sigc}_{L^2}\right)^2 - \frac{C_{\opt}}{\alpha+2} \|\nabla u(t)\|^{\frac{N\alpha+2b}{2}}_{L^2} \|u(t)\|^{\frac{4-2b-(N-2)\alpha}{2}+2\sigc}_{L^2} \nonumber \\
	&= F \left( \|\nabla u(t)\|_{L^2} \|u(t)\|^{\sigc}_{L^2} \right), \label{eq:F}
	\end{align}
	where 
	\[
	F(\lambda):= \frac{1}{2} \lambda^2 - \frac{C_{\opt}}{\alpha+2} \lambda^{\frac{N\alpha+2b}{2}}.
	\]
	Using \eqref{eq:poho-iden}, \eqref{eq:opt-cons} and \eqref{eq:prop-Q}, we see that
	\[
	F \left(\|\nabla Q\|_{L^2} \|Q\|^{\sigc}_{L^2} \right) = \frac{N\alpha-4+2b}{2(N\alpha+2b)} \left(\|\nabla Q\|_{L^2} \|Q\|^{\sigc}_{L^2} \right)^2 = E(Q) [M(Q)]^{\sigc}.
	\]
	It follows from \eqref{eq:ener-below}, \eqref{eq:F} and the conservation of mass and energy that
	\[
	F\left(\|\nabla u(t)\|_{L^2} \|u(t)\|^{\sigc}_{L^2} \right) \leq E(u_0) [M(u_0)]^{\sigc} < E(Q) [M(Q)]^{\sigc} = F\left(\|\nabla Q\|_{L^2} \|Q\|^{\sigc}_{L^2} \right)
	\]
	for all $t\in (-T_*,T^*)$. By \eqref{eq:grad-glob-below}, the continuity argument implies 
	\begin{align} \label{eq:claim-below-1-proof-1}
	\|\nabla u(t)\|_{L^2} \|u(t)\|^{\sigc}_{L^2} < \|\nabla Q\|_{L^2} \|Q\|^{\sigc}_{L^2}
	\end{align}
	for all $t\in (-T_*,T^*)$. Next, using \eqref{eq:ener-below}, we take $\vartheta=\vartheta(u_0,Q)>0$ such that
	\begin{align} \label{eq:vartheta}
	E(u_0) [M(u_0)]^{\sigc} \leq (1-\vartheta) E(Q) [M(Q)]^{\sigc}.
	\end{align}
	Using 
	\[
	E(Q)[M(Q)]^{\sigc} = \frac{N\alpha-4+2b}{2(N\alpha+2b)} \left(\|\nabla Q\|_{L^2} \|Q\|^{\sigc}_{L^2} \right)^2 = \frac{N\alpha-4+2b}{4(\alpha+2)} \left(\|\nabla Q\|_{L^2} \|Q\|^{\sigc}_{L^2}\right)^{\frac{N\alpha+2b}{2}},
	\]
	we we infer from \eqref{eq:F} and \eqref{eq:vartheta} that
	\begin{align} \label{eq:est-vartheta}
	\frac{N\alpha+2b}{N\alpha-4+2b} \left(\frac{\|\nabla u(t)\|_{L^2} \|u(t)\|^{\sigc}_{L^2}}{\|\nabla Q\|_{L^2} \|Q\|^{\sigc}_{L^2}} \right)^2 - \frac{4}{N\alpha-4+2b} \left(\frac{\|\nabla u(t)\|_{L^2} \|u(t)\|^{\sigc}_{L^2}}{\|\nabla Q\|_{L^2} \|Q\|^{\sigc}_{L^2}} \right)^{\frac{N\alpha+2b}{2}} \leq 1-\vartheta
	\end{align}
	for all $t\in (-T_*,T^*)$. Let us consider the function
	\begin{align} \label{eq:G-lambda}
	G(\lambda):= \frac{N\alpha+2b}{N\alpha-4+2b} \lambda^2 - \frac{4}{N\alpha-4+2b} \lambda^{\frac{N\alpha+2b}{2}}
	\end{align}
	with $0<\lambda<1$ due to \eqref{eq:claim-below-1-proof-1}. We see that $G$ is strictly increasing on $(0,1)$ with $G(0)=0$ and $G(1)=1$. It follows from \eqref{eq:G-lambda} that there exists $\rho>0$ depending on $\vartheta$ such that $\lambda \leq 1-\rho$ which is \eqref{eq:claim-below-1}. This finishes the first part of Theorem \ref{theo-dyna-below}.
	
	\vspace{3mm}
	
	\noindent {\bf Case 2. Blow-up.} Let $u_0 \in H^1$ satisfy \eqref{eq:ener-below} and \eqref{eq:grad-blow-below}. Let us prove \eqref{eq:est-solu-blow-below}. By the same argument as above using \eqref{eq:grad-blow-below} instead of \eqref{eq:grad-glob-below}, we have
	\begin{align} \label{eq:est-solu-blow-below-proof}
	\|\nabla u(t)\|_{L^2} \|u(t)\|^{\sigc}_{L^2} > \|\nabla Q\|_{L^2} \|Q\|^{\sigc}_{L^2}
	\end{align}
	for all $t\in (-T_*,T^*)$. Let $\vartheta$ be as in \eqref{eq:vartheta}. By the conservation laws of mass and energy together with \eqref{eq:est-solu-blow-below-proof} and \eqref{eq:prop-Q}, we have
	\begin{align*}
	G(u(t)) [M(u(t))]^{\sigc} &= \|\nabla u(t)\|^2_{L^2}\|u(t)\|^{2\sigc} - \frac{N\alpha+2b}{2(\alpha+2)} P(u(t)) [M(u(t))]^{\sigc} \\
	&=\frac{N\alpha+2b}{2} E(u(t)) [M(u(t))]^{\sigc} - \frac{N\alpha-4+2b}{4} \left(\|\nabla u(t)\|_{L^2} \|u(t)\|^{\sigc}_{L^2} \right)^2 \\
	&\leq \frac{N\alpha+2b}{2} (1-\vartheta) E(Q) [M(Q)]^{\sigc} - \frac{N\alpha-4+2b}{4} \left(\|\nabla Q\|_{L^2} \|Q\|^{\sigc}_{L^2} \right)^2 \\
	&=-\frac{N\alpha-4+2b}{4} \vartheta \left( \|\nabla Q\|_{L^2} \|Q\|^{\sigc}_{L^2}\right)^2
	\end{align*}
	for all $t\in (-T_*,T^*)$. This shows \eqref{eq:est-solu-blow-below} with 
	\[
	\delta:= \frac{N\alpha-4+2b}{4}\vartheta \|\nabla Q\|^2_{L^2} \left(\frac{M(Q)}{M(u_0)}\right)^{\sigc} >0.
	\]
	By Theorem \ref{theo-blow-crite}, the corresponding solution either blows up in finite time, or there exists a time sequence $(t_n)_{n\geq 1}$ satisfying $|t_n|\rightarrow \infty$ such that $\|\nabla u(t_n)\|_{L^2}\rightarrow \infty$ as $n\rightarrow \infty$. 
	
	\vspace{3mm}
	
	\noindent $\bullet$ {\bf Finite variance data.} If we assume in addition that $u_0 \in \Sigma$, then the corresponding solution blows up in finite time. It directly follows from Theorem \ref{theo-blow-crite}.
	
	\vspace{3mm}
	
	\noindent $\bullet$ {\bf Radially symmetric data.} If we assume in addition that $N\geq 2$, $\alpha \leq 4$, and $u_0$ is radially symmetric, then the corresponding solution blows up in finite time. This result was shown in \cite{Dinh-NA}. Note that in \cite{Dinh-NA}, $\alpha$ is assumed to be strictly smaller than 4. However, a closer look at the proof of \cite{Dinh-NA}, we see that $\alpha=4$ is allowed.
	
	\vspace{3mm}
	
	\noindent $\bullet$ {\bf Cylindrically symmetric data.} If we assume in addition that $N\geq 3$, $\alpha\leq 2$, and $u_0\in \Sigma_N$ (see \eqref{eq:Sigma-N}), then the corresponding solution blows up in finite time. To this end, let $\eta$ be a smooth radial function satisfying
	\[
	\eta(y) = \eta(\tau) =  \left\{
	\begin{array}{ccc}
	\tau^2 &\text{if}& \tau \leq 1, \\
	0 &\text{if} & \tau \geq 2, 
	\end{array}
	\right.
	\quad 
	\eta''(\tau) \leq 2, \quad \forall \tau = |y| \geq 0.
	\]
	Given $R>1$, we define the radial function
	\begin{align} \label{eq:psi-R}
	\psi_R(y) := R^2 \eta(y/R).
	\end{align}
	Set
	\begin{align} \label{eq:varphi-R-cylin}
	\varphi_R(x) := \psi_R(y) + x_N^2. 
	\end{align}
	Applying Remark \ref{rem-viri}, we have
	\begin{align*}
	V'_{\varphi_R}(t) = 2 \ima \int  \left(\nabla_y \psi_R(y) \cdot \nabla_y u(t,x) + 2x_N \partial_N u(t,x) \right) \overline{u}(t,x) dx
	\end{align*}
	and
	\begin{align*}
	V''_{\varphi_R}(t) &= -\int \Delta^2_y \psi_R(y) |u(t,x)|^2 dx + 4 \sum_{j,k=1}^{N-1} \rea \int \partial^2_{jk}\psi_R(y) \partial_j \overline{u}(t,x) \partial_k u(t,x) dx \\
	&\mathrel{\phantom{=}} - \frac{2\alpha}{\alpha+2} \int |x|^{-b} \Delta_y \psi_R(y) |u(t,x)|^{\alpha+2} dx -\frac{4b}{\alpha+2} \int |y|^2 \frac{\psi'_R(\tau)}{\tau} |x|^{-b-2} |u(t,x)|^{\alpha+2} dx \\
	&\mathrel{\phantom{=}} + 8 \|\partial_N u(t)\|^2_{L^2} -\frac{4\alpha}{\alpha+2} \int |x|^{-b} |u(t,x)|^{\alpha+2} dx - \frac{8b}{\alpha+2} \int x_N^2|x|^{-b-2} |u(t,x)|^{\alpha+2}dx.
	\end{align*} 
	We can rewrite it as
	\begin{align*}
	V''_{\varphi_R}(t) &= 8 G(u(t)) - 8 \|\nabla_y u(t)\|^2_{L^2} + \frac{4\left((N-1)\alpha +2b\right)}{\alpha+2} P(u(t)) \\
	&\mathrel{\phantom{=}} - \int \Delta^2_y \psi_R(y) |u(t,x)|^2 dx + 4\sum_{j,k=1}^{N-1} \rea \int \partial^2_{jk} \psi_R(y) \partial_j \overline{u}(t,x) \partial_k u(t,x) dx \\
	&\mathrel{\phantom{=}} - \frac{2\alpha}{\alpha+2} \int \Delta_y \psi_R(y) |x|^{-b} |u(t,x)|^{\alpha+2} dx - \frac{4b}{\alpha+2} \int |y|^2 \frac{\psi_R'(\tau)}{\tau} |x|^{-b-2} |u(t,x)|^{\alpha+2} dx \\
	&\mathrel{\phantom{=}}- \frac{8b}{\alpha+2} \int x_N^2 |x|^{-b-2} |u(t,x)|^{\alpha+2} dx.
	\end{align*}
	Rewriting it further, we get
	\begin{align*}
	V''_{\varphi_R}(t) &= 8 G(u(t)) - 8 \|\nabla_y u(t)\|^2_{L^2} + 4\sum_{j,k=1}^{N-1} \rea \int \partial^2_{jk} \psi_R(y) \partial_j \overline{u}(t,x) \partial_k u(t,x) dx \\
	&\mathrel{\phantom{=}} - \int \Delta^2_y \psi_R(y) |u(t,x)|^2 dx + \frac{2\alpha}{\alpha+2} \int \left(2(N-1)-\Delta_y \psi_R(y) \right) |x|^{-b} |u(t,x)|^{\alpha+2} dx \\
	&\mathrel{\phantom{=}} +\frac{4b}{\alpha+2} \int \left( 2|x|^2 - \frac{\psi'_R(\tau)}{\tau} |y|^2 - 2x_N^2\right) |x|^{-b-2} |u(t,x)|^{\alpha+2} dx.
	\end{align*}
	Since $u$ is radially symmetric with respect to the first $N-1$ variables, we use the fact that
	\[
	\partial_j = \frac{y_j}{\tau} \partial_\tau, \quad \partial^2_{jk} = \left(\frac{\delta_{jk}}{\tau} - \frac{y_j y_k}{\tau^3}\right) \partial_\tau + \frac{y_j y_k}{\tau^2} \partial^2_\tau, \quad \tau=|y|, \quad j, k=1, \cdots, N-1 
	\]
	to have
	\[
	\sum_{j,k=1}^{N-1} \partial^2_{jk} \psi_R(y) \partial_j \overline{u}(t,x) \partial_k u(t,x) = \psi''_R(\tau) |\partial_\tau u(t,x)|^2 \leq 2 |\partial_\tau u(t,x)|^2 = 2 |\nabla_y u(t,x)|^2.
	\]
	Thus we get
	\[
	4\sum_{j,k=1}^{N-1} \rea \int \partial^2_{jk} \psi_R(y) \partial_j \overline{u}(t,x) \partial_k u(t,x) dx - 8 \|\nabla_y u(t)\|^2_{L^2} \leq 0.
	\]
	By the conservation of mass and the fact $\|\Delta_y \psi_R\|_{L^\infty} \lesssim R^{-2}$, we have
	\[
	\left| \int \Delta^2_y \psi_R(y) |u(t,x)|^2 dx \right| \lesssim R^{-2}.
	\]
	Moreover, since $\psi_R(y) =|y|^2$ for $|y| \leq R$ and $\|\Delta_y\psi_R\|_{L^\infty} \lesssim 1$, we see that
	\begin{align*}
	\left| \int \left(2(N-1)-\Delta_y \psi_R(y)\right) |x|^{-b} |u(t,x)|^{\alpha+2} dx\right| \lesssim \int_{|y|\geq R} |x|^{-b} |u(t,x)|^{\alpha+2} dx.
	\end{align*}
	Similarly, we have
	\begin{align*}
	\left|\int \left(2|x|^2-\frac{\psi_R'(\tau)}{\tau} |y|^2 - 2x_N^2 \right) |x|^{-b-2} |u(t,x)|^{\alpha+2} dx \right| \lesssim \int_{|y|\geq R} |x|^{-b} |u(t,x)|^{\alpha+2} dx.
	\end{align*} 
	We thus obtain
	\begin{align} \label{eq:est-cylin}
	V''_{\varphi_R}(t) \leq 8 G(u(t)) + CR^{-2} + CR^{-b}\int_{|y|\geq R} |u(t,x)|^{\alpha+2} dx.
	\end{align}
	To estimate the last term in the right hand side of \eqref{eq:est-cylin}, we recall the following radial Sobolev embedding due to Strauss \cite{Strauss}: for any radial function $f:\R^{N-1} \rightarrow \C$, it holds that
	\begin{align} \label{eq:Strauss}
	\sup_{y\ne 0} |y|^{\frac{N-2}{2}} |f(y)| \leq C(N) \|f\|^{\frac{1}{2}}_{L^2_y} \|\nabla_y f\|^{\frac{1}{2}}_{L^2_y}.
	\end{align}
	We estimate
	\begin{align*}
	\int_{\R} \int_{|y|\geq R} |u(t,y,x_N)|^{\alpha+2} dy dx_N \leq \int_{\R}  \|u(t,x_N)\|^\alpha_{L^\infty_y(|y|\geq R)} \|u(t,x_N)\|^2_{L^2_y} dx_N.
	\end{align*}
	We consider separately two subcases: $\alpha=2$ and $\alpha<2$.
	
	\noindent {\bf Subcase 1. $\alpha=2$.} We have
	\begin{align*}
	\int_{\R} \int_{|y|\geq R} |u(t,y,x_N)|^{\alpha+2} dy dx_N \leq \left(\sup_{x_N\in \R} \|u(t,x_N)\|^2_{L^2_y}\right) \int_{\R} \|u(t,x_N)\|^2_{L^\infty_y(|y|\geq R)} dx_N.
	\end{align*}
	By the radial Sobolev embedding \eqref{eq:Strauss} and the conservation of mass, we have
	\begin{align*}
	\int_{\R} \|u(t,x_N)\|^2_{L^\infty_y(|y|\geq R)} dx_N &\lesssim R^{-\frac{N-2}{2}} \int_{\R} \|u(t,x_N)\|_{L^2_y} \|\nabla_y u(t,x_N)\|_{L^2_y} dx_N \\
	& \lesssim R^{-\frac{N-2}{2}} \left(\int_{\R} \|u(t,x_N)\|^2_{L^2_y} dx_N \right)^{1/2} \left(\int_{\R} \|\nabla_y u(t,x_N)\|^2_{L^2_y} dx_N\right)^{1/2} \\
	&= R^{-\frac{N-2}{2}} \|u(t)\|_{L^2_x} \|\nabla_y u(t)\|_{L^2_x} \\
	&\lesssim R^{-\frac{N-2}{2}} \|\nabla_y u(t)\|_{L^2_x}.
	\end{align*}
	Set $g(x_N) := \|u(t,x_N)\|^2_{L^2_y}$. We have
	\begin{align*}
	g(x_N) = \int_{-\infty}^{x_N} \partial_s g(s) ds =2 \int_{-\infty}^{x_N} \rea \int_{\R^{N-1}} \overline{u}(t,y,s) \partial_s u(t,y,s) dy ds \leq 2 \|u(t)\|_{L^2_x} \|\partial_N u(t)\|_{L^2_x}.
	\end{align*}
	Thus we get
	\begin{align*}
	\sup_{x_N \in \R} \|u(t,x_N)\|^2_{L^2_y} \leq C \|\partial_N u(t)\|_{L^2_x}.
	\end{align*}
	This shows that
	\[
	\int_{\R} \int_{|y|\geq R} |u(t,y,x_N)|^{\alpha+2} dy dx_N \lesssim R^{-\frac{N-2}{2}} \|\nabla_y u(t)\|_{L^2_x} \|\partial_N u(t)\|_{L^2_x}\lesssim R^{-\frac{N-2}{2}} \|\nabla u(t)\|^2_{L^2_x}.
	\]
	
	\noindent {\bf Subcase 2. $\alpha<2$.} We have
	\begin{align*}
	\int_{\R} \int_{|y|\geq R} |u(t,y,x_N)|^{\alpha+2} dy dx_N \leq \left(\int_{\R} \|u(t,x_N)\|^2_{L^\infty_y(|y|\geq R)} dx_N\right)^{\frac{\alpha}{2}} \left(\int_{\R} \|u(t,x_N)\|^{\frac{4}{2-\alpha}}_{L^2_y}  dx_N\right)^{\frac{2-\alpha}{2}}. 
	\end{align*}
	By the Gagliardo-Nirenberg inequality, we have
	\begin{align*}
	\int_{\R} \|u(t,x_N)\|_{L^2_y}^{\frac{4}{2-\alpha}} dx_N\lesssim \left\|\partial_N \left(\|u(t,x_N)\|_{L^2_y}\right)\right\|_{L^2_{x_N}}^{\frac{\alpha}{2-\alpha}} \|\|u(t,x_N)\|_{L^2_y} \|_{L^2_{x_N}}^{\frac{4-\alpha}{2-\alpha}}.
	\end{align*}
	By the Cauchy-Schwarz inequality, we see that 
	\begin{align*}
	2\left|\partial_N \left(\|u(t,x_N)\|_{L^2_y}\right)\right| \|u(t,x_N)\|_{L^2_y} &= |\partial_N \left(\|u(t,x_N)\|^2_{L^2_y}\right)| \\
	&= 2\left|\rea \int_{\R^{N-1}} \overline{u}(t,y,x_N) \partial_N u(t,y,x_N) dy\right| \\
	&\leq 2 \|u(t,x_N)\|_{L^2_y} \|\partial_N u(t,x_N)\|_{L^2_y}
	\end{align*}
	which implies that $\left|\partial_N\left(\|u(t,x_N)\|_{L^2_y}\right)\right| \leq \|\partial_N u(t,x_N)\|_{L^2_y}$. It follows that
	\begin{align*}
	\int_{\R} \|u(t,x_N)\|_{L^2_y}^{\frac{4}{2-\alpha}} dx_N &\lesssim \| \|\partial_N u(t,x_N)\|_{L^2_y} \|_{L^2_{x_N}}^{\frac{\alpha}{2-\alpha}} \|u(t)\|_{L^2_x}^{\frac{4-\alpha}{2-\alpha}} \\
	& = \|\partial_N u(t)\|^{\frac{\alpha}{2-\alpha}}_{L^2_x} \|u(t)\|_{L^2_x}^{\frac{4-\alpha}{2-\alpha}} \\
	&\lesssim \|\partial_N u(t)\|^{\frac{\alpha}{2-\alpha}}_{L^2_x}.
	\end{align*}
	Thus, by the Young inequality, we get
	\begin{align*}
	\int_{\R} \int_{|y|\geq R} |u(t,y,x_N)|^{\alpha+2} dy dx_N &\lesssim R^{-\frac{(N-2)\alpha}{4}}  \|\nabla_y u(t)\|^{\frac{\alpha}{2}}_{L^2_x} \|\partial_N u(t)\|_{L^2_x}^{\frac{\alpha}{2}} \\
	&\lesssim R^{-\frac{(N-2)\alpha}{4}} \left(\|\nabla_y u(t)\|_{L^2_x} \|\partial_N u(t)\|_{L^2_x} + 1\right) \\
	&\lesssim R^{-\frac{(N-2)\alpha}{4}} \|\nabla u(t)\|^2_{L^2_x} + C R^{-\frac{(N-2)\alpha}{4}}.
	\end{align*}
	Collecting the above subcases and using \eqref{eq:est-cylin}, we obtain
	\begin{align} \label{eq:est-V-psi-R}
	V''_{\varphi_R}(t) \leq 8 G(u(t)) + CR^{-2} + \left\{ 
	\begin{array}{lcl}
	CR^{-\frac{N-2}{2}-b} \|\nabla u(t)\|^2_{L^2} &\text{if}& \alpha=2, \\
	CR^{-\frac{(N-2)\alpha}{4}-b} \|\nabla u(t)\|^2_{L^2} + CR^{-\frac{(N-2)\alpha}{4}-b} &\text{if}& \alpha<2,
	\end{array}
	\right.
	\end{align}
	for all $t\in (-T_*,T^*)$. Under the assumptions \eqref{eq:ener-below} and \eqref{eq:grad-blow-below}, we have the following estimate due to \cite[(5.8)]{Dinh-NA}: for $\vareps>0$ small enough, there exists a constant $\delta=\delta(\vareps)>0$ such that
	\begin{align} \label{eq:est-vareps}
	8G(u(t)) + \vareps \|\nabla u(t)\|^2_{L^2} \leq -\delta
	\end{align}
	for all $t\in (-T_*,T^*)$. Thanks to \eqref{eq:est-V-psi-R}, we take $R>1$ sufficiently large to get
	\[
	V''_{\psi_R}(t) \leq -\frac{\delta}{2} <0
	\]
	for all $t\in (-T_*,T^*)$. The standard convexity argument yields $T_*,T^*<\infty$. The proof is complete.	
	\hfill $\Box$
	
	We are next interested in long time dynamics of $H^1$-solutions for \eqref{eq:inls} with data at the ground state threshold. To this end, we need the following lemmas.

	\begin{lemma} \label{lem-weak}
		Let $N\geq 1$, $0<b<\min \{2,N\}$, and $0<\alpha <\alpha(N)$. Let $(f_n)_{n\geq 1}$ be a bounded sequence in $H^1$. Then, there exist a subsequence still denoted by $(f_n)_{n\geq 1}$ and a function $f \in H^1$ such that:
		\begin{itemize}[leftmargin=5mm]
			\item $f_n \rightarrow f$ weakly in $H^1$.
			\item $f_n \rightarrow f$ strongly in $L^r_{\loc}$ for all $1\leq r <2^*$.		
			\item $\lim_{n\rightarrow \infty} P(f_n) = P(f)$ as $n\rightarrow \infty$, where $P$ is as in \eqref{eq:P}.
		\end{itemize}
	\end{lemma}
	
	\begin{proof}
		The first two items are well-known. Let us prove the last one. Let $\vareps>0$. Since $(f_n)_{n\geq 1}$ is bounded in $H^1$, we have for any $R>0$,
		\begin{align*}
		\left|\int_{|x| \geq R} |x|^{-b} \left(|f_n(x)|^{\alpha+2} - |f(x)|^{\alpha+2}\right) dx \right| & \leq R^{-b} \left(\|f_n\|^{\alpha+2}_{L^{\alpha+2}} + \|f\|^{\alpha+2}_{L^{\alpha+2}}\right) \\
		&\leq CR^{-b} \left(\|f_n\|_{H^1}^{\alpha+2} + \|f\|^{\alpha+2}_{H^1}\right) \\
		&\leq C R^{-b}.
		\end{align*}
		By choosing $R>0$ sufficiently large, we have
		\begin{align}\label{eq:est-out}
		\left|\int_{|x| \geq R} |x|^{-b} \left(|f_n(x)|^{\alpha+2} - |f(x)|^{\alpha+2}\right) dx \right| <\frac{\vareps}{2}.
		\end{align}
		On the other hand, we have
		\[
		\left|\int_{|x| \leq R} |x|^{-b} \left(|f_n(x)|^{\alpha+2} - |f(x)|^{\alpha+2}\right) dx \right| \leq \||x|^{-b}\|_{L^\delta(|x|\leq R)} \||f_n|^{\alpha+2} - |f|^{\alpha+2}\|_{L^\mu(|x| \leq R)}
		\]
		provided that $\delta, \mu \geq 1, 1=\frac{1}{\delta} + \frac{1}{\mu}$. The term $\||x|^{-b}\|_{L^\delta(|x|\leq R)}$ is finite provided that $\frac{N}{\delta}>b$. Thus $\frac{1}{\delta} >\frac{b}{N}$ and $\frac{1}{\mu}=1-\frac{1}{\delta} <\frac{N-b}{N}$. We next bound
		\[
		\||f_n|^{\alpha+2} - |f|^{\alpha+2}\|_{L^\mu(|x| \leq R)} \lesssim \left(\|f_n\|^{\alpha+1}_{L^\sigma} + \|f\|^{\alpha+1}_{L^\sigma}\right) \|f_n-f\|_{L^\sigma(|x|\leq R)}
		\]
		provided that
		\begin{align} \label{eq:est-ball}
		\frac{\alpha+2}{\sigma} = \frac{1}{\mu} < \frac{N-b}{N}. 
		\end{align}
		By the Sobolev embedding $H^1 \hookrightarrow L^r$ for any $2\leq r<2^*$ and the fact that $f_n \rightarrow f$ strongly in $L^r(|x| \leq R)$ for any $1\leq r <2^*$, we are able to choose $\sigma \in (2,2^*)$ so that \eqref{eq:est-ball} holds. Indeed, in the case $N\geq 3$, we choose $\sigma$ smaller but close to $\frac{2N}{N-2}$. We see that \eqref{eq:est-ball} is satisfied provided that
		\[
		\frac{(\alpha+2)(N-2)}{2N} < \frac{N-b}{N}.
		\]
		This condition is fulfilled since $\alpha<\frac{4-2b}{N-2}$. In the case $N=1,2$, we see that \eqref{eq:est-ball} is satisfied by choosing $\sigma$ sufficiently large. As a consequence, we get
		\begin{align} \label{eq:est-in}
		\left|\int_{|x| \leq R} |x|^{-b} \left(|f_n(x)|^{\alpha+2} - |f(x)|^{\alpha+2}\right) dx \right| \leq C \|f_n-f\|_{L^\sigma(|x|\leq R)} <\frac{\vareps}{2}
		\end{align}
		for $n$ sufficiently large. Collecting \eqref{eq:est-out} and \eqref{eq:est-in}, we prove the result.
	\end{proof}
	
	\begin{lemma} \label{lem-comp-mini-GN}
		Let $N\geq 1$, $0<b<\min \{2,N\}$, and $0<\alpha<\alpha(N)$. Let $Q$ be the unique positive radial solution to \eqref{eq:ground}. Let $(f_n)_{n\geq 1}$ be a sequence of $H^1$-functions satisfying
		\[
		M(f_n) = M(Q), \quad E(f_n) = E(Q), \quad \forall n\geq 1
		\]
		and 
		\[
		\lim_{n\rightarrow \infty} \|\nabla f_n\|_{L^2} = \|\nabla Q\|_{L^2}.
		\]
		Then there exists a subsequence still denoted by $(f_n)_{n\geq 1}$ such that
		\[
		f_n \rightarrow e^{i\theta} Q \quad \text{strongly in } H^1
		\]
		for some $\theta \in \R$ as $n\rightarrow \infty$. 
	\end{lemma}
	
	\begin{proof}
		Since $(f_n)$ is a bounded sequence in $H^1$, by Lemma \ref{lem-weak}, there exist a subsequence still denoted by $(f_n)_{n\geq 1}$ and a function $f\in H^1$ such that $f_n \rightarrow f$ weakly in $H^1$ and $P(f_n) \rightarrow P(f)$ as $n\rightarrow \infty$. We first observe that
		\begin{align*}
		P(f)=\lim_{n\rightarrow \infty} P(f_n) &= \lim_{n\rightarrow \infty} (\alpha+2) \left( \frac{1}{2} \|\nabla f_n\|^2_{L^2} - E(f_n)\right) \\
		&= (\alpha+2) \left(\frac{1}{2} \|\nabla Q\|^2_{L^2} - E(Q)\right) \\
		&=\frac{2(\alpha+2)}{N\alpha+2b} \|\nabla Q\|^2_{L^2} = P(Q).
		\end{align*}
		This shows that $f\ne 0$. Moreover, by the Gagliardo-Nirenberg inequality \eqref{eq:GN-ineq}, we have
		\[
		P(f)-C_{\opt} \|\nabla f\|^{\frac{N\alpha+2b}{2}}_{L^2} \|f\|^{\frac{4-2b-(N-2)\alpha}{2}}_{L^2} \leq 0.
		\]
		By the lower continuity of weak convergence, we have
		\[
		\|\nabla f\|_{L^2} \leq \liminf_{n\rightarrow \infty} \|\nabla f_n\|^2_{L^2}
		\]
		which implies that
		\begin{align*}
		P(f) - C_{\opt} \|\nabla f\|^{\frac{N\alpha+2b}{2}}_{L^2} \|f\|^{\frac{4-2b-(N-2)\alpha}{2}}_{L^2} &\geq \liminf_{n\rightarrow \infty} P(f_n) - C_{\opt} \|\nabla f_n\|^{\frac{N\alpha+2b}{2}}_{L^2} \|f_n\|^{\frac{4-2b-(N-2)\alpha}{2}}_{L^2} \\
		&= P(Q) - C_{\opt} \|\nabla Q\|^{\frac{N\alpha+2b}{2}}_{L^2} \|Q\|_{L^2}^{\frac{4-2b-(N-2)\alpha}{2}} =0.
		\end{align*}
		This shows that $f$ is an optimizer for the Gagliardo-Nirenberg inequality \eqref{eq:GN-ineq}. We also have
		\[
		\|\nabla f\|_{L^2} = \lim_{n\rightarrow \infty} \|\nabla f_n\|^2_{L^2},
		\]
		hence $f_n\rightarrow f$ strongly in $H^1$. We claim that there exists $\theta \in \R$ such that $f(x) =e^{i\theta} g(x)$, where $g$ is a non-negative radial optimizer for \eqref{eq:GN-ineq}. Indeed, since $\|\nabla (|f|)\|_{L^2} \leq \|\nabla f\|_{L^2}$, it is clear that $|f|$ is also an optimizer for \eqref{eq:GN-ineq} and
		\begin{align} \label{eq:prop-f}
		\|\nabla(|f|)\|_{L^2} = \|\nabla f\|_{L^2}.
		\end{align}
		Set $w(x):=\frac{f(x)}{|f(x)|}$. Since $|w(x)|^2=1$, it follows that $\rea(\overline{w} \nabla w(x)) =0$ and
		\[
		\nabla f(x) = \nabla(|f(x)|) w(x) + |f(x)| \nabla w(x) = w(x) (\nabla(|f(x)|) + |f(x)| \overline{w}(x) \nabla w(x))
		\]
		which implies $|\nabla f(x)|^2 = |\nabla (|f(x)|)|^2+ |f(x)|^2 |\nabla w(x)|^2$ for all $x\in \R^3$. From \eqref{eq:prop-f}, we get
		\[
		\int_{\R^3} |f(x)|^2 |\nabla w(x)|^2 dx =0
		\]
		which shows $|\nabla w(x)|=0$, hence $w(x)$ is a constant, and the claim follows with $g(x) = |f(x)|$. Moreover, by replacing $g$ with its symmetric rearrangement, we can assume that $g$ is radially symmetric. Since $g$ is an optimizer for \eqref{eq:GN-ineq}, $g$ must satisfy the Euler-Lagrange equation
		\[
		\left.\frac{d}{d\vareps}\right|_{\vareps=0} W(g+ \vareps \phi) =0,
		\]
		where $W$ is the Weinstein functional
		\[
		W(f):= P(f) \div \left[\|\nabla f\|^{\frac{N\alpha+2b}{2}}_{L^2} \|f\|^{\frac{4-2b-(N-2)\alpha}{2}}_{L^2} \right].
		\]
		A direct computation shows
		\begin{align*}
		-m \Delta g + n g -\frac{\alpha+2}{C_{\opt}} |x|^{-b} |g|^\alpha g =0,
		\end{align*}
		where
		\begin{align*}
		m:&=\frac{N\alpha+2b}{2} \|\nabla f\|^{\frac{N\alpha+2b-4}{2}}_{L^2} \|f\|^{\frac{4-2b-(N-2)\alpha}{2}}_{L^2}, \\
		n:&= \frac{4-2b-(N-2)\alpha}{2} \|\nabla f\|^{\frac{N\alpha+2b}{2}}_{L^2} \|f\|^{-\frac{2b+(N-2)\alpha}{2}}_{L^2}.
		\end{align*}
		By a change of variable $g(x) = \lambda \phi(\mu x)$ with $\lambda, \mu>0$ satisfying
		\[
		\mu^2=\frac{n}{m}, \quad \lambda^\alpha = \frac{nC_{\opt}}{\alpha+2} \mu^{-b},
		\] 
		we see that $\phi$ solves \eqref{eq:ground} and $W(g) = W(\phi) =C_{\opt}$. By the uniqueness of positive radial solution to \eqref{eq:ground} due to \cite{Genoud-2d, Yanagida, Toland}, we have $\phi \equiv Q$. As $\|g\|_{L^2}= \|Q\|_{L^2}$ and $\|\nabla g\|_{L^2}=\|\nabla Q\|_{L^2}$, we infer that $\lambda =\mu=1$. This shows that $f(x)=e^{i\theta} Q(x)$ for some $\theta \in \R$. The proof is complete.
	\end{proof}

	\noindent {\it Proof of Theorem \ref{theo-dyna-at}.}
	We consider separately three cases.
	
	\vspace{3mm}
	
	\noindent {\bf Case 1.} Let $u_0 \in H^1$ satisfy \eqref{eq:ener-at} and \eqref{eq:grad-at-1}. We first note that \eqref{eq:ener-at} and \eqref{eq:grad-at-1} are invariant under the scaling 
	\begin{align}\label{eq:scal-u0}
	u_0^\lambda(x):= \lambda^{\frac{2-b}{\alpha}} u_0(\lambda x), \quad \lambda>0.
	\end{align}
	By choosing a suitable scaling, we can assume that
	\begin{align} \label{eq:M-E-u0}
	M(u_0) = M(Q), \quad E(u_0) = E(Q).
	\end{align}
	Thus \eqref{eq:grad-at-1} becomes $\|\nabla u_0\|_{L^2}<\|\nabla Q\|_{L^2}$. We first claim that 
	\begin{align} \label{eq:claim-at-1}
	\|\nabla u(t)\|_{L^2}< \|\nabla Q\|_{L^2}
	\end{align}
	for all $t\in (-T_*,T^*)$. Assume by contradiction that there exists $t_0 \in (-T_*,T^*)$ such that $\|\nabla u(t_0)\|_{L^2} \geq \|\nabla Q\|_{L^2}$. By continuity, there exists $t_1 \in (-T_*,T^*)$ such that $\|\nabla u(t_1)\|_{L^2} = \|\nabla Q\|_{L^2}$. By the conservation of energy and \eqref{eq:ener-Q}, we see that
	\begin{align*}
	P(u(t_1)) &= (\alpha+2) \left( \frac{1}{2} \|\nabla u(t_1)\|^2_{L^2} - E(u(t_1))\right) \\
	&= (\alpha+2) \left( \frac{1}{2} \|\nabla Q\|^2_{L^2}- E(Q)\right) \\
	&= \frac{2(\alpha+2)}{N\alpha+2b} \|\nabla Q\|^2_{L^2}.
	\end{align*}
	This shows that $u(t_1)$ is an optimizer for the Gagliardo-Nirenberg inequality \eqref{eq:GN-ineq}. Arguing as in the proof of Lemma \ref{lem-comp-mini-GN}, we have $u(t_1) = e^{i\theta}Q$ for some $\theta \in \R$. Moreover, by the uniqueness of solution to \eqref{eq:inls}, we infer that $u(t) =e^{it} e^{i\theta} Q$ which contradicts \eqref{eq:grad-at-1}. This shows \eqref{eq:claim-at-1}. In particular, the solution exists globally in time. We now have two possibilities.
	
	\noindent {\bf First possibility.} If 
	\[
	\sup_{t\in \R} \|\nabla u(t)\|_{L^2} < \|\nabla Q\|_{L^2},
	\]
	then there exists $\rho>0$ such that 
	\[
	\|\nabla u(t)\|_{L^2} \leq (1-\rho)\|\nabla Q\|_{L^2}
	\]
	which, by \eqref{eq:M-E-u0}, implies that \eqref{eq:claim-below-1} holds for all $t\in \R$. By the same argument as in the proof of Theorem \ref{theo-dyna-below}, we prove \eqref{eq:est-solu-at}. In particular, if $N\geq 2$ and $0<b<\min \left\{2,\frac{N}{2}\right\}$, then by Theorem \ref{theo-scat-crite}, the solution scatters in $H^1$ in both directions.
	
	\noindent {\bf Second possibility.} If
	\[
	\sup_{t\in \R} \|\nabla u(t)\|_{L^2} = \|\nabla Q\|_{L^2},
	\]
	then there exists a time sequence $(t_n)_{n\geq 1} \subset \R$ such that
	\[
	M(u(t_n))= M(Q), \quad E(u(t_n)) = E(Q), \quad \lim_{n\rightarrow \infty} \|\nabla u(t_n)\|_{L^2} = \|\nabla Q\|_{L^2}.
	\]
	We notice that $|t_n| \rightarrow \infty$. Otherwise, passing to a subsequence if necessary, we have $t_n \rightarrow t_0$ as $n\rightarrow \infty$. By continuity of the solution, we have $u(t_n) \rightarrow u(t_0)$ strongly in $H^1$. This implies that $u(t_0)$ is an optimizer for \eqref{eq:GN-ineq} which is a contradiction.
	
	Applying Lemma \ref{lem-comp-mini-GN} with $f_n = u(t_n)$, we prove that up to a subsequence,
	\[
	u(t_n)\rightarrow e^{i\theta}Q \text{ strongly in } H^1
	\]
	for some $\theta \in \R$ as $n\rightarrow \infty$.
	
	\vspace{3mm}
	
	\noindent {\bf Case 2.} Let $u_0 \in H^1$ satisfy \eqref{eq:ener-at} and \eqref{eq:grad-at-2}. By the scaling \eqref{eq:scal-u0}, we can assume that
	\[
	M(u_0) = M(Q), \quad \|\nabla u_0\|_{L^2}= \|\nabla Q\|_{L^2}, \quad E(u_0) = E(Q).
	\]
	In particular, $u_0$ is an optimizer for $\eqref{eq:GN-ineq}$ which implies $u_0(x)=e^{i\theta} Q(x)$ for some $\theta \in \R$. By the uniqueness of solution to \eqref{eq:inls}, we have $u(t,x)=e^{it} e^{i\theta} Q(x)$.
	
	\vspace{3mm}
	
	\noindent {\bf Case 3.} Let $u_0 \in H^1$ satisfy \eqref{eq:ener-at} and \eqref{eq:grad-at-3}. As in Case 1, we can assume that
	\begin{align} \label{eq:M-E-blow}
	M(u_0) = M(Q), \quad E(u_0) = E(Q), \quad \|\nabla u_0\|_{L^2}>\|\nabla Q\|_{L^2}. 
	\end{align}
	Arguing as above, we prove that
	\[
	\|\nabla u(t)\|_{L^2} > \|\nabla Q\|_{L^2}
	\]
	for all $t\in (-T_*,T^*)$. Let us consider only positive times. The one for negative times is similar. If $T^*<\infty$, then we are done. Otherwise, if $T^*=\infty$, then we consider two possibilities.
	
	\noindent {\bf First possibility.} If 
	\[
	\sup_{t\in [0,\infty)} \|\nabla u(t)\|_{L^2} > \|\nabla Q\|_{L^2},
	\] 
	then there exists $\rho >0$ such that 
	\begin{align} \label{eq:est-firs-posi}
	\|\nabla u(t)\|_{L^2} \geq (1+\rho) \|\nabla Q\|_{L^2}
	\end{align}
	for all $t\in [0,\infty)$. By \eqref{eq:M-E-blow} and the conservation laws of mass and energy, we have
	\begin{align*}
	G(u(t)) [M(u(t))]^{\sigc} &= \frac{N\alpha+2b}{2}E(u(t)) [M(u(t))]^{\sigc} - \frac{N\alpha-4+2b}{4} \left(\|\nabla u(t)\|_{L^2} \|u(t)\|^{\sigc}_{L^2} \right)^2 \\
	&\leq \frac{N\alpha+2b}{2} E(Q) [M(Q)]^{\sigc} - \frac{N\alpha-4+2b}{4} \left((1+\rho) \|\nabla Q\|_{L^2} \|Q\|^{\sigc}_{L^2} \right)^2 \\
	&= -\frac{N\alpha-4+2b}{4} \left((1+\rho)^2-1\right) \left( \|\nabla Q\|_{L^2} \|Q\|^{\sigc}_{L^2}\right)^2
	\end{align*}
	for all $t\in [0,\infty)$. By Theorem \ref{theo-blow-crite}, there exists a time sequence $t_n\rightarrow \infty$ such that $\|\nabla u(t_n)\|_{L^2} \rightarrow \infty$ as $n\rightarrow \infty$.
	
	\noindent {\bf Second possibility.} If 
	\[
	\sup_{t\in[0,\infty)} \|\nabla u(t)\|_{L^2} = \|\nabla Q\|_{L^2},
	\]
	then there exists a time sequence $(t_n)_{n\geq 1}$ such that $\|\nabla u(t_n)\|_{L^2} \rightarrow \|\nabla Q\|_{L^2}$ as $n\rightarrow \infty$. Arguing as in Case 1, we show that $t_n \rightarrow \infty$ and 
	\[
	u(t_n) \rightarrow e^{i\theta} Q \text{ strongly in } H^1
	\]
	for some $\theta \in \R$ as $n\rightarrow \infty$. This completes the first part of Item (3) of Theorem \ref{theo-dyna-at}.
	
	Let us prove the second part of Item (3) of Theorem \ref{theo-dyna-at}. 
	
	\vspace{3mm}
	
	\noindent $\bullet$ {\bf Finite variance data.} If we assume in addition that $u_0 \in \Sigma$, then the first possibility cannot occur. In fact, if it occurs, then there exists $\delta>0$ such that
	\[
	G(u(t)) \leq -\delta
	\]
	for all $t\in [0,\infty)$. This is impossible by the convexity argument as
	\[
	\frac{d^2}{dt^2} \|xu(t)\|^2_{L^2} = 8 G(u(t)).
	\]
	
	\vspace{3mm}
	
	\noindent $\bullet$ {\bf Radially symmetric data.} If we assume in addition that $N\geq 2$, $\alpha \leq 4$, and $u_0$ is radially symmetric, then the first possibility cannot occur. In fact, suppose that the first possibility occurs, so \eqref{eq:est-firs-posi} holds. It follows from \eqref{eq:M-E-blow} and \eqref{eq:prop-Q} that
	\begin{align*}
	8 G(u(t)) &+ \vareps \|\nabla u(t)\|^2_{L^2} \\
	&= 4(N\alpha+2b) E(u(t)) [M(u(t))]^{\sigc} - (2N\alpha -4b+8 -\vareps) \|\nabla u(t)\|^2_{L^2} [M(u(t))]^{\sigc} \\
	&\leq  4(N\alpha+2b) E(Q) [M(Q)]^{\sigc} - (2N\alpha-4b+8-\vareps) (1+\rho)^2 \left(\|\nabla Q\|_{L^2} \|Q\|^{\sigc}_{L^2}\right)^2 \\
	&= -2(N\alpha-4+2b) \left(\|\nabla Q\|_{L^2} \|Q\|^{\sigc}_{L^2}\right)^2 (1+\rho)^2 \left[\frac{(1+\rho)^2-1}{(1+\rho)^2} - \frac{\vareps}{2(N\alpha-4+2b)}\right]
	\end{align*}
	for all $t\in [0,\infty)$. Taking $\vareps>0$ sufficiently small, there exists $\delta=\delta(\vareps)>0$ such that 
	\begin{align} \label{eq:est-G-vareps}
	8 G(u(t)) +\vareps \|\nabla u(t)\|^2_{L^2} \leq-\delta
	\end{align}
	for all $t\in [0,\infty)$. We recall the following estimate due to \cite[Lemma 3.4]{Dinh-NA}: for any $R>1$ and any $\vareps>0$,
	\[
	V''_{\varphi_R} (t) \leq 8 G(u(t)) + \left\{
	\renewcommand*{\arraystretch}{1.2}
	\begin{array}{lcl}
	CR^{-2} + CR^{-[2(N-1)+b]} \|\nabla u(t)\|^2_{L^2} &\text{if}& \alpha=4, \\
	CR^{-2} + C\vareps^{-\frac{\alpha}{4-\alpha}} R^{-\frac{2[(N-1)\alpha+2b]}{4-\alpha}} + \vareps \|\nabla u(t)\|^2_{L^2} &\text{if}& \alpha<4.
	\end{array}
	\right.
	\]
	Thanks to \eqref{eq:est-G-vareps}, we take $R>1$ sufficiently large if $\alpha=4$, and $\vareps>0$ sufficiently small and $R>1$ sufficiently large depending on $\vareps$, we obtain
	\[
	V''_{\varphi_R}(t) \leq -\frac{\delta}{2}
	\]	
	for all $t\in [0,\infty)$. This is impossible.
	
	\vspace{3mm}
	
	\noindent $\bullet$ {\bf Cylindrically symmetric data.} If we assume in addition that $N\geq 3$, $\alpha\leq 2$, and $u_0 \in \Sigma_N$, then the first possibility cannot occur. This is done by the same argument as above using \eqref{eq:est-V-psi-R} and \eqref{eq:est-G-vareps}. The proof of Theorem \ref{theo-dyna-at} is now complete. 
	\hfill $\Box$
	
	Finally, we study long time dynamics of $H^1$-solutions for \eqref{eq:inls} with data above the ground state threshold. 
	
	\noindent {\it Proof of Theorem \ref{theo-dyna-above}.} Let us consider two cases.
	
	\noindent {\bf Case 1.} Let $u_0 \in \Sigma$ satisfy \eqref{eq:ener-above-1}, \eqref{eq:ener-above-2}, \eqref{eq:gwp-above-1}, and \eqref{eq:gwp-above-2}. We will show that \eqref{scat-crite} holds. To this end, let us start with the following estimate: for $f \in \Sigma$, 
	\begin{align} \label{eq:est-fini-vari}
	\left( \ima \int \bar{f} x \cdot \nabla f dx\right)^2 \leq \|xf\|^2_{L^2} \left(\|\nabla f\|^2_{L^2} - [C_{\opt}]^{-\frac{4}{N\alpha+2b}} [M(f)]^{-\frac{4-2b-(N-2)\alpha}{N\alpha+2b}} [P(f)]^{\frac{4}{N\alpha+2b}} \right).
	\end{align}
	In fact, let $\lambda>0$. We have
	\begin{align*}
	\int |\nabla(e^{i\lambda|x|^2} f)|^2 dx = 4\lambda^2 \|xf\|^2_{L^2} + 4 \lambda \ima \int \bar{f} x\cdot \nabla f dx + \|\nabla f\|^2_{L^2}.
	\end{align*}
	By the Gagliardo-Nirenberg inequality \eqref{eq:GN-ineq}, we have
	\[
	[P(f)]^{\frac{4}{N\alpha+2b}} = [P(e^{i\lambda |x|^2} f)]^{\frac{4}{N\alpha+2b}} \leq [C_{\opt}]^{\frac{4}{N\alpha+2b}} \|\nabla(e^{i\lambda |x|^2} f)\|^2_{L^2} \|f\|^{\frac{2[4-2b-(N-2)\alpha]}{N\alpha+2b}}_{L^2}
	\]
	or
	\[
	\|\nabla(e^{i\lambda |x|^2}f)\|^2_{L^2} \geq [C_{\opt}]^{-\frac{4}{N\alpha+2b}} M(f)^{-\frac{4-2b-(N-2)\alpha}{N\alpha+2b}} [P(f)]^{\frac{4}{N\alpha+2b}}.
	\]
	It follows that
	\[
	4\lambda^2 \|xf\|^2_{L^2} + 4 \lambda \ima \int \bar{f} x\cdot \nabla f dx + \|\nabla f\|^2_{L^2} - [C_{\opt}]^{-\frac{4}{N\alpha+2b}} [M(f)]^{-\frac{4-2b-(N-2)\alpha}{N\alpha+2b}} [P(f)]^{\frac{4}{N\alpha+2b}} \geq 0
	\]
	for all $\lambda >0$. Since the left hand side is a quadratic polynomial in $\lambda$, its discriminant must be non-positive which proves \eqref{eq:est-fini-vari}.
	
	We also have
	\begin{align*}
	V''(t) &= 8 \|\nabla u(t)\|^2_{L^2} - \frac{4(N\alpha+2b)}{\alpha+2} P(u(t)) \\
	&= 16E(u(t)) - \frac{4(N\alpha-4+2b)}{\alpha+2} P(u(t)) \\
	&= 4(N\alpha+2b) E(u(t)) - 2(N\alpha-4+2b) \|\nabla u(t)\|^2_{L^2}
	\end{align*}
	which implies that
	\begin{align*}
	P(u(t)) &= \frac{\alpha+2}{4(N\alpha-4+2b)} \left( 16 E(u(t))-V''(t)\right), \\
	\|\nabla u(t)\|^2_{L^2} &= \frac{1}{2(N\alpha-4+2b)} \left(4(N\alpha+2b) E(u(t)) - V''(t)\right).
	\end{align*}
	Since $P(u(t)) \geq 0$, we have $V''(t) \leq 16 E(u(t))=16 E(u_0)$. Inserting the above identities to \eqref{eq:est-fini-vari}, we get
	\begin{multline*}
	(V'(t))^2 \leq 16 V(t) \Big[ \frac{1}{2(N\alpha-4+2b)} \left( 4(N\alpha+2b) E(u(t)) - V''(t)\right) \\
	-[C_{\opt}]^{-\frac{4}{N\alpha+2b}} [M(u(t))]^{-\frac{4-2b-(N-2)\alpha}{N\alpha+2b}} \Big(\frac{\alpha+2}{4(N\alpha-4+2b)} \left( 16 E(u(t))-V''(t)\right) \Big)^{\frac{4}{N\alpha+2b}} \Big]
	\end{multline*} 
	which implies
	\begin{align} \label{eq:g-V}
	(z'(t))^2 \leq 4 g (V''(t)),
	\end{align}
	where
	\[
	z(t):= \sqrt{V(t)}
	\]
	and 
	\begin{multline*}
	g(\lambda):= \frac{1}{2(N\alpha-4+2b)} \left( 4(N\alpha+2b) E - \lambda\right) \\
	-[C_{\opt}]^{-\frac{4}{N\alpha+2b}} M^{-\frac{4-2b-(N-2)\alpha}{N\alpha+2b}} \Big(\frac{\alpha+2}{4(N\alpha-4+2b)} \left( 16 E-\lambda\right) \Big)^{\frac{4}{N\alpha+2b}}
	\end{multline*}
	with $\lambda \leq 16E$. Here we have used the notation $E(u(t)) = E, M(u(t))=M$ due to the conservation of mass and energy. Since $N\alpha+2b>4$, we see that $g(\lambda)$ is decreasing on $(-\infty, \lambda_0)$ and increasing on $(\lambda_0, 16E)$, where $\lambda_0$ satisfies
	\begin{align} \label{eq:lambda0}
	\frac{N\alpha+2b}{2(\alpha+2)} = [C_{\opt}]^{-\frac{4}{N\alpha+2b}} M^{-\frac{4-2b-(N-2)\alpha}{N\alpha+2b}} \Big(\frac{\alpha+2}{4(N\alpha-4+2b)} (16E -\lambda_0) \Big)^{\frac{4-N\alpha-2b}{N\alpha+2b}}.
	\end{align}
	A direct calculation shows
	\begin{align*}
	g(\lambda_0) = \frac{1}{2(N\alpha-4+2b)}\left(4(N\alpha+2b)E-\lambda_0\right) - \frac{N\alpha+2b}{8(N\alpha-4+2b)} \left(16E -\lambda_0\right) =\frac{\lambda_0}{8}.
	\end{align*}
	Using the fact that
	\[
	C_{\opt} = \frac{2(\alpha+2)}{N\alpha+2b} \left(\frac{2(N\alpha+2b)}{N\alpha-4+2b} E(Q) [M(Q)]^{\sigc} \right)^{-\frac{N\alpha-4+2b}{4}},
	\]
	we infer from \eqref{eq:lambda0} that
	\[
	1=\frac{16 E(Q)[M(Q)]^{\sigc}}{(16E-\lambda_0) M^{\sigc}}
	\]
	or
	\begin{align} \label{eq:M-E-lambda0}
	\frac{EM^{\sigc}}{E(Q)[M(Q)]^{\sigc}} \left(1-\frac{\lambda_0}{16E}\right)=1.
	\end{align}
	Thus the assumption \eqref{eq:ener-above-1} is equivalent to 
	\begin{align} \label{eq:ener-above-1-equi}
	\lambda_0 \geq 0.
	\end{align}
	Moreover, the assumption \eqref{eq:ener-above-2} is equivalent to 
	\[
	(V'(0))^2 \geq 2V(0) \lambda_0
	\]
	or
	\begin{align} \label{eq:ener-above-2-equi}
	(z'(0))^2 \geq \frac{\lambda_0}{2} = 4 g(\lambda_0).
	\end{align}
	Similarly, the assumption \eqref{eq:gwp-above-2} is equivalent to
	\begin{align} \label{eq:gwp-above-2-equi}
	z'(0) \geq 0.
	\end{align}
	Finally, the assumption \eqref{eq:gwp-above-1} is equivalent to 
	\begin{align} \label{eq:gwp-above-1-equi}
	V''(0) >\lambda_0.
	\end{align} 
	Indeed, from \eqref{eq:gwp-above-1}, we have
	\begin{align*}
	V''(0) &= 16 E - \frac{4(N\alpha-4+2b)}{\alpha+2} P(u_0) \\
	&>16E - \frac{4(N\alpha-4+2b)}{\alpha+2} \frac{P(Q) [M(Q)]^{\sigc}}{M^{\sigc}} \\
	&=16 \left(E - \frac{E(Q)[M(Q)]^{\sigc}}{M^{\sigc}}\right) \\
	&=16E \left(1-\frac{E(Q)[M(Q)]^{\sigc}}{EM^{\sigc}}\right) \\
	&= \lambda_0,
	\end{align*}
	where we have used \eqref{eq:M-E-lambda0} to get the last equality. 
	
	Next, we claim that there exists $\delta_0>0$ small such that for all $t\in [0,T^*)$,
	\begin{align} \label{eq:claim-above-1}
	V''(t) \geq \lambda_0 + \delta_0.
	\end{align}
	Assume \eqref{eq:claim-above-1} for the moment, we prove \eqref{scat-crite}. We have
	\begin{align*}
	P(u(t)) [M(u(t))]^{\sigc} &= \frac{\alpha+2}{4(N\alpha-4+2b)} \left(16E - V''(t)\right) M^{\sigc} \\
	&\leq \frac{\alpha+2}{4(N\alpha-4+2b)} (16E -\lambda_0 -\delta_0) M^{\sigc} \\
	&= \frac{4(\alpha+2)}{N\alpha-4+2b} E(Q) [M(Q)]^{\sigc} - \frac{\alpha+2}{4(N\alpha-4+2b)}\delta_0 M^{\sigc} \\
	& = (1-\rho) P(Q) [M(Q)]^{\sigc} 
	\end{align*}
	for all $t\in [0,T^*)$, where $\rho:= \frac{\alpha+2}{4(N\alpha-4+2b)} \delta_0 \frac{M^{\sigc}}{P(Q)[M(Q)]^{\sigc}}>0$. Here we have used \eqref{eq:M-E-lambda0} to get the third line. This shows \eqref{scat-crite}. In particular, if $N\geq 2$ and $0<b<\min \left\{2,\frac{N}{2}\right\}$, then the solution scatters in $H^1$ forward in time.
	
	It remains to show \eqref{eq:claim-above-1}. By \eqref{eq:gwp-above-1-equi}, we take $\delta_1>0$ so that 
	\[
	V''(0) \geq \lambda_0 + 2\delta_1.
	\]
	By continuity, we have
	\begin{align} \label{eq:claim-above-1-proof-1}
	V''(t) >\lambda_0 + \delta_1, \quad \forall t\in [0,t_0).
	\end{align}
	for $t_0>0$ sufficiently small. By reducing $t_0$ if necessary, we can assume that
	\begin{align} \label{eq:claim-above-1-proof-2}
	z'(t_0) >2\sqrt{g(\lambda_0)}.
	\end{align}
	In fact, if $z'(0) >2\sqrt{g(\lambda_0)}$, then \eqref{eq:claim-above-1-proof-2} follows from the continuity argument. Otherwise, if $z'(0)=2\sqrt{g(\lambda_0)}$, then using the fact that
	\begin{align} \label{eq:claim-above-1-proof-3}
	z''(t) = \frac{1}{z(t)} \left(\frac{V''(t)}{2} - (z'(t))^2\right)
	\end{align}
	and \eqref{eq:gwp-above-1-equi}, we have $z''(0)>0$. This shows \eqref{eq:claim-above-1-proof-2} by taking $t_0>0$ sufficiently small. Thanks to \eqref{eq:claim-above-1-proof-2}, we take $\epsilon_0>0$ be a small constant so that
	\begin{align} \label{eq:claim-above-1-proof-4}
	z'(t_0) \geq 2\sqrt{g(\lambda_0)} + 2\epsilon_0.
	\end{align}
	We will prove by contradiction that 
	\begin{align} \label{eq:claim-above-1-proof-5}
	z'(t) > 2\sqrt{g(\lambda_0)} + \epsilon_0, \quad \forall t\geq t_0.
	\end{align}
	Suppose that it is not true and set
	\[
	t_1:= \inf \left\{ t\geq t_0 \ : \ z'(t) \leq 2\sqrt{g(\lambda_0)} + \epsilon_0 \right\}.
	\]
	By \eqref{eq:claim-above-1-proof-4}, we have $t_1 >t_0$. By continuity, we have
	\begin{align} \label{eq:claim-above-1-proof-6}
	z'(t_1) = 2\sqrt{g(\lambda_0)} + \epsilon_0
	\end{align}
	and
	\begin{align} \label{eq:claim-above-1-proof-7}
	z'(t) \geq 2\sqrt{g(\lambda_0)} + \epsilon_0, \quad \forall t\in [t_0,t_1].
	\end{align}
	By \eqref{eq:g-V}, we see that
	\begin{align} \label{eq:claim-above-1-proof-8}
	\left(2\sqrt{g(\lambda_0)} + \epsilon_0\right)^2 \leq (z'(t))^2 \leq 4g(V''(t)), \quad \forall t\in [t_0,t_1].
	\end{align}
	It follows that $g(V''(t)) > g(\lambda_0)$ for all $t\in [t_0,t_1]$, thus $V''(t) \ne \lambda_0$ and by continuity, $V''(t) >\lambda_0$ for all $t\in [t_0,t_1]$. 
	
	We will prove that there exists a constant $C>0$ such that
	\begin{align} \label{eq:claim-above-1-proof-9}
	V''(t) \geq \lambda_0 +\frac{\sqrt{\epsilon_0}}{C}, \quad \forall t\in [t_0,t_1].
	\end{align}
	Indeed, by the Taylor expansion of $g$ near $\lambda_0$ with the fact $g'(\lambda_0)=0$, there exists $a>0$ such that
	\begin{align} \label{eq:claim-above-1-proof-10}
	g(\lambda) \leq g(\lambda_0) + a(\lambda-\lambda_0)^2, \quad \forall \lambda :  |\lambda-\lambda_0| \leq 1.
	\end{align}
	If $V''(t) \geq \lambda_0+1$, then \eqref{eq:claim-above-1-proof-9} holds by taking $C$ large. If $\lambda_0<V''(t) \leq \lambda_0+1$, then by \eqref{eq:claim-above-1-proof-8} and \eqref{eq:claim-above-1-proof-10}, we get
	\[
	\left(2\sqrt{g(\lambda_0)} +\epsilon_0\right)^2 \leq (z'(t))^2 \leq 4g(V''(t)) \leq 4g(\lambda_0) + 4a (V''(t)-\lambda_0)^2
	\]
	thus
	\[
	4\epsilon_0\sqrt{g(\lambda_0)} + \epsilon_0^2 \leq 4a (V''(t)-\lambda_0)^2.
	\]
	This shows \eqref{eq:claim-above-1-proof-9} with $C=\sqrt{a} [g(\lambda_0)]^{-\frac{1}{4}}$. 
	
	However, by \eqref{eq:claim-above-1-proof-3}, \eqref{eq:claim-above-1-proof-6} and \eqref{eq:claim-above-1-proof-9}, we have
	\begin{align*}
	z''(t_1) &= \frac{1}{z(t_1)} \left(\frac{V''(t_1)}{2} - (z'(t_1))^2\right) \\
	&\geq \frac{1}{z(t_1)} \left(\frac{\lambda_0}{2} + \frac{\sqrt{\epsilon_0}}{2C}- \left(2\sqrt{g(\lambda_0)} +\epsilon_0\right)^2 \right) \\
	&\geq \frac{1}{z(t_1)} \left(\frac{\sqrt{\epsilon_0}}{2C} -4\epsilon_0 \sqrt{g(\lambda_0)} - \epsilon_0^2\right)>0
	\end{align*}
	provided that $\epsilon_0$ is taken small enough. This however contradicts \eqref{eq:claim-above-1-proof-6} and \eqref{eq:claim-above-1-proof-7}. This proves \eqref{eq:claim-above-1-proof-5}. Note that we have also proved \eqref{eq:claim-above-1-proof-9} for all $t\in [t_0,T^*)$. This together with \eqref{eq:claim-above-1-proof-1} imply \eqref{eq:claim-above-1} with $\delta_0=\min \left\{\delta_1,\frac{\sqrt{\epsilon_0}}{C}\right\}$.

	{\bf Case 2.} Let $u_0 \in \Sigma$ satisfy \eqref{eq:ener-above-1}, \eqref{eq:ener-above-2}, \eqref{eq:blow-above-1} and \eqref{eq:blow-above-2}. As in Step 1, we see that the conditions \eqref{eq:ener-above-1}, \eqref{eq:ener-above-2}, \eqref{eq:blow-above-1} and \eqref{eq:blow-above-2} are respectively equivalent to 
	\begin{align} \label{eq:equi-cond-above-blow}
	\lambda_0 \geq 0, \quad (z'(0))^2 \geq 4g(\lambda_0) =\frac{\lambda_0}{2}, \quad V''(0) <\lambda_0, \quad z'(0) \leq 0.
	\end{align}
	We claim that 
	\begin{align} \label{eq:claim-above-2}
	z''(t) <0, \quad \forall t\in [0,T^*).
	\end{align}
	Note that by \eqref{eq:claim-above-1-proof-3}, we have $z''(0)<0$. Assume by contraction that \eqref{eq:claim-above-2} does not hold. Then there exists $t_0 \in (0,T^*)$ such that
	\[
	z''(t) <0, \quad \forall t\in [0,t_0)
	\]
	and $z''(t_0) =0$. By \eqref{eq:equi-cond-above-blow}, we have
	\[
	z'(t) <z'(0) \leq -2\sqrt{g(\lambda_0)}, \quad \forall t\in (0,t_0].
	\]
	Hence $(z'(t))^2 >2g(\lambda_0)$ which combined with \eqref{eq:g-V} imply that 
	\[
	g(V''(t)) > g(\lambda_0), \quad \forall t\in (0,t_0].
	\]
	It follows that $V''(t) \ne \lambda_0$ for all $t\in (0,t_0]$, and by continuity, we have
	\[
	V''(t) <\lambda_0, \quad \forall t\in [0,t_0].
	\]
	By \eqref{eq:claim-above-1-proof-3}, we obtain
	\[
	z''(t_0) = \frac{1}{z(t_0)} \left(\frac{V''(t_0)}{2} -(z'(t_0))^2\right) <\frac{1}{z(t_0)} \left(\frac{\lambda_0}{2}-\frac{\lambda_0}{2}\right) =0
	\]
	which is absurd. 
	Now, assume by contradiction that the solution exists globally forward in time, i.e., $T^*=\infty$. By \eqref{eq:claim-above-2}, we see that
	\[
	z'(t) \leq z'(1) <z'(0)\leq 0, \quad \forall t\in [1,\infty).
	\]
	This contradicts with the fact that $z(t)$ is positive. The proof is complete.	
	\hfill $\Box$

	\section*{Acknowledgement}
	This work was supported in part by the Labex CEMPI (ANR-11-LABX-0007-01). V. D. D. would like to express his deep gratitude to his wife - Uyen Cong for her encouragement and support. 


\begin{bibdiv}
		\begin{biblist}
			
			\bib{ADM}{article}{
				author={Arora, A. K.},
				author={Dodson, B.},
				author={Murphy, J.},
				title={Scattering below the ground state for the 2$d$ radial nonlinear
					Schr\"{o}dinger equation},
				journal={Proc. Amer. Math. Soc.},
				volume={148},
				date={2020},
				number={4},
				pages={1653--1663},
				issn={0002-9939},
			}
			\bib{BF-arXiv}{article}{
				author={Bellazzini, J.},
				author={Forcella, L.},
				title={Dynamical collapse of cylindrical symmetric dipolar Bose-Einstein condensates},
				journal={preprint},
				eprint={http://arxiv.org/abs/2005.02894v1},
			}
			
			\bib{BF}{article}{
				author={De Bouard, A.},
				author={Fukuizumi, R.},
				title={Stability of standing waves for nonlinear Schr\"{o}dinger equations
					with inhomogeneous nonlinearities},
				journal={Ann. Henri Poincar\'{e}},
				volume={6},
				date={2005},
				number={6},
				pages={1157--1177},
				issn={1424-0637},
			}
			
			\bib{Campos}{article}{
				author={Campos, L. },
				title={Scattering of radial solutions to the inhomogeneous nonlinear Schr\"odinger equation},
				journal={Nonlinear Anal.},
				volume={202},
				date={2021},
				pages={112118},
			}
			
			\bib{CC}{article}{
				author={Campos, L. },
				author={Cardoso, M.},
				title={Blow up and scattering criteria above the threshold for the focusing inhomogeneous nonlinear Schr\"odinger equation},
				journal={preprint},
				eprint= {http://arxiv.org/abs/2001.11613},
			}
			
			\bib{CFGM}{article}{
				author={Cardoso, M.},
				author={Farah, L. G.},
				author={Guzm\'an, C. M.},
				author={Murphy, J. },	
				title={Scattering below the ground state for the intercritical non-radial inhomogeneous NLS},
				journal={preprint},
				eprint={http://arxiv.org/abs/2007.06165},
			}
			
			\bib{Cazenave}{book}{
				author={Cazenave, T.},
				title={Semilinear Schr\"{o}dinger equations},
				series={Courant Lecture Notes in Mathematics},
				volume={10},
				publisher={New York University, Courant Institute of Mathematical
					Sciences, New York; American Mathematical Society, Providence, RI},
				date={2003},
				pages={xiv+323},
				isbn={0-8218-3399-5},
			}
			
			\bib{Chen}{article}{
				author={Chen, J.},
				title={On a class of nonlinear inhomogeneous Schr\"{o}dinger equation},
				journal={J. Appl. Math. Comput.},
				volume={32},
				date={2010},
				number={1},
				pages={237--253},
				issn={1598-5865},
			}
			
			\bib{CG}{article}{
				author={Chen, J.},
				author={Guo, B.},
				title={Sharp global existence and blowing up results for inhomogeneous
					Schr\"{o}dinger equations},
				journal={Discrete Contin. Dyn. Syst. Ser. B},
				volume={8},
				date={2007},
				number={2},
				pages={357--367},
				issn={1531-3492},
			}
			
			
			\bib{CHL}{article}{
				author={Cho, Y.},
				author={Hong, S.},
				author={Lee, K.},
				title={On the global well-posedness of focusing energy-critical inhomogeneous NLS},
				journal={J. Evol. Equ.},
				date={2020},
				eprint={http://doi.org/10.1007/s00028-020-00558-1},
			}
			
			\bib{Dinh-NA}{article}{
				author={Dinh, V. D.},
				title={Blowup of $H^1$ solutions for a class of the focusing
					inhomogeneous nonlinear Schr\"{o}dinger equation},
				journal={Nonlinear Anal.},
				volume={174},
				date={2018},
				pages={169--188},
				issn={0362-546X},
			}
			
			\bib{Dinh-JEE}{article}{
				author={Dinh, V. D. },
				title={Energy scattering for a class of the defocusing inhomogeneous
					nonlinear Schr\"{o}dinger equation},
				journal={J. Evol. Equ.},
				volume={19},
				date={2019},
				number={2},
				pages={411--434},
				issn={1424-3199},
			}
			
			\bib{Dinh-weight}{article}{
				author={V. D. Dinh },
				title={Scattering theory in weighted $L^2$ space for a class of the defocusing inhomogeneous nonlinear Schr\"odinger equation},
				journal={to appear in Advances in Pure and Applied Mathematics},
				eprint={https://arxiv.org/abs/1710.01392},
			}
			
			\bib{Dinh-2D}{article}{
				author={V. D. Dinh},
				title={Energy scattering for a class of inhomogeneous nonlinear
					Schr\"{o}dinger equation in two dimensions},
				journal={J. Hyperbolic Differ. Equ.},
				volume={18},
				date={2021},
				number={1},
				pages={1--28},
				issn={0219-8916},
			}
			
			\bib{Dinh-DCDS}{article}{
				author={Dinh, V. D. },
				title={A unified approach for energy scattering for focusing nonlinear Schr\"odinger equations},
				journal={Discrete
					Contin. Dyn. Syst.},
				date={2020},
				volume={40},
				number={11},
				pages={6441--6471},
				issn={1078-0947},
			}
			
			\bib{DM-rad}{article}{
				author={Dodson, B.},
				author={Murphy, J.},
				title={A new proof of scattering below the ground state for the 3D radial
					focusing cubic NLS},
				journal={Proc. Amer. Math. Soc.},
				volume={145},
				date={2017},
				number={11},
				pages={4859--4867},
				issn={0002-9939},
			}
			
			
			\bib{DHR}{article}{
				author={Duyckaerts, T.},
				author={Holmer, J.},
				author={Roudenko, S.},
				title={Scattering for the non-radial 3D cubic nonlinear Schr\"{o}dinger
					equation},
				journal={Math. Res. Lett.},
				volume={15},
				date={2008},
				number={6},
				pages={1233--1250},
				issn={1073-2780},
			}
			
			\bib{DR}{article}{
				author={Duyckaerts, T.},
				author={Roudenko, S.},
				title={Threshold solutions for the focusing 3D cubic Schr\"{o}dinger
					equation},
				journal={Rev. Mat. Iberoam.},
				volume={26},
				date={2010},
				number={1},
				pages={1--56},
				issn={0213-2230},
			}
			
			
			\bib{DR-CMP}{article}{
				author={Duyckaerts , T.},
				author={Roudenko, S.},
				title={Going beyond the threshold: scattering and blow-up in the focusing
					NLS equation},
				journal={Comm. Math. Phys.},
				volume={334},
				date={2015},
				number={3},
				pages={1573--1615},
				issn={0010-3616},
			}
			
			\bib{Farah}{article}{
				author={Farah, L. G.},
				title={Global well-posedness and blow-up on the energy space for the
					inhomogeneous nonlinear Schr\"{o}dinger equation},
				journal={J. Evol. Equ.},
				volume={16},
				date={2016},
				number={1},
				pages={193--208},
				issn={1424-3199},
			}
			
			\bib{FG-JDE}{article}{
				author={Farah, L. G.},
				author={Guzm\'{a}n, C. M.},
				title={Scattering for the radial 3D cubic focusing inhomogeneous
					nonlinear Schr\"{o}dinger equation},
				journal={J. Differential Equations},
				volume={262},
				date={2017},
				number={8},
				pages={4175--4231},
				issn={0022-0396},
			}
			
			\bib{FG-BBMS}{article}{
				author={Farah , L. G.},
				author={Guzm\'{a}n, C. M.},
				title={Scattering for the radial focusing inhomogeneous NLS equation in
					higher dimensions},
				journal={Bull. Braz. Math. Soc. (N.S.)},
				volume={51},
				date={2020},
				number={2},
				pages={449--512},
				issn={1678-7544},
			}
			
			\bib{FXC}{article}{
				author={Fang, D.},
				author={Xie, J.},
				author={Cazenave, T.},
				title={Scattering for the focusing energy-subcritical nonlinear
					Schr\"{o}dinger equation},
				journal={Sci. China Math.},
				volume={54},
				date={2011},
				number={10},
				pages={2037--2062},
				issn={1674-7283},
			}
			
			\bib{FW}{article}{
				author={Fibich, G.},
				author={Wang, X.},
				title={Stability of solitary waves for nonlinear Schr\"{o}dinger equations
					with inhomogeneous nonlinearities},
				journal={Phys. D},
				volume={175},
				date={2003},
				number={1-2},
				pages={96--108},
				issn={0167-2789},
			}
			
			\bib{Foschi}{article}{
				author={Foschi, D.},
				title={Inhomogeneous Strichartz estimates},
				journal={J. Hyperbolic Differ. Equ.},
				volume={2},
				date={2005},
				number={1},
				pages={1--24},
				issn={0219-8916},
			}
			
			\bib{FO}{article}{
				author={Fukuizumi, R.},
				author={Ohta, M.},
				title={Instability of standing waves for nonlinear Schr\"{o}dinger equations
					with inhomogeneous nonlinearities},
				journal={J. Math. Kyoto Univ.},
				volume={45},
				date={2005},
				number={1},
				pages={145--158},
				issn={0023-608X},
			}
			
			\bib{GW}{article}{
				author={Gao, Y.},
				author={Wang, Z.},
				title={Below and beyond the mass-energy threshold: scattering for the
					{H}artree equation with radial data in $d\geq 5$},
				journal={Z. Angew. Math. Phys.},
				volume={71},
				date={2020},
				number={2},
				pages={Paper No. 52, 23},
			}
			
			\bib{GS}{article}{
				author={Genoud, F.},
				author={Stuart, C. A.},
				title={Schr\"{o}dinger equations with a spatially decaying nonlinearity:
					existence and stability of standing waves},
				journal={Discrete Contin. Dyn. Syst.},
				volume={21},
				date={2008},
				number={1},
				pages={137--186},
				issn={1078-0947},
			}
			
			
			\bib{Genoud-2d}{article}{
				author={Genoud, Fran\c{c}ois},
				title={A uniqueness result for $\Delta u-\lambda u+V(|x|)u^p=0$ on $\Bbb R^2$},
				journal={Adv. Nonlinear Stud.},
				volume={11},
				date={2011},
				number={3},
				pages={483--491},
				issn={1536-1365},
			}
			
			
			\bib{Glassey}{article}{
				author={Glassey, R. T.},
				title={On the blowing up of solutions to the Cauchy problem for nonlinear
					Schr\"{o}dinger equations},
				journal={J. Math. Phys.},
				volume={18},
				date={1977},
				number={9},
				pages={1794--1797},
				issn={0022-2488},
			}
			
			\bib{Guevara}{article}{
				author={Guevara, C. D.},
				title={Global behavior of finite energy solutions to the $d$-dimensional
					focusing nonlinear Schr\"{o}dinger equation},
				journal={Appl. Math. Res. Express. AMRX},
				date={2014},
				number={2},
				pages={177--243},
				issn={1687-1200},
			}
			
			\bib{Guzman}{article}{
				author={Guzm\'{a}n, C. M.},
				title={On well posedness for the inhomogeneous nonlinear Schr\"{o}dinger
					equation},
				journal={Nonlinear Anal. Real World Appl.},
				volume={37},
				date={2017},
				pages={249--286},
				issn={1468-1218},
			}
			
			\bib{Gill}{article}{
				author={Gill, T. S.},
				title={Optical guiding of laser beam in nonuniform plasma},
				journal={Pramana},
				volume={55},
				date={2000},
				pages={835-842},
				issn={0973-7111},
			}
			
			\bib{JC}{article}{
				author={Jeanjean, L.},
				author={Le Coz, S.},
				title={An existence and stability result for standing waves of nonlinear
					Schr\"{o}dinger equations},
				journal={Adv. Differential Equations},
				volume={11},
				date={2006},
				number={7},
				pages={813--840},
			}
			
			
			\bib{KT}{article}{
				author={Keel, M.},
				author={Tao, T.},
				title={Endpoint Strichartz estimates},
				journal={Amer. J. Math.},
				volume={120},
				date={1998},
				number={5},
				pages={955--980},
				issn={0002-9327},
			}
			
			\bib{KM}{article}{
				author={Kenig, C. E.},
				author={Merle, F.},
				title={Global well-posedness, scattering and blow-up for the
					energy-critical, focusing, non-linear Schr\"{o}dinger equation in the radial
					case},
				journal={Invent. Math.},
				volume={166},
				date={2006},
				number={3},
				pages={645--675},
				issn={0020-9910},
			}
			
			\bib{LT}{article}{
				author = {Liu, C. S.},
				author = {Tripathi, V. K.},
				title = {Laser guiding in an axially nonuniform plasma channel},
				journal = {Physics of Plasmas},
				volume = {1},
				number = {9},
				pages = {3100-3103},
				year = {1994},
			}
			
			\bib{LWW}{article}{
				author={Liu, Y.},
				author={Wang, X.},
				author={Wang, K.},
				title={Instability of standing waves of the Schr\"{o}dinger equation with
					inhomogeneous nonlinearity},
				journal={Trans. Amer. Math. Soc.},
				volume={358},
				date={2006},
				number={5},
				pages={2105--2122},
				issn={0002-9947},
			}
			
			
			\bib{Martel}{article}{
				author={Martel, Y.},
				title={Blow-up for the nonlinear Schr\"{o}dinger equation in nonisotropic
					spaces},
				journal={Nonlinear Anal.},
				volume={28},
				date={1997},
				number={12},
				pages={1903--1908},
				issn={0362-546X},
			}
			
			
			\bib{Merle}{article}{
				author={Merle, F.},
				title={Nonexistence of minimal blow-up solutions of equations
					$iu_t=-\Delta u-k(x)|u|^{4/N}u$ in ${\bf R}^N$},
				language={English, with English and French summaries},
				journal={Ann. Inst. H. Poincar\'{e} Phys. Th\'{e}or.},
				volume={64},
				date={1996},
				number={1},
				pages={33--85},
				issn={0246-0211},
			}
			
			\bib{MMZ}{article}{
				author={Miao, C. },
				author={Murphy, J.},
				author={Zheng, J.},
				title={Scattering for the non-radial inhomogeneous NLS},
				journal={to appear in Mathematical Research Letters},
				eprint= {http://arxiv.org/abs/1912.01318},
			}
			
			\bib{OK}{book}{
				author={Opic, B.},
				author={Kufner, A.},
				title={Hardy-type inequalities},
				series={Pitman Research Notes in Mathematics Series},
				volume={219},
				publisher={Longman Scientific \& Technical, Harlow},
				date={1990},
				pages={xii+333},
			}
			
			\bib{RS}{article}{
				author={Rapha\"{e}l, P.},
				author={Szeftel, J.},
				title={Existence and uniqueness of minimal blow-up solutions to an
					inhomogeneous mass critical NLS},
				journal={J. Amer. Math. Soc.},
				volume={24},
				date={2011},
				number={2},
				pages={471--546},
				issn={0894-0347},
			}
			
			\bib{Strauss}{article}{
				author={Strauss, W. A.},
				title={Existence of solitary waves in higher dimensions},
				journal={Comm. Math. Phys.},
				volume={55},
				date={1977},
				number={2},
				pages={149--162},
				issn={0010-3616},
			}	
			
			\bib{SS}{book}{
				author={Sulem, C.},
				author={Sulem, P. L.},
				title={The nonlinear Schr\"{o}dinger equation},
				series={Applied Mathematical Sciences},
				volume={139},
				note={Self-focusing and wave collapse},
				publisher={Springer-Verlag, New York},
				date={1999},
				pages={xvi+350},
				isbn={0-387-98611-1},
			}
			
			\bib{Tao-DPDE}{article}{
				author={Tao, T.},
				title={On the asymptotic behavior of large radial data for a focusing
					non-linear Schr\"{o}dinger equation},
				journal={Dyn. Partial Differ. Equ.},
				volume={1},
				date={2004},
				number={1},
				pages={1--48},
				issn={1548-159X},
			}
			
			\bib{Tao}{book}{
				author={Tao , T.},
				title={Nonlinear dispersive equations},
				series={CBMS Regional Conference Series in Mathematics},
				volume={106},
				note={Local and global analysis},
				publisher={Published for the Conference Board of the Mathematical
					Sciences, Washington, DC; by the American Mathematical Society,
					Providence, RI},
				date={2006},
				pages={xvi+373},
				isbn={0-8218-4143-2},
			}
			
			
			
			\bib{Toland}{article}{
				author={Toland, J. F.},
				title={Uniqueness of positive solutions of some semilinear
					Sturm-Liouville problems on the half line},
				journal={Proc. Roy. Soc. Edinburgh Sect. A},
				volume={97},
				date={1984},
				pages={259--263},
				issn={0308-2105},
				review={\MR{751198}},
			}
			
			\bib{TM}{article}{
				author={Towers, I.},
				author={Malomed, B. A.},
				title={Stable $(2+1)$-dimensional solitons in a layered medium with
					sign-alternating Kerr nonlinearity},
				journal={J. Opt. Soc. Amer. B Opt. Phys.},
				volume={19},
				date={2002},
				number={3},
				pages={537--543},
				issn={0740-3224},
			}
			
			
			\bib{XZ}{article}{
				author={Xu, C.},
				author={Zhao, T.}
				title={A remark on the scattering theory for the 2D radial focusing INLS},
				journal={preprint},
				eprint={https://arxiv.org/abs/1908.00743},
			}
			
			\bib{Yanagida}{article}{
				author={Yanagida, E.},
				title={Uniqueness of positive radial solutions of $\Delta
					u+g(r)u+h(r)u^p=0$ in ${\bf R}^n$},
				journal={Arch. Rational Mech. Anal.},
				volume={115},
				date={1991},
				number={3},
				pages={257--274},
				issn={0003-9527},
			}
			
			\bib{Zhu}{article}{
				author={Zhu, S.},
				title={Blow-up solutions for the inhomogeneous Schr\"{o}dinger equation with
					$L^2$ supercritical nonlinearity},
				journal={J. Math. Anal. Appl.},
				volume={409},
				date={2014},
				number={2},
				pages={760--776},
				issn={0022-247X},
			}
			
		\end{biblist}
	\end{bibdiv}
\end{document}